\documentclass[11pt,onecolumn]{article}
\usepackage{times}
\usepackage{mathrsfs}
\usepackage[utf8]{inputenc} 
\usepackage[T1]{fontenc}
\usepackage{url}
\usepackage{ifthen}
\usepackage{cite}
\usepackage[cmex10]{amsmath}
\usepackage{graphicx,subfigure}
\usepackage{amssymb}
\usepackage{dsfont}
\usepackage{enumitem}
\usepackage{color}
\usepackage{epstopdf}
\usepackage[export]{adjustbox}
\usepackage[font=footnotesize]{caption}
\usepackage[utf8]{inputenc} 
\usepackage{hyperref}
\usepackage{comment}
\usepackage{placeins}

\ifpdf
    \graphicspath{{figures/PNG/}{figures/PDF/}{figures/}}
\else
    \graphicspath{{figures/EPS/}{figures/}}
\fi

\newtheorem{theorem}{Theorem}
\newtheorem{lemma}{Lemma}
\newtheorem{corollary}{Corollary}

\newtheorem{definition}{Definition}

\def \F {{\cal F}}

\newcommand{\dff}{\stackrel{\scriptscriptstyle\triangle}{=}}

\def \P{\mathbb{P}}

\title{\Large  
{\bf \textcolor{black}{Probability-Maximizing Change Detection: Finite-Window Optimality}}}
\date{}

\author{Ali Tajer\thanks{Electrical, Computer, and Systems Engineering Department, Rensselaer Polytechnic Institute, Troy, NY.} \and  Javad~Heydari\thanks{Cruise LLC, Seattle, WA.}}

\begin{document}

\maketitle

\begingroup
\renewcommand{\thefootnote}{}
\footnotetext{The results in this paper were partly presented in the 2019 IEEE International Symposium on Information Theory.}
\addtocounter{footnote}{-1}
\endgroup

\begin{abstract}
{\color{black}
This paper investigates probability-maximizing sequential change detection, a formulation in which performance is measured by the probability of stopping within an admissible interval of duration $\xi$ after a change rather than by the expected detection delay. Earlier work introduced this viewpoint in a Bayesian setting and subsequently formalized it under Lorden- and Pollak-type minimax criteria for the case in which successful detection must occur on the \textbf{first} post-change observation, i.e., $\xi=1$. This paper generalizes this framework in two directions. First, the decision maker is allowed to stop within a prescribed window of $\xi\in\mathbb{N}$ post-change observations. Second, the monitored process is allowed to experience multiple, non-overlapping transient change episodes with unknown onset times and durations, so that success consists of stopping within the admissible window associated with any one of these episodes. False alarms are controlled through an average run-length constraint. 

For exact finite-sample analysis, the paper introduces a survival-weighted average success criterion, which represents the probability of successfully detecting a randomly encountered change opportunity conditional on the detector being active at its onset. It is established that this criterion admits an exact representation as the expected \textbf{truncated} Shiryaev--Roberts (TSR) statistic at the stopping time normalized by the average run length, and it characterizes its exactly optimal stopping rule. The optimal procedure has finite memory and aggregates the likelihood-ratio evidence corresponding to all possible change onsets within the most recent window and compares the resulting TSR statistic with a \textbf{state-dependent} continuation boundary obtained from an optimal-stopping formulation. In the special case that the objective is to detect on the first post-change observation ($\xi=1$), the memory and state dependence disappear, and the TSR test reduces to the classical Shewhart rule, which is also exactly optimal under the Lorden- and Pollak-type minimax criteria. The paper also establishes structural properties of the associated value and continuation functions that support computation and interpretation of the optimal boundary. Furthermore, it is shown that a \textbf{constant}-boundary TSR procedure is asymptotically minimax optimal under both Pollak and Lorden criteria when the admissible window grows beyond the information-accumulation scale
$\frac{\log \epsilon}{{\sf D}_{\rm KL}(F_1\|F_0)}$, where $F_0$ and $F_1$ denote the nominal and post-change distributions, respectively, ${\sf D}_{\rm KL}(F_1\|F_0)$ is their Kullback--Leibler divergence, and $\epsilon$ denotes the required average run length to false alarm.
}
\end{abstract}


\section{Introduction}
{\color{black}
Classical quickest change detection is principally concerned with a single change that occurs at an unknown time and persists thereafter. Its standard formulations fall into two broad categories: Bayesian and minimax~\cite{Shiryaev1963,Lorden,Pollak}. In the Bayesian framework, the change-point is modeled as a random variable with a specified prior distribution, and the stopping rule is designed by averaging performance over both the change-point prior and the observation process~\cite{Shiryaev1963}. In the minimax framework, the change-point is treated as an unknown but nonrandom time, and performance is assessed under its least favorable occurrence. Lorden's criterion considers the worst conditional expected delay over all change times and all possible pre-change observation histories~\cite{Lorden}, whereas Pollak's criterion considers the worst conditional expected delay over all change times, conditioned on no alarm having occurred before the change~\cite{Pollak}. Although these formulations differ in how uncertainty about the change-point is represented and how worst-case performance is quantified, they share the objective of favoring an \emph{earlier post-change alarm over a later one,} subject to false-alarm control. They do not, however, directly encode the requirement that detection occur within a prescribed admissible window.

When detections outside that window are operationally equivalent to misses, it is natural instead to maximize the probability that the stopping time falls within the admissible post-change interval.
The \emph{probability-maximizing} framework for change-point detection was first introduced by Bojdecki~\cite{Bojdecki}. In a Bayesian setting, Bojdecki formulated an optimal-stopping problem in which success is defined by the stopping time lying in a prescribed neighborhood of a random change time. In the zero-tolerance case, success requires stopping at the first observation generated under the post-change distribution. This formulation treated premature or late stopping through the success event itself, rather than through a separate average-run-length constraint. Subsequent Bayesian and semi-Bayesian developments extended the probability-maximizing perspective to dependent observations and to composite post-change models~\cite{Sarnowski,Pollak-Krieger}.

Moustakides subsequently placed this probability-maximizing viewpoint in the standard false-alarm-constrained framework of sequential change detection~\cite{Moustakides:2014}. Specifically, probability-maximizing counterparts of the Shiryaev, Lorden, and Pollak criteria were formulated by crediting a decision when it occurs within a prescribed interval after the change. For the non-Bayesian Lorden- and Pollak-type formulations, the exact solution developed in~\cite{Moustakides:2014} concerns the one-sample case: conditioned on no earlier alarm, the procedure is rewarded only when it stops on the first post-change observation. The optimal rule is then the classical Shewhart test with a constant threshold selected to meet the average-run-length constraint. The same framework was further studied for independent non-identically distributed observations, multiple possible post-change distributions, and Markovian data~\cite{Moustakides:2014,Moustakides:2015}. This paper studies the probability-maximizing framework through two complementary performance perspectives: a survival-weighted average success criterion and Pollak- and Lorden-type worst-case criteria, and advances the underlying model in two directions.

\begin{enumerate}
\item \textbf{A multi-sample admissible detection window.} We credit a detection whenever it occurs within the first $\xi$ observations after a change onset, rather than only on the first post-change observation. Thus, $\xi=1$ is exactly the one-sample problem solved by Moustakides, whereas $\xi\geq 2$ is the nontrivial multi-sample extension addressed in this paper. All instants in this window are regarded as successful; stopping outside it is counted as a miss for that change onset.

\item \textbf{Multiple transient change opportunities.} We allow an unknown finite number of non-overlapping change episodes, with unknown onset times and durations. The detector operates once over the entire observation stream and seeks to stop in the admissible window associated with one of these episodes. The Pollak- and Lorden-type criteria aggregate the conditional success probabilities associated with the possible onsets and take the worst case over all admissible change configurations. Under the persistent-change convention specified in Section~\ref{sec:model}, the single persistent-change model is recovered by taking one episode of infinite duration. More generally, the behavior after the first $\xi$ post-change observations is immaterial to the success event as long as every episode lasts at least $\xi$ samples. The survival-weighted average criterion introduced in this paper evaluates one randomly encountered change opportunity, whereas the multiple-episode model is retained for the Pollak- and Lorden-type worst-case analysis.
\end{enumerate}
The extension from one admissible post-change observation to a window containing two or more observations fundamentally changes the structure of the exact average-optimal procedure. In the one-sample setting, the optimal rule is the classical Shewhart test, which bases each decision only on the current observation. With a longer admissible window, the exact average-optimal rule acquires finite memory and combines the evidence associated with all possible change onsets within the most recent observations. It can therefore be viewed as a finite-memory generalization of the Shewhart test. Furthermore, its underlying detection statistic is a truncated version of the classical Shiryaev--Roberts statistic~\cite{Shiryaev1963,Roberts1966}. Unlike the conventional Shiryaev--Roberts procedure, however, the proposed rule does not generally compare this statistic with a fixed threshold. Instead, its stopping boundary depends on the recent likelihood-ratio history, reflecting the value of retaining the accumulated evidence for future decisions. We refer to this state-dependent generalization as the \emph{truncated Shiryaev--Roberts} (TSR) test. When the admissible window contains only the first post-change observation, the memory and state dependence disappear, and the TSR test reduces to the classical Shewhart rule established in the one-sample probability-maximizing setting.

Our analysis establishes exact finite-sample optimality for a survival-weighted average success criterion and develops complementary guarantees under the Pollak- and Lorden-type worst-case criteria. The average criterion measures the probability of successful detection for a randomly encountered change opportunity, conditional on the detector still being active at its onset. It admits an exact representation as the ratio between the expected TSR evidence available at the stopping time and the expected time to stopping. We solve the resulting ratio optimization through an optimal-stopping formulation and obtain the state-dependent TSR rule as its exact optimizer. For a single change opportunity, the Pollak- and Lorden-type success probabilities are upper-bounded by this survival-weighted average criterion; for $\xi\geq2$, we do not claim that these worst-case bounds are attained at finite ARL.

For $\xi=2$, the stopping region admits a scalar state-dependent threshold representation; for general $\xi$, the sufficient state consists of the most recent $\xi-1$ likelihood ratios. We also characterize monotonicity and convexity properties of the continuation value and describe an offline--online implementation of the average-optimal rule. For $\xi=1$, the TSR rule reduces to Shewhart and remains exactly optimal under both the average and the Pollak- and Lorden-type minimax criteria.
Additionally, we identify a complementary regime in which the simpler constant-boundary TSR procedure is asymptotically minimax optimal. Specifically, when the admissible window grows beyond a specific rate, the state-dependent continuation boundary is not needed for first-order optimality.

The present manuscript revises and substantially extends our preliminary work in~\cite{C74}, which introduced the multiple-opportunity model, developed the state-dependent stopping structure in detail for $\xi=2$, and stated its general-window extension under the Pollak- and Lorden-type criteria. It also assumed a known common episode duration and stated finite-sample minimax attainment for $\xi\geq2$. The present manuscript makes a necessary distinction that was not developed in the conference version: the state-dependent TSR rule is exactly finite-sample optimal for the newly explicit survival-weighted average criterion $\mathcal A_\xi$, whose exact TSR identity is proved here, whereas for the Pollak- and Lorden-type criteria the finite-sample result at $\xi\geq2$ is an upper bound rather than an attainment claim. Furthermore, this work allows unknown, episode-specific durations subject only to $\delta_i\geq\xi$, establishes structural and computational properties of the continuation boundary, provides complete proofs, and proves growing-window asymptotic minimax optimality for the simpler constant-boundary TSR rule.

The present formulation is related to, but distinct from, several other strands of change-detection research.
Work on persistent changes with transient dynamics studies processes that pass through intermediate post-change regimes before settling permanently into a new state~\cite{GM-VV,GF-VV,VKM,VKM2}; in our model, each episode instead returns to the nominal distribution.
A separate literature studies finite-duration or temporary changes under a variety of formulations, including Bayesian quickest detection of transient changes, detection in multiple on--off processes, sampling-constrained transient detection, and likelihood-ratio, cumulative-sum, finite-moving-average, and window-limited procedures~\cite{premkumar2010bayesian,zhao2010quickest,Nikiforov,guepie2017dynamicprofile,C74,Ebi,egearoca2018fma,egearoca2022twostrategies,mana2022centralized,mana2023arbitrary,nikiforov2025unknownprofile}.
Related models include temporary changes in multivariate time series, intermittent changes of unknown duration, unknown appearance and disappearance times, and detection or estimation of multiple transient periods~\cite{watson2022temporary,sokolov2023intermittent,tartakovsky2021appearance,baron2023multiple}.
Our focus differs in the combination of an explicit probability-maximizing success window, an exactly optimized survival-weighted average criterion, Lorden- and Pollak-type minimax protection, an average-run-length false-alarm constraint, and asymptotic minimax guarantees for multiple transient opportunities.

The remainder of the paper is organized as follows. Section~\ref{sec:model} introduces the observation model and the average and worst-case probability-maximizing formulations. Section~\ref{sec:Xi2} develops the two-sample case $\xi=2$ and establishes exact average optimality, and Section~\ref{sec:Xi} extends the analysis to arbitrary $\xi\geq 2$. Section~\ref{sec:Xi1} treats the one-sample case and recovers the classical Shewhart rule under both the average and minimax criteria. Section~\ref{sec:comp} discusses structural properties and computational implementation. Section~\ref{sec:constant_tsr_asymptotics} develops worst-case minimax guarantees and establishes asymptotic minimax optimality of a constant-boundary TSR procedure in a joint ARL--window regime. Numerical results are presented in Section~\ref{sec:sim}, and concluding remarks are provided in Section~\ref{sec:conclusion}.
}

\section{\textcolor{black}{Observation} Model and \textcolor{black}{Probability-Maximizing Formulation}}
\label{sec:model}

\subsection{\textcolor{black}{Observation and Change-Episode Model}}

\textcolor{black}{Throughout, let $\mathbb{N}\dff\{1,2,\dots\}$ and $\mathbb{N}_0\dff\mathbb{N}\cup\{0\}$.} Consider the discrete-time observation sequence $\mathcal{X}\dff\{X_t:t\in\mathbb{N}\}$. Define \textcolor{black}{its natural filtration $\{\F_t:t\in\mathbb{N}_0\}$ by}
\begin{align}
\F_0\dff\{\emptyset,\Omega\}\ , \qquad
\F_t\dff\sigma\big(X_1,\dots,X_t\big),\quad t\in\mathbb{N}\ .
\end{align}
\textcolor{black}{For a fixed finite integer $s\geq1$, let}
\begin{align}\label{eq:change-points}
\mathcal{S}\dff\{\gamma_i:i\in\{1,\dots,s\}\}\subset\mathbb{N}\ ,\qquad 1 \leq \gamma_1<\cdots<\gamma_s\ ,
\end{align}
\textcolor{black}{denote} the set of \textcolor{black}{change-episode onset times. At each $\gamma_i$, the observation} distribution changes from the nominal cumulative distribution function (cdf) $F_0$ to a distinct \textcolor{black}{cdf} $F_1$. \textcolor{black}{Let} $\delta_i\in\mathbb{N}$ \textcolor{black}{denote} the duration of the \textcolor{black}{episode beginning} at $\gamma_i$. \textcolor{black}{The onset times and durations are unknown to the detector. For a prescribed admissible-window length $\xi\geq1$, we require only}
\begin{align}
\textcolor{black}{\delta_i\geq\xi\ ,\qquad i\in\{1,\dots,s\}\ ,}
\end{align}
{\color{black}so that the entire $\xi$-sample detection window following each onset is generated under $F_1$. The durations themselves need not be known; beyond the requirement $\delta_i\geq\xi$, no common episode-duration parameter enters the formulation.}
\textcolor{black}{The change episodes} do not overlap and are separated by at least one nominal \textcolor{black}{observation}; equivalently,
$\gamma_{i+1}\geq \gamma_i+\delta_i+1$ for all $i\in\{1,\dots,s-1\}$\textcolor{black}{. Otherwise}, two adjacent \textcolor{black}{episodes} can be merged and viewed as a single \textcolor{black}{episode}. Define the set of \textcolor{black}{affected time indices} as
\begin{align}
\mathcal{U}\dff \bigcup_{i=1}^s \{\gamma_i,\gamma_i+1,\dots,\gamma_i+\delta_i-1\}\subset\mathbb{N}\ .
\end{align}
\textcolor{black}{We also admit the single persistent-change convention $s=1$ and $\delta_1=\infty$, in which case $\mathcal{U}\dff\{t\in\mathbb{N}:t\geq\gamma_1\}$. Since the success event depends only on the first $\xi$ observations following the onset, all subsequent results apply under this convention without modification.}
\textcolor{black}{The observations follow} the dichotomous model
\begin{equation}\label{eq:H}
    \begin{array}{cl}
      X_t\sim F_0\;, & t\in\mathbb{N}\setminus \mathcal{U} \\
      X_t\sim F_1\;, & t\in\mathcal{U}
    \end{array}\ .
\end{equation}
\begin{figure}[t]
    \centering
    \includegraphics[width=0.65\linewidth]{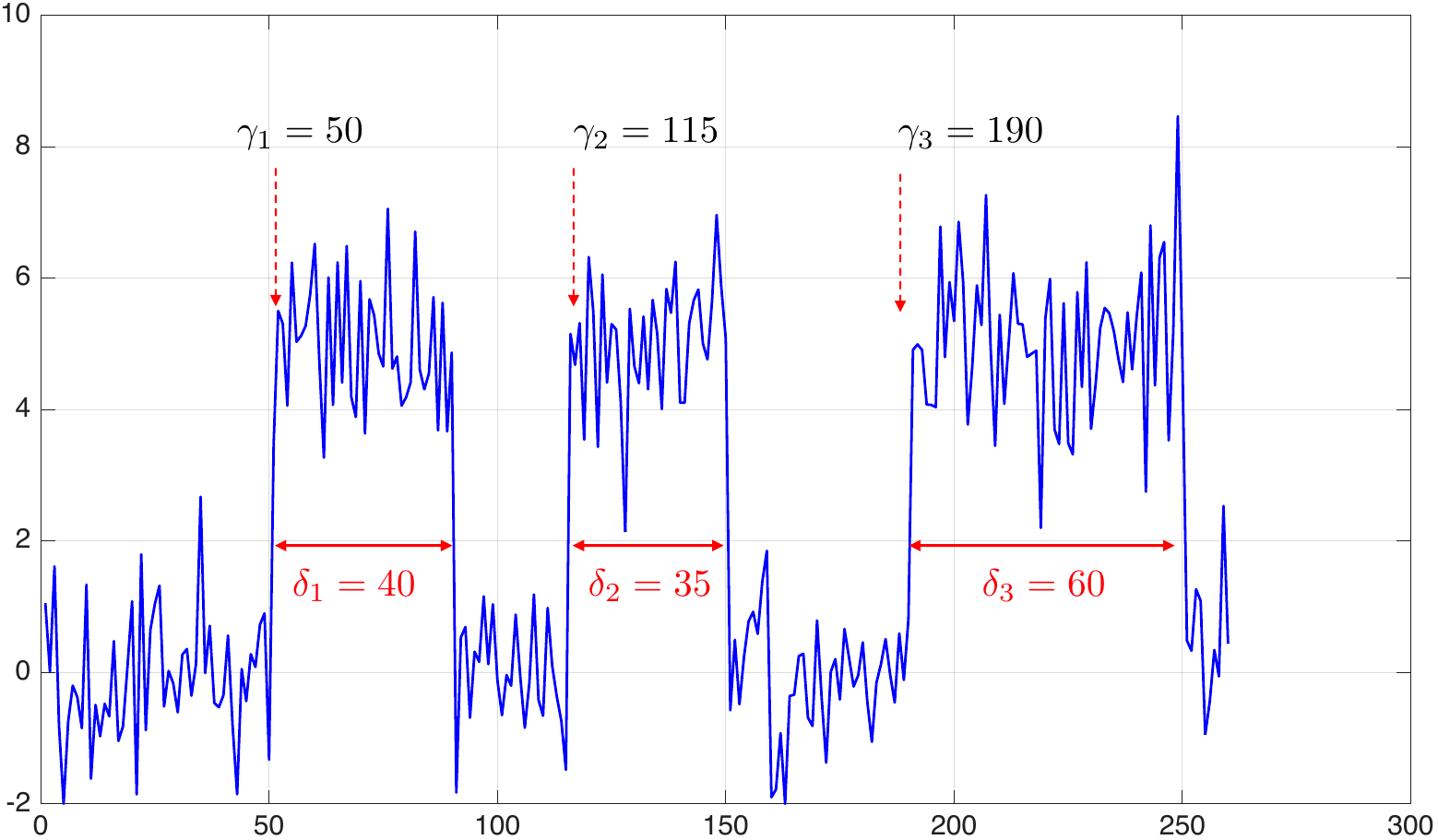}
    \caption{Illustration of \textcolor{black}{multiple non-overlapping change episodes}.}
    \label{fig:model}
\end{figure}
Figure~\ref{fig:model} illustrates a configuration with $s=3$ episodes beginning at $\gamma_1=50$, $\gamma_2=115$, and $\gamma_3=190$, with durations $\delta_1=40$, $\delta_2=35$, and $\delta_3=60$, respectively. Conditioned on the change configuration $(\mathcal{S},\{\delta_i:i\in\{1,\dots,s\}\})$, the observations are assumed independent and satisfy~\eqref{eq:H}. Let $\mu$ be a common dominating measure for $F_0$ and $F_1$, with densities $f_0$ and $f_1$, and assume $F_1\ll F_0$ so that the likelihood ratio $f_1/f_0$ is well defined. For each fixed $s$, let $\Theta_s(\xi)$ denote the class of configurations $\theta=(\mathcal{S},\{\delta_i\}_{i=1}^s)$ satisfying $\delta_i\geq\xi$ and $\gamma_{i+1}\geq\gamma_i+\delta_i+1$. When $s=1$, we augment $\Theta_1(\xi)$ with the persistent configuration $\delta_1=\infty$ under the convention specified above. We denote the probability measure and expectation induced by $\theta$ by $\mathbb{P}_\theta$ and $\mathbb{E}_\theta$. The no-change measure and expectation are denoted by $\mathbb{P}_\infty$ and $\mathbb{E}_\infty$. \textcolor{black}{Although $s$ indexes the performance class $\Theta_s(\xi)$, the onset times and durations are unknown, and the stopping rule derived below does not use them.}

\textcolor{black}{The assumptions above delimit the scope of the results. Knowledge of $F_0$ and $F_1$ permits exact evaluation of the one-step likelihood ratio $f_1(X_t)/f_0(X_t)$. In applications in which these distributions are estimated, a plug-in implementation does not generally retain the exact finite-sample guarantees unless model uncertainty is incorporated explicitly. Conditional independence permits the likelihood ratio for each candidate window to be factored into one-step likelihood ratios and yields a finite-dimensional, time-homogeneous Markov state under the no-change measure. For dependent observations, the probability-maximizing formulation remains applicable, but both the likelihood ratios and the sufficient state must account for the dependence structure, and the present analysis does not apply without modification. Finally, the condition $\delta_i\geq\xi$ ensures that every observation in an admissible window follows the post-change distribution; the duration $\delta_i$ remains unknown and is not used by the stopping rule.}

\subsection{\textcolor{black}{Admissible Detection Windows}}

\textcolor{black}{A sequential decision rule is an $\{\F_t\}_{t\in\mathbb{N}_0}$-stopping time $\tau$ taking values in $\mathbb{N}_0$. When a result explicitly permits randomization, an auxiliary random variable $B$, independent of the observations, is made available at time zero, and the stopping time is understood with respect to the enlarged filtration $\{\F_t\vee\sigma(B)\}_{t\in\mathbb{N}_0}$. For an episode beginning at $\gamma_i$, a decision is credited as successful when}
\begin{align}
\textcolor{black}{\gamma_i\leq\tau<\gamma_i+\xi\ .}
\end{align}
\textcolor{black}{Thus, the admissible window consists of the first $\xi$ observations generated under $F_1$ after that onset. All stopping instants within this window are treated equally. The stopping time itself is not truncated or forced to stop by $\gamma_i+\xi$; a decision outside the window is permitted but is not counted as a successful detection of that onset. This distinction separates the finite-window success criterion from a finite-horizon stopping problem. Hence, the formulation replaces the conventional ranking of post-change delays with a probability-maximizing objective: conditional on the procedure having survived until a change onset, it seeks to maximize the probability that the eventual stopping time lies in the associated admissible window. The behavior of an episode after its first $\xi$ post-change observations is immaterial to this success event. Consequently, the same formulation covers both finite-duration episodes with $\delta_i\geq\xi$ and persistent changes. In the next two subsections, we formalize minimax and average decision criteria.}

{\color{black} Note that the global stopping time remains an infinite-horizon random variable and may be arbitrarily large. The admissible window does not impose a deterministic deadline on $\tau$; it defines the $0$--$1$ event that is credited as a successful response to a particular onset. This objective is distinct from expected-delay optimization in which a delay criterion ranks earlier post-change stops above later ones and continues to assign value to sufficiently late stops, whereas the present criterion treats all stops within the admissible window equally and assigns no credit outside it for that onset. In the multiple-episode model, missing one admissible window does not terminate monitoring, and the procedure may still stop successfully in the window of a later episode.}

\subsection{\textcolor{black}{Pollak- and Lorden-Type Minimax Criteria}}

\textcolor{black}{By adopting the approach in~\cite{Moustakides:2014}, for each fixed $s$, we assess performance uniformly over the unknown configurations in $\Theta_s(\xi)$. The Pollak-type criterion is}
\begin{align}\label{eq:Pollak_criterion}
\mathcal{L}_{\rm P}(\tau)
\dff
\inf_{\theta\textcolor{black}{\in\Theta_s(\xi)}}
\sum_{i=1}^s
\mathbb{P}_\theta\!\left(
\gamma_i\leq\tau<\gamma_i+\xi
\,\middle|\,
\tau\geq\gamma_i
\right)\ .
\end{align}
\textcolor{black}{Each summand is the probability of stopping in the admissible window of episode $i$, conditioned on no alarm having occurred before its onset. Accordingly, the Pollak-type criterion averages over the pre-change histories compatible with survival to $\gamma_i$ and then takes the least favorable admissible configuration.}

\textcolor{black}{The Lorden-type criterion strengthens this protection by also taking the essential infimum over the information available immediately before each onset:}
\begin{align}\label{eq:Lorden_criterion}
\mathcal{L}_{\rm L}(\tau)
\dff
\inf_{\theta\textcolor{black}{\in\Theta_s(\xi)}}
\sum_{i=1}^s
\operatorname*{ess\,inf}_{\F_{\gamma_i-1}}
\mathbb{P}_\theta\!\left(
\gamma_i\leq\tau<\gamma_i+\xi
\,\middle|\,
\F_{\gamma_i-1},\tau\geq\gamma_i
\right)\ .
\end{align}
\textcolor{black}{Thus, both criteria treat the onset locations as unknown and nonrandom, while the Lorden-type criterion additionally protects against the least favorable pre-change observation history. Directly from the definitions,}
\begin{align}
\mathcal{L}_{\rm L}(\tau)\leq\mathcal{L}_{\rm P}(\tau)\ .
\end{align}
\textcolor{black}{The quantities in~\eqref{eq:Pollak_criterion} and~\eqref{eq:Lorden_criterion} are aggregate episode-wise success scores: each term is a conditional probability, and the sum evaluates the detector across the $s$ available opportunities. For fixed $s$, dividing either criterion by $s$ gives the corresponding average conditional success probability and does not change the optimizing stopping rule. We retain the unnormalized sum because it leads to cleaner expressions in the optimality analysis that follows.}

False alarms are controlled through the average run length (ARL) under the no-change measure, $\mathbb{E}_\infty[\tau]$. Let $\mathcal{T}$ denote the resulting class of stopping times taking values in $\mathbb{N}_0$, including the independently randomized extension described above when randomization is invoked. The Pollak-type optimization problem is
\begin{align}\label{eq:p1}
\mathcal{P}_{\rm P}(\varepsilon)
\dff
\sup_{\tau\in\mathcal{T}}\ \mathcal{L}_{\rm P}(\tau)
\quad\text{s.t.}\quad
\mathbb{E}_\infty[\tau]\geq\varepsilon\ ,
\end{align}
and the Lorden-type problem is
\begin{align}\label{eq:p2}
\mathcal{P}_{\rm L}(\varepsilon)
\dff
\sup_{\tau\in\mathcal{T}}\ \mathcal{L}_{\rm L}(\tau)
\quad\text{s.t.}\quad
\mathbb{E}_\infty[\tau]\geq\varepsilon\ ,
\end{align}
where $\varepsilon\geq1$ specifies the required ARL to false alarm.

\subsection{\textcolor{black}{Survival-weighted Average Success Criterion}}

{\color{black}
Recall that $\mathcal T$ denotes the class of stopping times, possibly randomized at time zero, taking values in $\mathbb N_0$. For each $\gamma\in\mathbb N$, define
$\mathbb P_\gamma^\xi$ as an auxiliary local-change measure on
$\mathcal F_{\gamma+\xi-1}$ that coincides with $\mathbb P_\infty$
on $\mathcal F_{\gamma-1}$ and under which
$X_\gamma,\ldots,X_{\gamma+\xi-1}$ are independent with density $f_1$.
Equivalently, for each $k\in\{1,\ldots,\xi\}$,
\begin{align}
\left.
\frac{d\mathbb P_\gamma^\xi}{d\mathbb P_\infty}
\right|_{\mathcal F_{\gamma+k-1}}
=
\prod_{j=0}^{k-1}\ell_{\gamma+j}\ .
\end{align}
No distributional specification beyond time $\gamma+\xi-1$ is needed, since the success event associated with onset $\gamma$ is measurable with respect to $\mathcal F_{\gamma+\xi-1}$. For a stopping time $\tau$ that has survived until $\gamma$, define the $\gamma$-specific conditional success probability
\begin{align}
 p_{\gamma}^{\xi}(\tau)
 \triangleq
 \mathbb P_{\gamma}^{\xi}\!\left(
 \gamma\leq\tau<\gamma+\xi
 \,\middle|\,
 \tau\geq\gamma
 \right)\ ,
 \label{eq:onset_success_probability}
\end{align}
with the convention $p_{\gamma}^{\xi}(\tau)=0$ whenever $\mathbb P_\infty(\tau\geq\gamma)=0$. For $0<\mathbb E_\infty[\tau]<\infty$, define the survival weights
\begin{align}
 w_\gamma(\tau)
 \triangleq
 \frac{\mathbb P_\infty(\tau\geq\gamma)}{\mathbb E_\infty[\tau]}\ ,
 \qquad \gamma\geq 1\ .
 \label{eq:survival_weights}
\end{align}
The weight $w_\gamma(\tau)$ has a natural operational interpretation. Its numerator is the probability, under the no-change law, that the detector remains active through time $\gamma$, and hence that a change opportunity occurring at $\gamma$ can still be acted upon. The normalization by $\mathbb E_\infty[\tau]$ makes these survival probabilities sum to one. Equivalently, consider repeated independent
monitoring cycles, with the detector restarted after each alarm, and select an active monitoring instant at random from the resulting
long-run operation. Then $w_\gamma(\tau)$ is the probability that the selected instant occupies position $\gamma$ within its monitoring cycle. Thus, change opportunities at times that the detector reaches more frequently receive proportionally greater weight. The tail-sum identity implies $w_\gamma(\tau)\geq0$ and $\sum_{\gamma\geq1}w_\gamma(\tau)=1$. Based on these, we define the \emph{survival-weighted average success probability}
\begin{align}
 \mathcal A_\xi(\tau)
 \triangleq
 \sum_{\gamma=1}^{\infty}
 w_\gamma(\tau)\,p_{\gamma}^{\xi}(\tau)\ .
 \label{eq:average_success_criterion}
\end{align}
Because it is a convex combination of conditional success probabilities, $0\leq\mathcal A_\xi(\tau)\leq1$. 
Note that under repeated operation with restart after each alarm, $w_\gamma(\tau)$ is the long-run fraction of active monitoring instants that occupy position $\gamma$ within a renewal cycle. Thus, $\mathcal A_\xi(\tau)$ is the average probability of successful detection per encountered change opportunity. 
The criterion $\mathcal A_\xi$ evaluates one randomly encountered opportunity. The general multiple-episode model is retained for the worst-case criteria below; without an explicit reset between episodes, we do not identify its finite-sample aggregate score with a simple multiple of $\mathcal A_\xi$. The exact finite-sample problem studied in this paper is
\begin{align}
 \mathcal P_{\rm A}(\varepsilon,\xi)
 \triangleq
 \sup_{\tau\in\mathcal T}
 \mathcal A_\xi(\tau)
 \quad\text{s.t.}\quad
 \mathbb E_\infty[\tau] \geq \varepsilon\ .
 \label{eq:average_optimization_problem}
\end{align}
}
{\color{black}
\subsection{Truncated Shiryaev--Roberts Statistic}

In this subsection, we formally define the
truncated Shiryaev--Roberts statistic.
We show that stopping rules induced by the TSR statistic enjoy two complementary optimality guarantees. Under the survival-weighted average success criterion, comparison with a state-dependent boundary yields exact finite-sample optimality. Under the Pollak- and Lorden-type probability-maximizing criteria, comparison with a constant boundary yields first-order asymptotic minimax optimality in the joint growing-ARL and growing-window regime specified in Section~\ref{sec:constant_tsr_asymptotics}.

\begin{definition}[{\color{black} TSR} Statistic]
Given the sequence of likelihood-ratio statistics
$\{\ell_t:t\in\mathbb{N}\}$, for a given window length
$\xi\in\mathbb{N}$, define the \emph{{\color{black} TSR} statistic} at time $t$ as
\begin{align}
W_t^{(\xi)}
\triangleq
\sum_{k=1}^{\xi}\prod_{j=0}^{k-1}\ell_{t-j}
=
\ell_t+\ell_t\ell_{t-1}
+\ell_t\ell_{t-1}\ell_{t-2}
+\cdots+
\prod_{j=0}^{\xi-1}\ell_{t-j}\ {\color{black}.}
\end{align}
This statistic is the window-limited, or truncated, version of the classical Shiryaev--Roberts statistic.
The corresponding \emph{TSR} test is the stopping rule
\begin{align}
\mathscr{T}(b_t)
\triangleq
\inf\{t\geq1:W_t^{(\xi)}\geq b_t\}\  ,
\end{align}
where $b_t$ is an $\F_t$-measurable boundary. In the optimal rule derived below, this boundary is state dependent rather than a fixed or deterministic time-varying threshold.
\end{definition}
To show the connection to the Shiryaev--Roberts statistics, re-indexing the statistic gives
\begin{align}
W_t^{(\xi)}=\sum_{j=t-\xi+1}^{t}\prod_{k=j}^{t}\ell_k\ ,
\end{align}
with the usual initialization convention near the beginning of the sequence. Hence, the statistic is exactly a window-limited, or truncated, version of the classical Shiryaev--Roberts statistic~\cite{Shiryaev1963,Roberts1966}. The distinction lies in the complete stopping rule: the classical Shiryaev--Roberts procedure uses a constant boundary, whereas exact optimality for the survival-weighted average probability-maximizing problem yields a state-dependent continuation boundary determined by the most recent $\xi-1$ likelihood ratios. This statistic is also distinct from a windowed CuSum statistic, which takes the maximum over the likelihood ratios associated with the candidate onset locations; the TSR statistic sums those likelihood ratios because a stop at time $t$ is successful under any one of the compatible onset locations.

When $s=1$ and the episode persists indefinitely, the model reduces to a single persistent change with an admissible post-change window. When $\xi=1$, success requires stopping on the first post-change observation, and the average-optimal rule reduces to Shewhart. In this one-sample case, the established Pollak- and Lorden-type probability-maximizing formulations also admit Shewhart as an exact minimax solution. For $\xi\geq 2$, recent samples can jointly contribute to a successful stop, leading to the finite-memory TSR structure characterized. A formulation that assigns a decreasing reward to later stopping times inside the admissible window is also possible, but it defines a different objective from the equal-credit probability-maximizing criterion studied here.
}

\section{\textcolor{black}{Probability-maximizing Detection: Admissible Window $\xi=2$}}
\label{sec:Xi2}

In this section, we characterize the \textcolor{black}{stopping rule that is exactly optimal for the survival-weighted average success criterion $\mathcal{A}_2$}. To simplify the exposition and highlight the main ideas, we begin with the special case $\xi=2$. In this setting, \textcolor{black}{a decision is credited as successful for onset $\gamma_i$ when the stopping time $\tau$ satisfies
$\gamma_i\leq\tau<\gamma_i+2$, i.e., when it stops on either the first or the second post-change observation. The stopping rule itself is not truncated and may continue if it does not stop within that window.
If this two-sample window is missed, monitoring continues; in the multiple-episode setting, the same rule may still stop successfully in the admissible window of a later episode.
The multiple-episode model remains relevant to the Pollak- and Lorden-type worst-case criteria studied later; the exact finite-sample criterion optimized in this section concerns one randomly encountered change opportunity, as formalized by $\mathcal{A}_2$.}
We first establish the core results for this case and then extend the analysis to arbitrary values of $\xi$ in Section~\ref{sec:Xi}. Finally, we treat the case $\xi=1$ separately, as it leads to a qualitatively different structure of the optimal rule.

{\color{black}{Throughout this section, the finite-sample optimization concerns the survival-weighted average criterion in~\eqref{eq:average_success_criterion}}. For the choice of $\xi=2$, we characterize the optimal decision rule, i.e., the solution to}
{\eqref{eq:average_optimization_problem}}, in three steps:
\begin{enumerate}
\item {we first derive} an upper bound for the \textcolor{black}{single-opportunity worst-case} objective functions {$\mathcal{L}_{\rm P}(\tau)$ and $\mathcal{L}_{\rm L}(\tau)$} \textcolor{black}{and identify the ratio appearing in that bound with the exact average criterion $\mathcal{A}_2(\tau)$};
\item subsequently, we characterize a decision rule that maximizes the \textcolor{black}{exact average criterion}; and
\item we show that the {stopping} rule of step $2$ \textcolor{black}{is exactly optimal for $\mathcal{A}_2$ under the ARL constraint.}
\end{enumerate}
Throughout the analysis, we denote the likelihood ratio of the sample
collected at time $t\in\mathbb{N}$ by $\ell_t$, i.e.,
\begin{align}
\ell_t
\triangleq
\frac{f_1(X_t)}{f_0(X_t)}\ ,
\label{eq:likelihood_ratio}
\end{align}
and adopt the convention $\ell_0=0$. {\color{black}{Before deriving the worst-case bounds, we establish a fundamental
identity that connects the survival-weighted average success criterion
to the TSR statistic. Although we state the identity for an arbitrary
window length $\xi$, in this section we use its specialization to
$\xi=2$.
\begin{lemma}[Average-success identity]\label{lemma:average_identity}
For every stopping time $\tau$ satisfying
$0<\mathbb E_\infty[\tau]<\infty$,
\begin{align}
 \mathcal A_\xi(\tau)
 =
 \frac{\mathbb E_\infty[W_\tau^{(\xi)}]}
 {\mathbb E_\infty[\tau]}\ .
 \label{eq:average_identity}
\end{align}
\end{lemma}
\begin{proof}
See Appendix~\ref{app:average_identity}.
\end{proof}
For $\xi=2$, Lemma~\ref{lemma:average_identity} specializes to
\begin{align}
\mathcal A_2(\tau)
=
\frac{\mathbb E_\infty[
\ell_\tau+\ell_{\tau-1}\ell_\tau]}
{\mathbb E_\infty[\tau]}\ .
\label{eq:average_identity_xi2}
\end{align}}}
We now derive an upper bound on the single-opportunity performance
criteria $\mathcal L_{\rm P}(\tau)$ and
$\mathcal L_{\rm L}(\tau)$ for $\xi=2$.

\begin{theorem}[Upper Bound for $\xi=2$]\label{thm:upper_bound2}
For any stopping time $\tau {\in\mathcal{T}}$ satisfying {$0<\mathbb{E}_\infty[\tau] < \infty$}, and for any change set {$\mathcal{S}$} with $|{\mathcal{S}}| = s$, the performance criteria $\mathcal{L}_{\rm P}(\tau)$ and $\mathcal{L}_{\rm L}(\tau)$, defined in~\eqref{eq:Pollak_criterion} and~\eqref{eq:Lorden_criterion}, respectively, for $\xi=2$ satisfy
\begin{align}\label{eq:up:d}
\mathcal{L}_{\rm L}(\tau)
\leq \mathcal{L}_{\rm P}(\tau)
\leq s\cdot 
\frac{\mathbb{E}_\infty[\ell_\tau+\ell_{\tau-1}\ell_\tau]}
{\mathbb{E}_\infty[\tau]}
\textcolor{black}{\; =s\cdot \mathcal{A}_2(\tau)}\, .
\end{align}
\end{theorem}

\begin{proof}
See Appendix~\ref{app:thm:upper_bound2} \textcolor{black}{for the worst-case inequalities. The final identity follows from Lemma~\ref{lemma:average_identity}.}
\end{proof}
{\color{black}The bound in Theorem~\ref{thm:upper_bound2} also has an information-theoretic converse interpretation. It applies to every admissible stopping rule and is independent of any particular detector construction; after imposing the ARL constraint, it therefore specifies a universal upper performance limit for the probability of successful finite-window detection. Furthermore, by Lemma~\ref{lemma:average_identity}, the upper bound is exactly the survival-weighted average success probability $\mathcal{A}_2(\tau)$. The remainder of the analysis constructs the TSR stopping rule and proves that it exactly maximizes this average criterion. We do not assert that the single-opportunity Pollak- or Lorden-type upper bound in~\eqref{eq:up:d} is attained at finite $\varepsilon$ for $\xi=2$. First, note that according to Lemma~\ref{lemma:average_identity} we have
\begin{align}\label{eq:A2_ratio}
\mathcal{A}_2(\tau)
=
\frac{\mathbb{E}_\infty[\ell_\tau+\ell_{\tau-1}\ell_\tau]}
{\mathbb{E}_\infty[\tau]}\ .
\end{align}}
In the next step, we seek the stopping time $\tau$ that maximizes the exact average criterion in~\eqref{eq:A2_ratio} subject to the false alarm constraint specified in~\eqref{eq:average_optimization_problem}. This leads to the optimization problem
\begin{align}\label{eq:q}
{\mathcal{Q}(\varepsilon)\dff}
\sup_{{\tau\in\mathcal{T}}} \; & \displaystyle 
\frac{{\mathbb{E}_\infty[\ell_\tau+\ell_{\tau-1}\ell_\tau]}} 
{{\mathbb{E}_\infty[\tau]}} \quad 
\text{s.t.} \quad  \mathbb{E}_\infty[\tau] \geq \varepsilon\ .
\end{align}
By~\eqref{eq:A2_ratio}, $\mathcal{Q}(\varepsilon)$ is precisely the optimal value of the survival-weighted average problem $\mathcal{P}_{\rm A}(\varepsilon,2)$ in~\eqref{eq:average_optimization_problem}.
To solve ${\mathcal{Q}(\varepsilon)}$, we next show that the supremum over the inequality constraint can equivalently be achieved over stopping times satisfying the equality constraint. Specifically, for any feasible\footnote{A stopping time $\nu$ is called feasible if ${\mathbb{E}_\infty[\nu]\geq\varepsilon}$ and ${\mathbb{E}_\infty[\nu]<\infty}$.} stopping time $\nu$ with associated ratio
\begin{align}
\frac{{\mathbb{E}_\infty[\ell_\nu+\ell_{\nu-1}\ell_\nu]}}
{{\mathbb{E}_\infty[\nu]}}\ ,
\end{align}
there exists a (possibly randomized) stopping time $\nu'$ satisfying
\begin{align}
{\mathbb{E}_\infty[\nu']=\varepsilon}\ ,
\end{align}
and
\begin{align}
\frac{{\mathbb{E}_\infty[\ell_{\nu'}+\ell_{\nu'-1}\ell_{\nu'}]}}
{{\mathbb{E}_\infty[\nu']}}
= \frac{{\mathbb{E}_\infty[\ell_\nu+\ell_{\nu-1}\ell_\nu]}}
{{\mathbb{E}_\infty[\nu]}} \ .
\end{align}
This reduction is formalized in the following lemma, which incorporates the ARL to false alarm as an equality constraint. 
\begin{lemma}[Reducing the ARL constraint to equality]\label{lemma:eq}
Fix $\varepsilon \geq 1$. For any feasible stopping time $\nu$ satisfying ${\mathbb{E}_\infty[\nu] \geq \varepsilon}$, there exists a (possibly randomized) stopping time $\nu'$ such that
\begin{enumerate}
    \item ${\mathbb{E}_\infty[\nu'] = \varepsilon}$, i.e., the false alarm constraint is satisfied with equality; and
    \item the performance ratio is preserved, namely
    \begin{align}\label{eq:ratio}
        \frac{{\mathbb{E}_\infty[\ell_{\nu'} + \ell_{\nu'-1}\ell_{\nu'}]}}
        {{\mathbb{E}_\infty[\nu']}}
        =
        \frac{{\mathbb{E}_\infty[\ell_\nu + \ell_{\nu-1}\ell_\nu]}}
        {{\mathbb{E}_\infty[\nu]}} \ .
    \end{align}
\end{enumerate}
\end{lemma}
\begin{proof}
See Appendix~\ref{app:lemma:eq}.
\end{proof}
Therefore, in solving ${\mathcal{Q}(\varepsilon)}$, the inequality constraint can be replaced by an equality constraint without loss of optimality. Consequently, it suffices to consider the equivalent problem
\begin{align}\label{eq:qbar}
{\widetilde{\mathcal{Q}}(\varepsilon) \dff}
\sup_{{\tau\in\mathcal{T}}} & \displaystyle 
\frac{{\mathbb{E}_\infty[\ell_\tau+\ell_{\tau-1}\ell_\tau]}}
{{\mathbb{E}_\infty[\tau]}} \quad 
\text{s.t.} \quad {\mathbb{E}_\infty[\tau] = \varepsilon}\ .
\end{align}
The equality-constrained problem in~\eqref{eq:qbar} can, in turn, be reformulated as the following unconstrained Lagrangian problem:
\begin{align}\label{eq:r}
{\mathcal{V}(\lambda) \dff}
\sup_{{\tau\in\mathcal{T}}}
{\mathbb{E}_\infty \big[\ell_\tau+\ell_{\tau-1}\ell_\tau - \lambda \tau\big]}\ ,
\end{align}
where $\lambda > 0$ is a Lagrange multiplier chosen so that the resulting optimizer satisfies {$\mathbb{E}_\infty[\tau] = \varepsilon$}. This equivalence is established formally in the following lemma.

\begin{lemma}[Lagrangian equivalence of the ARL-equality problem]\label{lemma:lag}
{Fix $\varepsilon\ge 1$. For $\mathcal{V}(\lambda)$ defined in~\eqref{eq:r}, there exists $\lambda_\varepsilon>0$ and a (possibly randomized) stopping time $\tau^\star\in\mathcal{T}$ that attains $\mathcal{V}(\lambda_\varepsilon)$ and satisfies}
\begin{align}
{\mathbb{E}_\infty[\tau^\star]=\varepsilon\ ,}
\end{align}
{and consequently $\tau^\star$ solves $\widetilde{\mathcal{Q}}(\varepsilon)$\ .}
\end{lemma}

\begin{proof}
See Appendix~\ref{proof:lemma:lag}.
\end{proof}
We proceed by analyzing the Lagrangian problem {$\mathcal{V}(\lambda)$}.
For $t \in \mathbb{N}$, define
\begin{align}
    W_t \dff \ell_t + \ell_{t-1}\ell_t\ ,
\end{align}
with the convention that {$W_0 = 0$}. Further, define
\begin{align}
\label{eq:admissible}
{\mathcal{T}_t \dff 
\left\{
\tau \in \mathcal{T}:\ \tau \ge t \;\;\; \text{a.s.}
\right\}\ ,}
\end{align}
{where $\mathcal{T}$ denotes the class of stopping times, possibly randomized at time zero, that take values in $\mathbb{N}_0$ and satisfy $\mathbb{P}_\infty(\tau<\infty)=1$.}
Given $\F_t$, define the infinite-horizon value function
\begin{align}
J_t(\F_t)
\dff 
\operatorname*{ess\,sup}_{\tau \in {\mathcal{T}_t}}
\mathbb{E}_\infty
 \left[
{W_\tau} - \lambda(\tau - t)
\,\middle|\, 
\F_t
\right]\ ,
\label{eq:inf_JR_def}
\end{align}
and the continuation value function;
\begin{align}
R_t(\F_t)
\dff
\mathbb{E}_\infty
 \left[
J_{t+1}(\F_{t+1})
\,\middle|\,
\F_t
\right]\ .
\label{eq:bell-Rt}
\end{align}
The essential supremum in~\eqref{eq:inf_JR_def} is taken with respect to the almost-sure order over $\F_t$--measurable random variables and therefore is well defined. Furthermore, by definition of the essential supremum, for any $\tau \in {\mathcal{T}_t}$,
\begin{align}
J_t(\F_t)
\ge
\mathbb{E}_\infty
 \left[
{W_\tau} - \lambda(\tau - t)
\,\middle|\,
\F_t
\right]
\quad \text{a.s.}
\label{eq:inf_J_dominates_any_tau}
\end{align}
Using this property, we next establish the Bellman identity that relates $J_t(\F_t)$ to the immediate reward {$W_t$} and the continuation value $R_t(\F_t)$. For notational simplicity, we henceforth write $J_t$ and $R_t$ in place of $J_t(\F_t)$ and $R_t(\F_t)$, respectively.

\begin{lemma}[Bellman identity for the infinite horizon]\label{lemma:bellman_identity}
Fix $\lambda>0$. For every $t\ge1$ we almost surely have
\begin{align}
J_t(\F_t)\;=\; \max\big\{\,{W_t},\ -\lambda + R_t\,\big\}\ .
\label{eq:bell-identity}
\end{align}
\end{lemma}
\begin{proof}
    See Appendix \ref{app:lemma:bellman}.
\end{proof}
The following lemma determines a stopping time that optimizes {$\mathcal{V}(\lambda)$} defined in~\eqref{eq:r}, which we subsequently show \textcolor{black}{solves the exact average-success problem in~\eqref{eq:q}}.
\begin{lemma}\label{lemma:5}
The following stopping rule is an optimal solution to the optimization problem in~\eqref{eq:r}. 
\begin{align}\label{eq:stop:hh}
{\tau_\lambda^\star\dff\inf\{t\ge 1\;:\; \ell_t+\ell_{t-1}\ell_t\geq -\lambda +R_t\} = {\mathscr T}(-\lambda+R_t)\ .}
\end{align}
\end{lemma}
\begin{proof}
See Appendix~\ref{app:lemma:5}.
\end{proof}
Finally, we show that $\ell_{t-1}$ and $\ell_t$ are sufficient statistics for describing the terms $J_t$ and $R_t$. This is formalized in the next lemma.
\begin{lemma}[Sufficient statistics]\label{lemma:suff_infinite}
There exist Borel functions $J:\,[0,\infty)^2 \to\mathbb R$ and $R:\,[0,\infty) \to\mathbb R$ such that, almost surely,
\begin{align}
J_t (\F_t)\,=\, J(\ell_{t-1},\ell_t)\ ,
\qquad \mbox{and} \qquad 
R_t (\F_t)\,=\, R(\ell_t)\ . \label{eq:ss_claim}
\end{align}
In particular, $(\ell_{t-1},\ell_t)$ is a sufficient statistic for $J_t$ and $\ell_t$ is a sufficient statistic for $R_t$.
\end{lemma}

\begin{proof}
See Appendix \ref{app:lemma:suff_infinite}.
\end{proof}
Based on these properties, we have the following expressions for $R$ and $J$, which are instrumental in specifying the optimal stopping time.
\begin{align}
R(\ell_t) &= {\mathbb{E}_\infty} \big[J(\ell_t,\ell_{t+1})\ \big|\ \ell_t\big]\ ,
\label{eq:inf_R_stationary}
\\
J(\ell_{t-1},\ell_t) &= \max\Big\{\ \ell_t+\ell_{t-1}\ell_t,\ -\lambda+R(\ell_t)\ \Big\}\ .
\label{eq:inf_J_stationary}
\end{align}
With these notations, the stopping time ${\tau_\lambda^\star}$ defined in~\eqref{eq:stop:hh} can be restated as\footnote{In the rest of the paper, when it is clear from the context, we use the shorthand $\tau^\star$ for $\tau_\lambda^\star$}
\begin{align}
{\tau_\lambda^\star\ =\ \inf\Big\{ t\ge 1\; : \ \ell_t+\ell_{t-1}\ell_t\ \ge\ -\lambda+R(\ell_t)\Big\}\ .}
\label{eq:inf_tau_star_stationary}
\end{align}
\textcolor{black}{Lemmas~\ref{lemma:average_identity}, ~\ref{lemma:eq},~\ref{lemma:lag}, and~\ref{lemma:5}, collectively, establish that the stopping rule specified in~\eqref{eq:inf_tau_star_stationary} maximizes the exact survival-weighted average criterion $\mathcal{A}_2$ subject to the ARL constraint.}
After selecting $\lambda=\lambda_\varepsilon$ so that
$\mathbb{E}_\infty[\tau_{\lambda_\varepsilon}^\star]=\varepsilon$,
define $\tau^\star\dff\tau_{\lambda_\varepsilon}^\star$. \textcolor{black}{The following theorem formalizes the resulting exact average optimality.}

\begin{theorem}[{\color{black} TSR Average Optimality} for $\xi=2$]\label{th:average_opt_xi2}
For the stopping time $\tau^\star$ defined in~\eqref{eq:inf_tau_star_stationary} with $\lambda=\lambda_\varepsilon$ selected so that $\mathbb{E}_\infty[\tau^\star]=\varepsilon$, we have
\begin{align}
\textcolor{black}{
\mathcal{A}_2(\tau^\star)
=
\frac{\mathbb{E}_\infty[\ell_{\tau^\star}+\ell_{\tau^\star-1}\ell_{\tau^\star}]}
{\mathbb{E}_\infty[\tau^\star]}
=
\sup_{\substack{\tau: \mathbb{E}_\infty[\tau]\geq \varepsilon}}
\mathcal{A}_2(\tau)\ .}
\label{eq:exact_average_opt_xi2}
\end{align}
\end{theorem}
\begin{proof}
\textcolor{black}{Lemma~\ref{lemma:average_identity} identifies $\mathcal{A}_2(\tau)$ with the ratio in~\eqref{eq:q}. Lemma~\ref{lemma:eq} reduces the ARL inequality to equality without changing that ratio. Lemma~\ref{lemma:lag} identifies a multiplier $\lambda_\varepsilon$ whose Lagrangian optimizer has ARL $\varepsilon$, and Lemma~\ref{lemma:5} identifies that optimizer as the stopping rule in~\eqref{eq:inf_tau_star_stationary}. These statements together prove~\eqref{eq:exact_average_opt_xi2}.}
\end{proof}
\textcolor{black}{Theorem~\ref{thm:upper_bound2} additionally implies}
\begin{align}
\textcolor{black}{
\mathcal{L}_{\rm L}(\tau^\star)
\leq
\mathcal{L}_{\rm P}(\tau^\star)
\leq
s\cdot \mathcal{A}_2(\tau^\star)\ .}
\end{align}
\textcolor{black}{This relation records the worst-case performance bound associated with the average-optimal rule, but no finite-sample equality or Pollak-/Lorden-optimality claim is made for $\xi=2$.}

Hence, the \textcolor{black}{exact average-optimal} test involves comparing
$(\ell_t+\ell_{t-1}\ell_t)$ with the
\textcolor{black}{state-dependent continuation boundary}
$-\lambda+R(\ell_t)$.
Here $\lambda=\lambda_\varepsilon$ is selected so that the ARL constraint
$\mathbb{E}_\infty[\tau^\star]=\varepsilon$ holds.
The value of this \textcolor{black}{boundary} depends on the last likelihood ratio $\ell_t$.
We discuss the computational aspect of finding $R(\ell)$ in Section~\ref{sec:comp}.

\section{\textcolor{black}{Probability-maximizing Detection: General Admissible Window $\xi\geq 2$}}
\label{sec:Xi}

In this section, we generalize the results of Section~\ref{sec:Xi2} to an arbitrary
\textcolor{black}{admissible detection-window length $\xi\geq 2$}. The arguments follow the same line of reasoning as for $\xi=2$. Therefore, we state
the key definitions and results and omit the proofs for brevity.
\textcolor{black}{As in the two-sample case, the window determines whether a stop is
credited as successful; it does not truncate the stopping rule. The exact finite-sample result generalized in this section is optimality with respect to the survival-weighted average success criterion $\mathcal{A}_\xi$; the Pollak- and Lorden-type criteria remain complementary worst-case measures and are related to $\mathcal{A}_\xi$ through the single-opportunity upper bound.}
For this purpose, at any time $t$ we define
\begin{align}
\ell_\xi^t
\dff
(\ell_{t-\xi+1},\dots,\ell_t)
\in[0,\infty)^\xi\ ,
\end{align}
as the collection of the last $\xi$ likelihood ratio terms up to time $t$.
For notational convenience (and consistent with Section~\ref{sec:Xi2}), we set
$\ell_0=0$ and extend this convention by letting $\ell_t=0$ for all $t\leq 0$.
We also define the $(\xi-1)$-lag vector
\begin{align}
\ell_{\xi-1}^t
\dff
(\ell_{t-\xi+2},\dots,\ell_t)
\in[0,\infty)^{\xi-1}\ .
\end{align}
Recall the \textcolor{black}{TSR statistic at} time $t$:
\begin{align}
W_t^{(\xi)}
=
\sum_{k=1}^{\xi}
\prod_{j=0}^{k-1}\ell_{t-j}\ .
\end{align}
\textcolor{black}{By Lemma~\ref{lemma:average_identity}, for every stopping time satisfying $0<\mathbb{E}_\infty[\tau]<\infty$,
\begin{align}
\mathcal{A}_\xi(\tau)
=
\frac{\mathbb{E}_\infty[W_\tau^{(\xi)}]}
{\mathbb{E}_\infty[\tau]}\ .
\label{eq:average_identity_general_section}
\end{align}
Furthermore, the reduction of the ARL inequality to equality and the Lagrangian equivalence arguments in Lemmas~\ref{lemma:eq} and~\ref{lemma:lag} extend directly after replacing $W_t^{(2)}$ by $W_t^{(\xi)}$.}
In the first step, we define the following unconstrained Lagrangian problem
\textcolor{black}{that generalizes~\eqref{eq:r} from the two-sample case to an
arbitrary window length $\xi\geq2$}:
\begin{align}\label{eq:r_xi}
\mathcal{V}_\xi(\lambda)
\dff
\sup_{\tau\in\mathcal{T}}
\mathbb{E}_\infty
\big[
{\color{black}W_\tau^{(\xi)}}-\lambda\tau
\big]\ .
\end{align}
Here, $\mathcal{T}$ denotes the class of stopping times, possibly randomized at time zero, with values in $\mathbb{N}_0$. Similarly to Section~\ref{sec:Xi2}, define for
$t\in\mathbb{N}$ the truncated class
\begin{align}
\mathcal{T}_t
\dff
\{\tau\in\mathcal{T}:\tau\geq t\ \text{a.s.}\},
\end{align}
and the value and continuation functions
\begin{align}
\label{eq:reward0_general}
J_t^{\xi}(\F_t)
\dff
\operatorname*{ess\,sup}_{\tau\in\mathcal{T}_t}
\mathbb{E}_\infty
\left[
{\color{black}W_\tau^{(\xi)}}-\lambda(\tau-t)
\,\middle|\ ,
\F_t
\right]\ ,
\end{align}
and
\begin{align}\label{eq:cost2_general}
R_t^{\xi}(\F_t)
\dff
\mathbb{E}_\infty
\left[
J_{t+1}^{\xi}(\F_{t+1})
\,\middle|\ ,
\F_t
\right]\ .
\end{align}
In particular, the Bellman identity takes the form
\begin{align}
J_t^{\xi}(\F_t)
=
\max
\big\{
W_t^{(\xi)},
-\lambda+R_t^{\xi}(\F_t)
\big\}\ .
\end{align}
Similar to the case of $\xi=2$, we first show that the problem at hand is a Markov optimal stopping problem.
\begin{lemma}[Sufficient Statistic]\label{lemma:ss:general}
\textcolor{black}{Fix $\xi\geq2$} and $\lambda>0$.
There exist Borel functions
\begin{align}
J_\xi:[0,\infty)^\xi\to\mathbb{R}
\qquad\text{and}\qquad
R_\xi:[0,\infty)^{\xi-1}\to\mathbb{R}\ ,
\end{align}
such that, almost surely, for all $t\in\mathbb{N}$
\begin{align}
J_t^{\xi}(\F_t)
=
J_\xi(\ell_\xi^t)\ ,
\qquad
R_t^{\xi}(\F_t)
=
R_\xi(\ell_{\xi-1}^t)\ .
\end{align}
In particular, $\ell_\xi^t$ is a sufficient statistic for
$J_t^{\xi}$ and $\ell_{\xi-1}^t$ is a sufficient statistic for
$R_t^{\xi}$.
\end{lemma}

\begin{proof}
The proof arguments parallel those of the proof of
Lemma~\ref{lemma:suff_infinite}.
\end{proof}
Based on Lemma~\ref{lemma:ss:general}, we therefore write
\begin{align}
J_t^{\xi}=J_\xi(\ell_\xi^t)
\qquad\text{and}\qquad
R_t^{\xi}=R_\xi(\ell_{\xi-1}^t)\ .
\end{align}
With these definitions, the extension of the $\xi=2$ results to general $\xi$
is summarized next.

\begin{theorem}[TSR Test for general $\xi$]\label{thm:general}
For any integer $\xi\geq 2$, the following statements are true.

\begin{enumerate}[leftmargin=*]

\item \emph{\textcolor{black}{Worst-case upper bound and exact average identity:}}
For any stopping time
$\tau\in\mathcal{T}$ satisfying
$0<\mathbb{E}_\infty[\tau]<\infty$,
\textcolor{black}{the Pollak- and Lorden-type criteria satisfy}
\begin{align}\label{eq:upper:general}
\mathcal{L}_{\rm L}(\tau)
\leq
\mathcal{L}_{\rm P}(\tau)
\textcolor{black}{\leq
s\cdot \mathcal{A}_\xi(\tau)}
&
=s\cdot 
\frac{
\mathbb{E}_\infty
\big[
{\color{black}W_\tau^{(\xi)}}
\big]
}{
\mathbb{E}_\infty[\tau]
}
=s\cdot 
\frac{
\mathbb{E}_\infty
\left[
\sum_{k=1}^{\xi}
\prod_{j=0}^{k-1}\ell_{\tau-j}
\right]
}{
\mathbb{E}_\infty[\tau]
}\ .
\end{align}

\item \emph{\textcolor{black}{Stopping rule maximizing the exact average criterion:}}
Let $\lambda>0$ and let $R_\xi$ be as in
Lemma~\ref{lemma:ss:general}. The stopping time
\begin{align}\label{eq:st:general}
\tau_{\lambda,\xi}^\star
\dff
\inf
\Big\{
t\geq1:
W_t^{(\xi)}
\geq
-\lambda+R_\xi(\ell_{\xi-1}^t)
\Big\}
=
\mathscr{T}
\big(
-\lambda+R_\xi(\ell_{\xi-1}^t)
\big)
\end{align}
\textcolor{black}{maximizes the Lagrangian in~\eqref{eq:r_xi}. After choosing $\lambda=\lambda_\varepsilon$ so that}
\begin{align}
\mathbb{E}_\infty
[
\tau_{\lambda_\varepsilon,\xi}^\star
]
=
\varepsilon\ ,
\end{align}
\textcolor{black}{the resulting stopping rule solves the exact average-success problem $\mathcal{P}_{\rm A}(\varepsilon,\xi)$ in~\eqref{eq:average_optimization_problem}.}

\item \emph{\textcolor{black}{Exact average optimality:}}
For $\tau_\xi^\star
\dff
\tau_{\lambda_\varepsilon,\xi}^\star$ we have
\begin{align}
\textcolor{black}{
\mathcal{A}_\xi(\tau_\xi^\star)}
&
=
\frac{
\mathbb{E}_\infty
\big[
{\color{black}W_{\tau_\xi^\star}^{(\xi)}}
\big]
}{
\mathbb{E}_\infty[\tau_\xi^\star]
}
=
\frac{
\mathbb{E}_\infty
\left[
\sum_{k=1}^{\xi}
\prod_{j=0}^{k-1}
\ell_{\tau_\xi^\star-j}
\right]
}{
\mathbb{E}_\infty[\tau_\xi^\star]
}
=
\sup_{\substack{\tau\in\mathcal{T}:\\
\varepsilon\leq\mathbb{E}_\infty[\tau]<\infty}}
\mathcal{A}_\xi(\tau)\ .
\label{eq:average_opt_general}
\end{align}

\end{enumerate}
\end{theorem}
Hence, following the same line of arguments as in
Section~\ref{sec:Xi2}, the statements of
Theorem~\ref{thm:general} collectively establish that the stopping
rule specified in~\eqref{eq:st:general} is an \textcolor{black}{exactly optimal solution to
\eqref{eq:average_optimization_problem} for every $\xi\geq2$. For the single-opportunity worst-case criteria, the same theorem yields}
\begin{align}
\mathcal{L}_{\rm L}(\tau_\xi^\star)
\leq
\mathcal{L}_{\rm P}(\tau_\xi^\star)
\leq
s\cdot \mathcal{A}_\xi(\tau_\xi^\star)\ .
\end{align}
In particular, the \textcolor{black}{exact average-optimal} test compares the statistic
$W_t^{(\xi)}$, which is a function of the most recent $\xi$
likelihood ratios $\ell_\xi^t$, with the
\textcolor{black}{state-dependent continuation boundary}
\begin{align}
-\lambda+R_\xi(\ell_{\xi-1}^t)\ .
\end{align}
The \textcolor{black}{boundary} depends on the most recent
$\xi-1$ likelihood ratios through $\ell_{\xi-1}^t$.
We discuss the computational aspect of finding $R_\xi(\cdot)$ in
Section~\ref{sec:comp}.

\section{\textcolor{black}{Probability-maximizing Detection: Admissible Window $\xi=1$}}
\label{sec:Xi1}

In this section, we treat the case of $\xi=1$,
\textcolor{black}{i.e., the one-sample admissible-window setting.}
Since
\begin{align}
\{\gamma_i\leq\tau<\gamma_i+1\}
=
\{\tau=\gamma_i\}\ ,
\end{align}
a \textcolor{black}{decision is credited as successful for onset $\gamma_i$ only when}
$\tau=\gamma_i$.
This setting is treated independently since the analysis has important differences,
\textcolor{black}{and it provides the one-sample benchmark from which the general TSR rule is extended.}
The optimal test simplifies to the well-known Shewhart test,
\textcolor{black}{recovering the exact solution established for the one-sample probability-maximizing minimax problem in~\cite{Moustakides:2014}.}

{\color{black}We note that although the resulting Shewhart rule uses only the current observation, this is a conclusion of a sequential optimization problem rather than a reduction to an isolated hypothesis test. The onset time is unknown, observations arrive successively, and at every time the detector chooses between stopping and preserving the opportunity to act later. These decisions are coupled through the ARL constraint, and the optimization ranges over all admissible stopping times, including rules that use the complete observation history. Shewhart optimality therefore establishes that past samples have no value for the one-sample probability-maximizing objective; when $\xi>1$, the optimal rule instead acquires the finite memory characterized in Sections~\ref{sec:Xi2} and~\ref{sec:Xi}.

In this one-sample case, the survival-weighted average criterion and the established Pollak- and Lorden-type minimax criteria are optimized by the same rule. Thus, unlike the setting $\xi\geq2$, the distinction between exact finite-sample average optimality and worst-case minimax optimality disappears when $\xi=1$.} Specifically, in this setting, the optimal test involves a one-sample likelihood-ratio test.
Formally, at each time $t$ and based on the observation $X_t$, we form the likelihood-ratio value
$\ell_t$ defined in~\eqref{eq:likelihood_ratio}.
The Shewhart test compares $\ell_t$ with a fixed threshold $\alpha$ and declares a change when
$\ell_t$ exceeds $\alpha$. The main difference between the optimal test in the settings $\xi=1$ and $\xi\geq2$ is that,
when $\xi=1$, the optimal test compares the most recent likelihood ratio with a time-invariant
threshold, i.e., it is the Shewhart test. When $\xi\geq2$, the optimal test instead compares a
function of the most recent $\xi$ likelihood ratios with a
\textcolor{black}{state-dependent continuation boundary}
that depends on the most recent $(\xi-1)$ likelihood ratios.
The stopping time of the Shewhart test is
\begin{align}\label{eq:stop}
\tau_{\rm s}
\dff
\inf\,\{t\geq1:\ell_t\geq\alpha\}
=
\mathscr{T}(\alpha)\ .
\end{align}
The value of the threshold $\alpha$ is chosen such that the average run length to a false alarm
is guaranteed not to be smaller than $\varepsilon$, i.e.,
$\mathbb{E}_\infty[\tau_{\rm s}]\geq\varepsilon$.
Under $\mathbb{P}_\infty$, $\{\ell_t\}_{t\in\mathbb{N}}$ are i.i.d., so
$\tau_{\rm s}$ is geometric with parameter
$\mathbb{P}_\infty(\ell_1\geq\alpha)$, and hence
\begin{align}
\mathbb{E}_\infty[\tau_{\rm s}]
=
\frac{1}{\mathbb{P}_\infty(\ell_1\geq\alpha)}\ .
\end{align}
Accordingly, a convenient choice is to enforce the ARL constraint with equality, and
$\alpha$ can be computed by solving
\begin{align}\label{eq:threshold}
\mathbb{P}_\infty(\ell_1\geq\alpha)
=
\varepsilon^{-1}\ .
\end{align}
{\color{black}{Relationship in~\eqref{eq:threshold} assumes that the desired probability can be attained by a deterministic threshold. If the distribution of $\ell_1$ under $\mathbb{P}_\infty$ has an atom at the threshold, the Shewhart rule may instead randomize on the event $\{\ell_1=\alpha\}$ so that its one-step stopping probability is exactly $\varepsilon^{-1}$. For notational simplicity, we use the deterministic form below. Since $W_t^{(1)}=\ell_t$, Lemma~\ref{lemma:average_identity} yields
\begin{align}\label{eq:average_xi1}
\mathcal{A}_1(\tau)
=
\frac{\mathbb{E}_\infty[\ell_\tau]}
{\mathbb{E}_\infty[\tau]}\ .
\end{align}}}
An argument analogous to the one used to obtain
Theorem~\ref{thm:upper_bound2} yields the following $\xi=1$ upper bound:
\textcolor{black}{for any fixed $s\geq1$ and any}
$\tau\in\mathcal{T}$ with
$0<\mathbb{E}_\infty[\tau]<\infty$,
\begin{align}\label{eq:upper_xi1}
\mathcal{L}_{\rm L}(\tau)
\leq
\mathcal{L}_{\rm P}(\tau)
\leq
s\cdot
\frac{\mathbb{E}_\infty[\ell_\tau]}
{\mathbb{E}_\infty[\tau]}
\textcolor{black}{=s\,\mathcal{A}_1(\tau)}\ .
\end{align}
In this section, we establish the optimality of the Shewhart test formalized
in~\eqref{eq:stop} and~\eqref{eq:threshold} \textcolor{black}{for the average-success problem in~\eqref{eq:average_optimization_problem}, as well as} for the problems in
\eqref{eq:p1} and~\eqref{eq:p2}.
We leverage Lemma~\ref{lemma:eq} and prove the following properties:

\begin{enumerate}
\item The Shewhart test is feasible.
\item It maximizes \textcolor{black}{$\mathcal{A}_1(\tau)$ and, equivalently,} the upper bound on
$\mathcal{L}_{\rm P}(\tau)$ and
$\mathcal{L}_{\rm L}(\tau)$ in~\eqref{eq:upper_xi1},
subject to the false-alarm constraint.
\item For the Shewhart test, the
\textcolor{black}{performance criteria}
$\mathcal{L}_{\rm P}(\tau)$ and
$\mathcal{L}_{\rm L}(\tau)$
\textcolor{black}{attain}
the maximized upper bound\textcolor{black}{, so the rule is simultaneously exact average-optimal and exact minimax-optimal}.
\end{enumerate}

These properties are formalized in the following lemma and two theorems.

\begin{lemma}[Feasibility of Shewhart]\label{lemma:feasible}
The Shewhart stopping time $\tau_{\rm s}$ with threshold $\alpha$ chosen according
to~\eqref{eq:threshold} satisfies
$\mathbb{E}_\infty[\tau_{\rm s}]=\varepsilon$,
and hence is feasible for~\textcolor{black}{\eqref{eq:average_optimization_problem} and for} \eqref{eq:p1}--\eqref{eq:p2}.
\end{lemma}

\begin{proof}
See Appendix~\ref{app:lemma:feasible}.
\end{proof}
In the next theorem, we establish that the Shewhart test maximizes the upper bound
{\color{black}{in~\eqref{eq:upper_xi1}, which is exactly the average criterion up to the factor $s$}}.

\begin{theorem}\label{lemma:upper_bound}
The Shewhart test is a solution to
\begin{align}\label{eq:p3}
\sup_{\tau\in\mathcal{T}}
\quad &
\frac{\mathbb{E}_\infty[\ell_\tau]}
{\mathbb{E}_\infty[\tau]}
\qquad
\text{\rm s.t.}
\qquad
\mathbb{E}_\infty[\tau]
=
\varepsilon\ .
\end{align}
\textcolor{black}{Equivalently, it solves}
\begin{align}
\textcolor{black}{
\sup_{\tau\in\mathcal{T}}
\quad
\mathcal{A}_1(\tau)
\qquad
\text{\rm s.t.}
\qquad
\mathbb{E}_\infty[\tau]
=
\varepsilon\ .}
\label{eq:p3_average}
\end{align}
\end{theorem}

\begin{proof}
See Appendix~\ref{app:lemma:upper_bound}.
\end{proof}
The following theorem proves that, for the Shewhart test, the
\textcolor{black}{performance criteria attain their common upper bound}
and, consequently, that the Shewhart test is an optimal solution to
\textcolor{black}{the average-success problem~\eqref{eq:average_optimization_problem} and to}
\eqref{eq:p1} and~\eqref{eq:p2} for $\xi=1$.

\begin{theorem}[Optimality of Shewhart for $\xi=1$]\label{th:Sh}
\textcolor{black}{Fix $s\geq1$ and take the infima defining
$\mathcal{L}_{\rm P}$ and $\mathcal{L}_{\rm L}$ over $\Theta_s(1)$.}
The Shewhart test with the stopping time and threshold given in
\eqref{eq:stop} and~\eqref{eq:threshold}, respectively, is an optimal
solution to \textcolor{black}{the survival-weighted average problem and to} both problems in~\eqref{eq:p1} and~\eqref{eq:p2}, i.e.,
\begin{align}
\textcolor{black}{\mathcal{A}_1(\tau_{\rm s})}
&=
\textcolor{black}{
\sup_{\tau\in\mathcal{T}:\,
\mathbb{E}_\infty[\tau]\geq\varepsilon}
\mathcal{A}_1(\tau)}
\label{eq:xi1_average_optimality}
\\
\mathcal{L}_{\rm L}(\tau_{\rm s})
&=
\mathcal{L}_{\rm P}(\tau_{\rm s})
\textcolor{black}{=s\,\mathcal{A}_1(\tau_{\rm s})}
=
\sup_{\tau\in\mathcal{T}:\,
\mathbb{E}_\infty[\tau]\geq\varepsilon}
\mathcal{L}_{\rm L}(\tau)
=
\sup_{\tau\in\mathcal{T}:\,
\mathbb{E}_\infty[\tau]\geq\varepsilon}
\mathcal{L}_{\rm P}(\tau)\ .
\label{eq:xi1_minimax_optimality}
\end{align}
\end{theorem}

\begin{proof}
\textcolor{black}{The average-optimality statement follows from Lemma~\ref{lemma:eq}, Theorem~\ref{lemma:upper_bound}, and the identity~\eqref{eq:average_xi1}. The minimax equalities and optimality statements are proved in} Appendix~\ref{app:th:Sh}.
\end{proof}
Finally, the Shewhart test is not only optimal but also simple to implement:
at each time $t$, it computes the likelihood ratio $\ell_t$ and compares it
with the fixed threshold $\alpha$, stopping and declaring a change at the first
time $\ell_t\geq\alpha$.
\textcolor{black}{Thus, the one-sample member of the TSR family is precisely the classical Shewhart rule: the finite-memory statistic reduces to the current likelihood ratio, and the state-dependent continuation boundary reduces to a constant threshold.}
\textcolor{black}{The one-sample setting is therefore the boundary case in which the exact average-optimal rule also attains the Pollak- and Lorden-type worst-case optima at finite $\varepsilon$.}

\section{\textcolor{black}{Structure, Computation, and Complexity}}\label{sec:comp}

{\color{black}For $\xi\geq2$, the exact average-optimal TSR stopping rule in}
Theorem~\ref{thm:general} can be written as
\begin{align}\label{eq:tau_opt_xi_comp}
{\tau_{\lambda,\xi}^\star}
\dff
\inf\Bigg\{
t\ge 1:
{W_t^{(\xi)}}
\ge
-\lambda+
{R_\xi}\big({\ell_{\xi-1}^{\,t}}\big)
\Bigg\}\ .
\end{align}
Here $\lambda=\lambda_\varepsilon$ is chosen so that
$\mathbb{E}_\infty[\tau_{\lambda_\varepsilon,\xi}^\star]
=\varepsilon$, and we write
$\tau_\xi^\star\dff
\tau_{\lambda_\varepsilon,\xi}^\star$.
Note that, for $\xi=2$, $R_\xi$ reduces to the scalar continuation
function $R:[0,\infty)\to\mathbb{R}$ used below. Here
\begin{align}
\ell_{\xi-1}^{\,t}
=
(\ell_{t-\xi+2},\ldots,\ell_t)
\end{align}
and
\begin{align}\label{eq:stop_reward_xi}
{W_t^{(\xi)}}
\dff
\sum_{k=1}^{\xi}
\prod_{{j}=0}^{k-1}\ell_{t-{j}}
=
\ell_t
\Bigg(
1+
\sum_{k=1}^{\xi-1}
\prod_{{j}=1}^{k}\ell_{t-{j}}
\Bigg)\ .
\end{align}
\textcolor{black}{The online implementation consists of two distinct
operations: updating the TSR statistic
$W_t^{(\xi)}$ and evaluating the state-dependent continuation
boundary.}
Implementation of~\eqref{eq:tau_opt_xi_comp} also requires evaluating
the \textcolor{black}{continuation value}
$R_\xi(\ell_{\xi-1}^{\,t})$, which can be computed \emph{offline} by
dynamic programming and stored for real-time use.

{\color{black}This separation makes the scalability of the method explicit. Online implementation requires $O(\xi)$ arithmetic operations and $O(\xi)$ memory, while for $\xi=2$ it reduces to one scalar state-dependent threshold evaluation using only the current and previous likelihood ratios. The principal limitation is offline: a direct tensor-grid representation with $M$ points per state coordinate has order $M^{\xi-1}$ states, before quadrature and interpolation. Thus, exact boundary computation is practical for small windows, especially $\xi=2$ and potentially $\xi=3$; substantially larger windows require sparse grids, function approximation, or approximate dynamic programming.}

\paragraph{Structural Properties of the Value Functions.}
The following property is the key enabler for an efficient
offline--online implementation.
For fixed past likelihood ratios
$\ell_{t-\xi+1}^{\,t-1}$, the immediate stopping reward
$W_t^{(\xi)}$ in~\eqref{eq:stop_reward_xi} is \emph{affine} in the
current likelihood ratio $\ell_t$.
Moreover, for fixed $\ell_{t-\xi+2}^{\,t-1}$, the continuation value
$R_\xi(\ell_{\xi-1}^{\,t})$ is \emph{nondecreasing} and
\emph{convex} in $\ell_t$.

\begin{lemma}[Monotonicity and convexity]
\label{lemma:mono_convex}
Fix $\lambda>0$ and $\xi\ge 2$. For any fixed values of the past
likelihood ratios $\ell_{t-\xi+1}^{\,t-1}$, the reward-to-go
$J_t^{\xi}$ and the continuation value $R_t^{\xi}$ are
nondecreasing and convex functions of the current likelihood ratio
$\ell_t$, when holding
$\ell_{t-\xi+1}^{\,t-1}$ fixed.
\end{lemma}

\begin{proof}
See Appendix~\ref{proof:lemma:mono_convex}.
\end{proof}
{\color{black}The monotonicity has a direct continuation-value interpretation. For $\xi=2$, if the procedure continues from time $t$ to time $t+1$, the current likelihood ratio becomes the lagged component of the next stopping reward,
\begin{align}
W_{t+1}^{(2)}=\ell_{t+1}(1+\ell_t).
\end{align}
Increasing $\ell_t$ therefore increases the next-step stopping reward for every realization of $\ell_{t+1}$; taking the maximum of stopping and continuing and then averaging preserves this pointwise monotonicity. For a general window, the same principle applies because strong current evidence remains in the finite-memory state and contributes to future TSR products until it exits the window. This does not mean that stronger evidence necessarily discourages immediate stopping, since the immediate stopping reward also increases with $\ell_t$. The decision is governed by the difference between the immediate reward and the continuation value; the Lipschitz result below controls how rapidly the latter can grow in the two-sample case.}

{\color{black}Lemma~\ref{lemma:mono_convex} implies that, for each fixed history $\ell_{t-\xi+1}^{\,t-1}$, the difference between the affine stopping reward and the convex continuation value is concave in $\ell_t$. Hence the stopping set is an interval in the current likelihood ratio (possibly empty, unbounded, or degenerate), and its boundary has at most two points.} In the important special case
$\xi=2$, this boundary can be sharpened to a
\textcolor{black}{single state-dependent threshold in the current
likelihood ratio} via the Lipschitz property below, which
substantially simplifies computation.

\paragraph{Single-threshold reduction for $\xi=2$.}
When $\xi=2$, the stopping rule depends on
$(\ell_{t-1},\ell_t)$ and has the form
\begin{align}\label{eq:tau_opt_xi2_comp}
\tau_\lambda^\star
\dff
\inf\Big\{
t\ge 1:
\ell_t+\ell_{t-1}\ell_t
\ge
-\lambda+R(\ell_t)
\Big\}\ .
\end{align}
For each fixed $\ell_{t-1}$, define the scalar equation
\begin{align}\label{eq:threshold_equation_comp}
(1+\ell_{t-1})\,\ell
=
-\lambda+R(\ell)\ .
\end{align}

\begin{lemma}[Lipschitz property of $R$ for $\xi=2$]
\label{lemma:lipschitz_R}
Assume $\xi=2$. Under $\mathbb{P}_\infty$, the function $R$ is
nondecreasing and $1$-Lipschitz: for all
$\ell_2\ge\ell_1\ge0$,
\begin{align}\label{eq:lipschitz_R}
0
\le
R(\ell_2)-R(\ell_1)
\le
\ell_2-\ell_1\ .
\end{align}
\end{lemma}

\begin{proof}
See Appendix~\ref{proof:lemma:lipschitz_R}.
\end{proof}

\begin{corollary}[\color{black}{Unique state-dependent threshold for $\xi=2$}]
\label{cor:unique_threshold}
Assume $\xi=2$. For any fixed $\ell_{t-1}>0$, the function
\begin{align}
g_{\ell_{t-1}}(\ell)
\dff
(1+\ell_{t-1})\ell+\lambda-R(\ell)
\end{align}
is strictly increasing in $\ell$, and therefore
\eqref{eq:threshold_equation_comp} admits at most one solution.
When a solution exists, denote it by $\Gamma(\ell_{t-1})$.
Then the {\color{black}average-optimal} stopping rule~\eqref{eq:tau_opt_xi2_comp} reduces
to the single-threshold form
\begin{align}\label{eq:single_threshold_rule}
\tau_\lambda^\star
=
\inf\{
t\ge1:
\ell_t\ge\Gamma(\ell_{t-1})
\}\ .
\end{align}
\end{corollary}

\paragraph{Computational implication.}
Corollary~\ref{cor:unique_threshold} turns the online decision into a
single comparison in~\eqref{eq:single_threshold_rule}.
Moreover, since $g_{\ell_{t-1}}(\cdot)$ is increasing,
$\Gamma(\ell_{t-1})$ can be computed robustly by bisection.
\textcolor{black}{Once the function $\Gamma(\cdot)$ has been computed
offline, the real-time decision for $\xi=2$ requires only the current
and previous likelihood ratios and a one-dimensional boundary
evaluation.}


\paragraph{Offline computation of $R$.}
The continuation value $R(\cdot)$ can be computed offline via value
iteration or, equivalently, by truncating to a large finite horizon
and iterating backward, using the stationary Bellman equation
specialized to $\xi=2$:
\begin{align}\label{eq:bellman_xi2_comp}
R(\ell)
=
\mathbb{E}_\infty
\left[
\max\{
L(1+\ell),
-\lambda+R(L)
\}
\right]\ ,
\end{align}
where $L\overset{d}{=}\ell_1$ under
$\mathbb{P}_\infty$.
A practical implementation discretizes $\ell$ on a grid
$\{\ell^{(m)}\}_{m=1}^M$ and approximates the expectation in
\eqref{eq:bellman_xi2_comp} via numerical quadrature, or Monte Carlo
when needed. Each iteration then costs $O(M)$ evaluations of the
right-hand side, multiplied by the number of quadrature nodes, and
convergence is typically achieved in a moderate number of iterations.
For general $\xi>2$, the sufficient statistic is
\begin{align}
\ell_{\xi-1}^{\,t}
=
(\ell_{t-\xi+2},\ldots,\ell_t)
\in[0,\infty)^{\xi-1}
\end{align}
(cf. Section~\ref{sec:Xi}), and the same approach applies on an
$(\xi-1)$-dimensional grid. In this case, the cost grows as
$O(M^{\xi-1})$.
In this regime, Lemma~\ref{lemma:mono_convex} is especially useful
because it implies that, for each fixed history
$\ell_{t-\xi+1}^{\,t-1}$, the stopping region is described by a
low-complexity boundary in the scalar variable $\ell_t$.
Thus, \textcolor{black}{for online implementation,} it suffices to
store the boundary rather than the full value function on the entire
grid\textcolor{black}{, although computing that boundary offline
still requires approximating the multidimensional continuation
value}.
\textcolor{black}{Accordingly, the principal scalability limitation
for large $\xi$ is the $(\xi-1)$-dimensional offline state space, not
the update of the truncated Shiryaev--Roberts statistic itself.}

\paragraph{Calibrating $\lambda$ to meet
$\mathbb{E}_\infty[\tau]=\varepsilon$.}
{\color{black}The multiplier $\lambda$ is chosen so that the resulting average-optimal rule}
satisfies the ARL constraint with equality.
For $\xi=2$ under the threshold representation
\eqref{eq:single_threshold_rule}, the ARL can be computed without
Monte Carlo as the solution to a fixed-point equation.
Let $m(x)$ denote the expected remaining time to stop given
$\ell_{t-1}=x$ under $\mathbb{P}_\infty$. Then
\begin{align}
\label{eq:arl_fixed_point}
m(x)= 1+ \mathbb{E}_\infty
\big[ m(L)\ , \mathbf{1}\{L<\Gamma(x)\}
\big]\ ,
\qquad
L\overset{d}{=}\ell_1\ ,
\end{align}
and
\begin{align}
\mathbb{E}_\infty[
\textcolor{black}{\tau_\lambda^\star}
]
=
m(0)\ ,
\end{align}
where we recall that $\ell_0=0$.
Equation~\eqref{eq:arl_fixed_point} can be solved numerically on the
same grid used for $R$, yielding a fast and stable calibration of
$\lambda$, for example through bisection in $\lambda$, since a larger
$\lambda$
\textcolor{black}{assigns greater cost to continued sampling in the
auxiliary optimal-stopping problem and typically decreases}
the ARL.

{\color{black}
\section{Minimax Guarantees and Constant-Boundary Asymptotics}
\label{sec:constant_tsr_asymptotics}

Sections~\ref{sec:Xi2}--\ref{sec:Xi1} established exact finite-sample optimality under the survival-weighted average criterion. We now return to the Pollak- and Lorden-type worst-case criteria in~\eqref{eq:p1}--\eqref{eq:p2} and for $\xi\geq 2$, we identify a joint growing ARL and $\xi$ regime in which a simpler constant-boundary TSR procedure is asymptotically minimax optimal. Intuitively, an ARL requirement of order $\varepsilon$ imposes a logarithmic evidence scale of order $\log\varepsilon$, while a window of length $\xi$ provides approximately $\xi {\sf D}_{\rm KL}(F_1\Vert F_0)$ units of post-change log-likelihood information. When the latter dominates, the constant-boundary rule succeeds uniformly with probability approaching one and attains the largest possible minimax score. To formalize these, we start with specifying a constant-boundary TSR test.

A related distinction between a structurally complex optimal rule and a simpler asymptotically optimal threshold rule appears in quickest detection of a Markov process across a sensor array~\cite{raghavan2010quickest}.

\paragraph{Constant-boundary TSR procedure.}
For a boundary $A>0$ and an admissible-window length $\xi\geq1$, define
\begin{align}
\tau^{\rm c}_{A,\xi}
&\triangleq
\inf\left\{
 t\geq1:
 W_t^{(\xi)}\geq A
\right\}\; ,
\label{eq:constant_tsr_rule}
\end{align}
with the convention $\inf\varnothing=\infty$. The rule uses the same finite-memory TSR statistic as the average-optimal procedure, but replaces its state-dependent continuation boundary by the scalar $A$.

\begin{lemma}[ARL guarantee for a constant boundary]
\label{lemma:constant_tsr_arl}
For every $\xi\geq1$ and $A>0$,
\begin{align}
\mathbb E_\infty
\left[
\tau^{\rm c}_{A,\xi}
\right]
&\geq
A\; .
\label{eq:constant_tsr_arl}
\end{align}
\end{lemma}

\begin{proof}
See Appendix~\ref{proof:lemma:constant_tsr_arl}.
\end{proof}
Therefore, the choice $A=\varepsilon$ guarantees the ARL constraint. In the asymptotic regime considered below, Theorem~\ref{thm:constant_tsr_asymptotic_optimality} additionally establishes that the resulting stopping time is proper and has finite mean.

\paragraph{Uniform finite-window success guarantee.}
Let $\{Z_r:r\geq1\}$ be i.i.d. according to $F_1$, and define the post-change log-likelihood random walk by
\begin{align}
Y_r
&\triangleq
\log \ell(Z_r)\; ,
\qquad
S_k
\triangleq
\sum_{r=1}^{k}Y_r\; .
\label{eq:postchange_partial_llr}
\end{align}
For $A>0$ and $\xi\geq1$, define the finite-window crossing probability
\begin{align}
\bar p_\xi(A)
&\triangleq
\mathbb P_1
\left(
\max_{1\leq k\leq\xi}S_k
\geq
\log A
\right)\; .
\label{eq:finite_window_crossing_probability}
\end{align}
This is the probability that the likelihood-ratio product associated with the true onset crosses $A$ at some point in the admissible window. To aggregate this guarantee over all $s$ episodes, we note that it can be readily shown that there exists $b\in(0,1)$ such that
\begin{align}
\mathbb P_1(\ell_1\leq b)
&>
0\; .
\label{eq:survivability_condition}
\end{align}
Since $F_1\ll F_0$, condition~\eqref{eq:survivability_condition} also implies $\mathbb P_\infty(\ell_1\leq b)>0$. It ensures that every finite episode onset remains reachable with positive probability for the constant-boundary rules considered below.

\begin{lemma}[Uniform window-success lower bound]
\label{lemma:constant_tsr_detection}
Fix $s\geq1$, $\xi\geq1$, and $A>0$, and let $\tau=\tau^{\rm c}_{A,\xi}$. For every $\theta\in\Theta_s(\xi)$ and every episode onset $\gamma_i$,
\begin{align}
\mathbb E_\theta
\left[
\mathbf 1_{\{\gamma_i\leq\tau<\gamma_i+\xi\}}
\,\middle|\,
\mathcal F_{\gamma_i-1}
\right]
&\geq
\bar p_\xi(A)
\mathbf 1_{\{\tau\geq\gamma_i\}}\; .
\label{eq:conditional_window_success_bound}
\end{align}
Whenever $\mathbb P_\theta(\tau\geq\gamma_i)>0$, the corresponding episode-wise Pollak and Lorden success terms are therefore at least $\bar p_\xi(A)$. If condition~\eqref{eq:survivability_condition} holds and
\begin{align}
A
&>
\frac{b}{1-b}\; ,
\label{eq:survivability_boundary_condition}
\end{align}
then $\mathbb P_\theta(\tau\geq\gamma_i)>0$ for every $\theta\in\Theta_s(\xi)$ and every $i\in\{1,\ldots,s\}$, and
\begin{align}
\mathcal L_Q
\left(
\tau^{\rm c}_{A,\xi}
\right)
&\geq
s\,\bar p_\xi(A)\; ,
\qquad
Q\in\{{\rm L},{\rm P}\}\; .
\label{eq:aggregate_window_success_bound}
\end{align}
\end{lemma}

\begin{proof}
See Appendix~\ref{proof:lemma:constant_tsr_detection}.
\end{proof}
For $Q\in\{{\rm L},{\rm P}\}$, define the optimal minimax value
\begin{align}
\mathsf V_Q^\star(\varepsilon,\xi;s)
&\triangleq
\sup_{\tau\in\mathcal T:\,
\mathbb E_\infty[\tau]\geq\varepsilon}
\mathcal L_Q(\tau)\; .
\label{eq:true_minimax_value}
\end{align}
Because $\mathcal L_Q$ is the sum of $s$ episode-wise success probabilities,
\begin{align}
\mathsf V_Q^\star(\varepsilon,\xi;s)
&\leq
s\; .
\label{eq:trivial_minimax_upper_bound}
\end{align}
Consequently, whenever $\tau^{\rm c}_{A,\xi}$ is feasible and~\eqref{eq:survivability_boundary_condition} holds, Lemma~\ref{lemma:constant_tsr_detection} gives
\begin{align}
s\,\bar p_\xi(A)
&\leq
\mathcal L_Q
\left(
\tau^{\rm c}_{A,\xi}
\right)
\leq
\mathsf V_Q^\star(\varepsilon,\xi;s)
\leq
s\; ,
\qquad
Q\in\{{\rm L},{\rm P}\}\; .
\label{eq:finite_sample_tsr_sandwich}
\end{align}

\paragraph{Joint ARL--window asymptotics.}
Assume
\begin{align}
\mathbb P_1(\ell_1>0)
&=
1\; ,
\qquad
\mathbb E_1
\left[
\left|\log\ell_1\right|
\right]
<
\infty\; ,
\label{eq:asymptotic_tsr_integrability}
\end{align}
and define the post-change information number
\begin{align}
I
&\triangleq
\mathbb E_1[\log\ell_1]
=
{\sf D}_{\rm KL}(F_1\Vert F_0)
\in
(0,\infty)\; .
\label{eq:postchange_information_number}
\end{align}

\begin{theorem}[Asymptotic minimax optimality of a constant-boundary TSR test]
\label{thm:constant_tsr_asymptotic_optimality}
Let $\varepsilon\to\infty$, and let $\xi_\varepsilon\to\infty$ be an integer-valued sequence satisfying
\begin{align}
\limsup_{\varepsilon\to\infty}
\frac{\log\varepsilon}{\xi_\varepsilon}
&<
I\; .
\label{eq:constant_tsr_growth_condition}
\end{align}
Define
\begin{align}
\tau_\varepsilon^{\rm c}
&\triangleq
\tau^{\rm c}_{\varepsilon,\xi_\varepsilon}\; .
\label{eq:asymptotic_constant_tsr_rule}
\end{align}
Then, for every fixed $s\geq1$, the following statements hold.
\begin{enumerate}
\item For all sufficiently large $\varepsilon$, the procedure is proper and feasible:
\begin{align}
\varepsilon
&\leq
\mathbb E_\infty
\left[
\tau_\varepsilon^{\rm c}
\right]
<
\infty\; .
\label{eq:constant_tsr_proper_feasible}
\end{align}

\item Under both minimax criteria,
\begin{align}
\mathcal L_Q
\left(
\tau_\varepsilon^{\rm c}
\right)
&\longrightarrow
s\; ,
\qquad
Q\in\{{\rm L},{\rm P}\}\; .
\label{eq:constant_tsr_score_to_s}
\end{align}

\item The procedure is asymptotically optimal in the additive sense:
\begin{align}
\mathsf V_Q^\star
\left(
\varepsilon,\xi_\varepsilon;s
\right)
-
\mathcal L_Q
\left(
\tau_\varepsilon^{\rm c}
\right)
&\longrightarrow
0\; ,
\qquad
Q\in\{{\rm L},{\rm P}\}\; .
\label{eq:constant_tsr_additive_optimality}
\end{align}

\item The procedure is asymptotically optimal in the relative sense:
\begin{align}
\frac{
\mathcal L_Q
\left(
\tau_\varepsilon^{\rm c}
\right)
}{
\mathsf V_Q^\star
\left(
\varepsilon,\xi_\varepsilon;s
\right)
}
&\longrightarrow
1\; ,
\qquad
Q\in\{{\rm L},{\rm P}\}\; .
\label{eq:constant_tsr_relative_optimality}
\end{align}
\end{enumerate}
\end{theorem}
\begin{proof}
See Appendix~\ref{proof:thm:constant_tsr_asymptotic_optimality}.
\end{proof}
Condition~\eqref{eq:constant_tsr_growth_condition} places the admissible window strictly above the information scale. In fact, it implies
\begin{align}
\xi_\varepsilon I
-
\log\varepsilon
&\longrightarrow
\infty\; .
\label{eq:information_scale_interpretation}
\end{align}
Thus, the typical post-change information accumulated within the admissible window eventually exceeds the logarithmic evidence required by an ARL of order $\varepsilon$ by a diverging margin.

\begin{corollary}[Window length above the information scale]
\label{cor:window_above_information_scale}
For any fixed $\delta>0$, let
\begin{align}
\xi_\varepsilon
&\triangleq
\left\lceil
\frac{(1+\delta)\log\varepsilon}{I}
\right\rceil\; .
\label{eq:explicit_information_scaling}
\end{align}
Then $\tau^{\rm c}_{\varepsilon,\xi_\varepsilon}$ satisfies all conclusions of Theorem~\ref{thm:constant_tsr_asymptotic_optimality}. The same conclusions hold, whenever such finite thresholds exist, for exactly ARL-calibrated boundaries $\widehat A_\varepsilon$ satisfying
\begin{align}
\mathbb E_\infty
\left[
\tau^{\rm c}_{\widehat A_\varepsilon,\xi_\varepsilon}
\right]
&=
\varepsilon\; .
\label{eq:exactly_calibrated_constant_boundary}
\end{align}
\end{corollary}

\begin{proof}
See Appendix~\ref{proof:cor:window_above_information_scale}.
\end{proof}
The corollary shows that, above the scale $\log\varepsilon/{\sf D}_{\rm KL}(F_1\Vert F_0)$, the state-dependent continuation boundary is unnecessary for first-order minimax optimality. This conclusion does not extend to a fixed short window: if $\xi$ remains fixed while $\varepsilon\to\infty$, condition~\eqref{eq:constant_tsr_growth_condition} fails and the finite-window success probability may vanish. For the criteria normalized by $s$, the limiting value in Theorem~\ref{thm:constant_tsr_asymptotic_optimality} is one.
}

\section{Numerical Evaluations}
\label{sec:sim}

{\color{black}
In this section, we numerically illustrate the finite-window behavior of the TSR procedures and compare them with classical sequential detectors. Because the criteria in~\eqref{eq:Pollak_criterion}--\eqref{eq:Lorden_criterion} sum the episode-wise success probabilities, the reported detection probabilities are shown on a per-episode scale (equivalently, the criteria are normalized by $s$ when multiple episodes are simulated). For Monte Carlo comparisons labeled ``worst-case detection probability,'' the infimum over the onset time is approximated by the minimum over the onset locations included in the simulation. Detection thresholds are calibrated under $\mathbb P_\infty$, and when a horizontal axis reports ARL we use the empirically attained value of $\mathbb E_\infty[\tau]$ so that residual calibration differences remain visible.
The exact finite-sample optimality results for $\xi\geq2$ established in Sections~\ref{sec:Xi2}--\ref{sec:Xi} concern the survival-weighted average criterion $\mathcal A_\xi$; the Pollak-type and grid-worst quantities retained below therefore provide complementary onset-robustness diagnostics rather than the criterion for which exact finite-sample optimality is claimed.

We begin with a unit-variance Gaussian mean-shift model, $F_0=\mathcal N(0,1)$ and $F_1=\mathcal N(\mu_1,1)$. Figure~\ref{fig:decay_PL} illustrates the decay of the per-episode Pollak-type success probability as the false-alarm requirement $\mathbb E_\infty[\tau]\geq\varepsilon$ becomes more stringent, for $\mu_1=1$ and $2$ and for admissible-window lengths $\xi=1,2,3$. In the weaker-signal regime ($\mu_1=1$), the success probability decreases markedly with $\varepsilon$; $\xi=1$ exhibits the steepest decline, while $\xi=3$ consistently benefits from the longer admissible window. For the stronger signal ($\mu_1=2$), the decay is substantially milder and the probabilities remain close to one over the displayed range, so the incremental benefit of enlarging $\xi$ is less pronounced.
}

\begin{figure}
    \centering
    \includegraphics[width=0.95\linewidth]{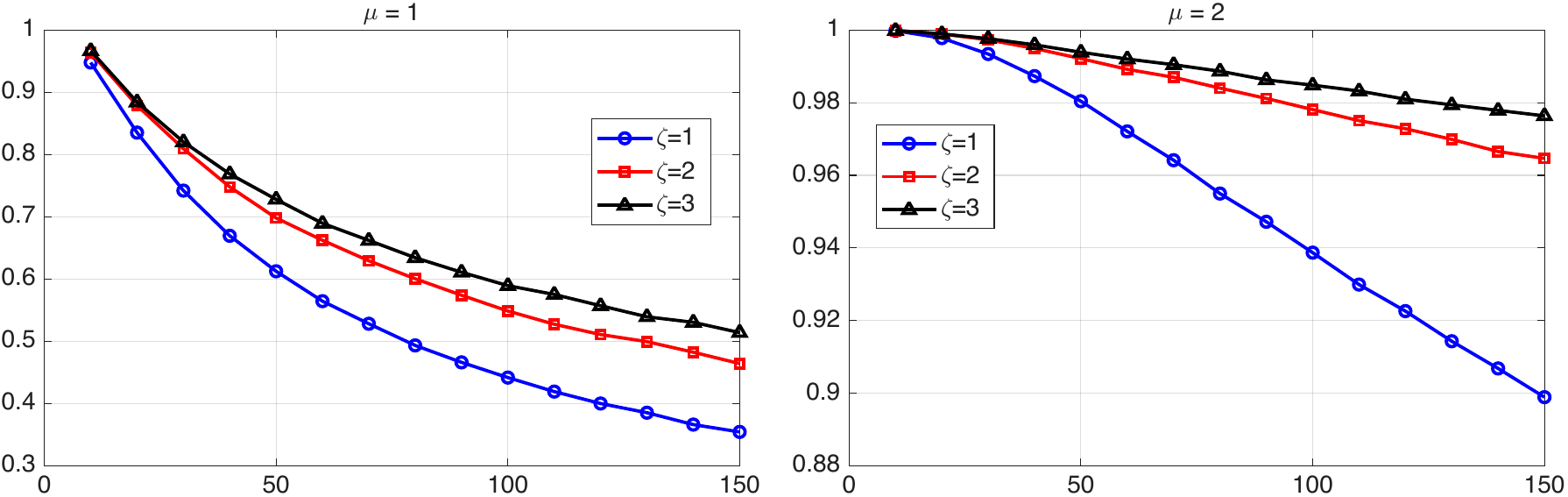}
    \caption{{\color{black}Per-episode Pollak-type success probability versus the ARL false-alarm constraint $\varepsilon$ for two Gaussian mean shifts and admissible-window lengths $\xi=1,2,3$.}}
    \label{fig:decay_PL}
\end{figure}

Figure~\ref{fig:convergence} shows the convergence behavior of the numerical procedure used to compute $R(\ell)$, plotting the error versus the iteration index $n$ on a logarithmic scale. The vertical axis represents the error in computing $R(\ell)$ (log scale), while the horizontal axis shows the iteration number up to $n=500$. The curve exhibits a sharp initial decrease in error during the first few iterations, followed by a steady and nearly linear decay on the logarithmic scale, indicating approximately exponential convergence. After roughly $400$ iterations, the error is reduced by several orders of magnitude (from about $10^2$ down to below $10^{-2}$), demonstrating both numerical stability and strong contraction properties of the underlying recursion. The smooth monotonic decline without visible oscillations further suggests that the update rule for $R(\ell)$ is well-conditioned and does not suffer from instability or divergence, thereby confirming the effectiveness of the computational scheme.

\begin{figure}[t]
    \centering
    \begin{minipage}{0.48\linewidth}
        \centering
        \includegraphics[width=\linewidth]{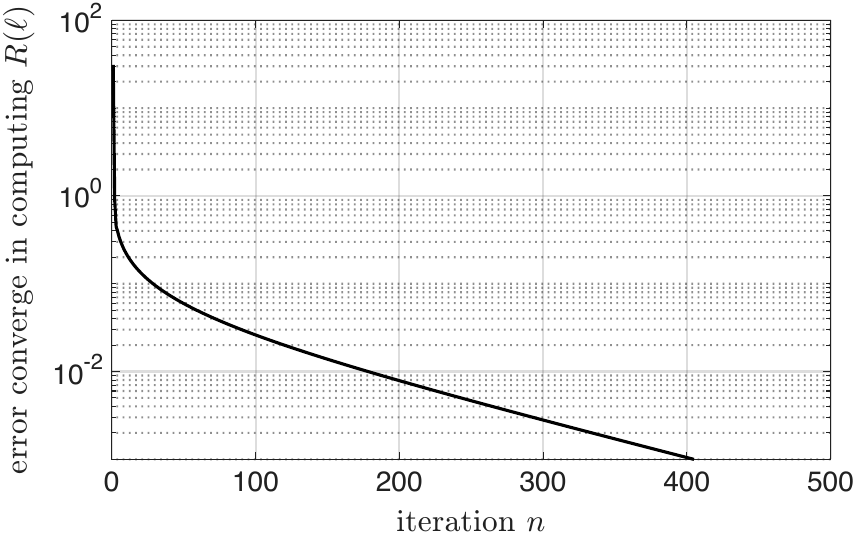}
        \caption{Convergence of iteratively  computing the state-dependent  $R(\ell)$.}
        \label{fig:convergence}
    \end{minipage}
    \hfill
    \begin{minipage}{0.48\linewidth}
        \centering
        \includegraphics[width=\linewidth]{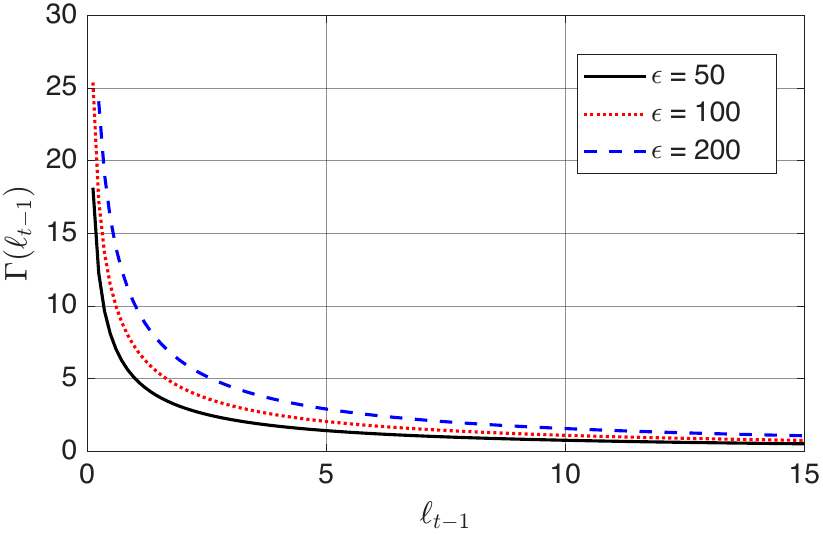}
        \caption{{\color{black}State-dependent current-likelihood-ratio threshold $\Gamma(\ell_{t-1})$ for $\xi=2$ under different ARL constraints.}}
        \label{fig:threshold_gamma}
    \end{minipage}
\end{figure}

{\color{black}
Figure~\ref{fig:threshold_gamma} illustrates the single-threshold representation of the $\xi=2$ rule by plotting $\Gamma(\ell_{t-1})$ against the previous likelihood ratio for $\varepsilon=50,100,$ and $200$. Recall that
\begin{align}
\tau_\lambda^\star
=
\inf\left\{
t:\ell_t(1+\ell_{t-1})\geq-\lambda+R(\ell_t)
\right\},
\end{align}
and, for fixed $\ell_{t-1}=\ell$, $\Gamma(\ell)$ is the solution of $x(1+\ell)=-\lambda+R(x)$ whenever that solution exists. Thus, $\Gamma(\ell)$ is the minimum current likelihood ratio required for stopping given the retained evidence $\ell_{t-1}$. The curves decrease with $\ell_{t-1}$, illustrating that stronger recent evidence lowers the amount of new evidence required to stop. Increasing the ARL requirement shifts the threshold upward, reflecting the more conservative decision rule required for stronger false-alarm protection. Importantly, this is a state-dependent infinite-horizon boundary; the horizontal axis is $\ell_{t-1}$, not time.
}

\textcolor{black}{To assess the onset sensitivity of the average-optimal TSR rule,} we evaluate the detection probability as a function of the change-point location. Specifically, for each candidate transient onset time $\gamma$, we estimate
\begin{align}
\widehat{\pi}_\gamma
=
\mathbb{P}_\gamma\left(
\gamma\leq\tau<\gamma+\xi
\mid
\tau\geq\gamma
\right),
\end{align}
that is, the probability of detecting within the admissible window conditioned on not having stopped before the change. \textcolor{black}{The exact average-optimality theory does not require this conditional probability to be constant across $\gamma$.} Figure~\ref{fig:equalizer_validation} shows a nearly flat profile around $0.21$--$0.23$. \textcolor{black}{Thus, in this numerical regime, optimizing the survival-weighted average does not appear to sacrifice performance at particular onset locations.} The remaining fluctuations are consistent with the finite Monte Carlo uncertainty indicated by the confidence intervals; no finite-horizon or terminal-time effect is involved, since the stopping problem itself is infinite-horizon.

\begin{figure}[t]
    \centering
    \begin{minipage}{0.48\linewidth}
        \centering
        \includegraphics[width=\linewidth]{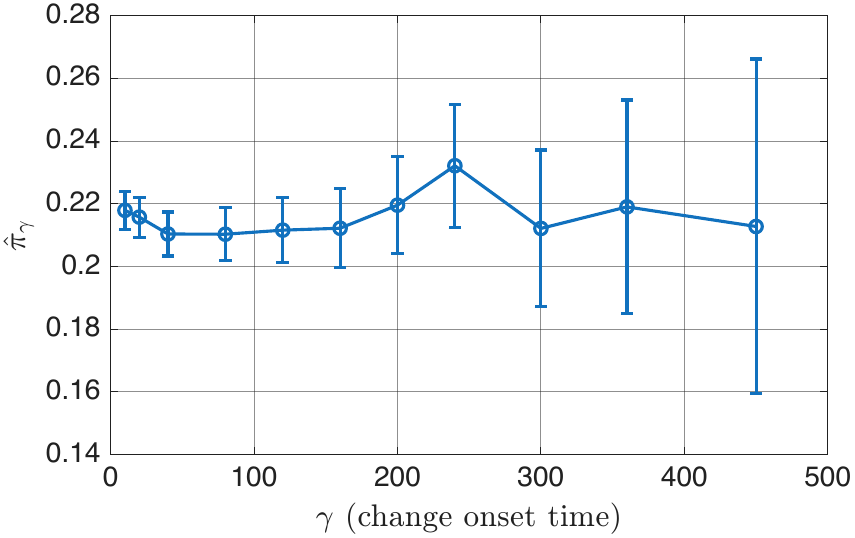}
        \captionof{figure}{Estimated conditional detection probability $\widehat{\pi}_\gamma$ as a function of the change-point location $\gamma$, \textcolor{black}{illustrating the onset sensitivity of the average-optimal TSR rule.}}
        \label{fig:equalizer_validation}
    \end{minipage}
    \hfill
    \begin{minipage}{0.48\linewidth}
        \centering
        \includegraphics[width=\linewidth]{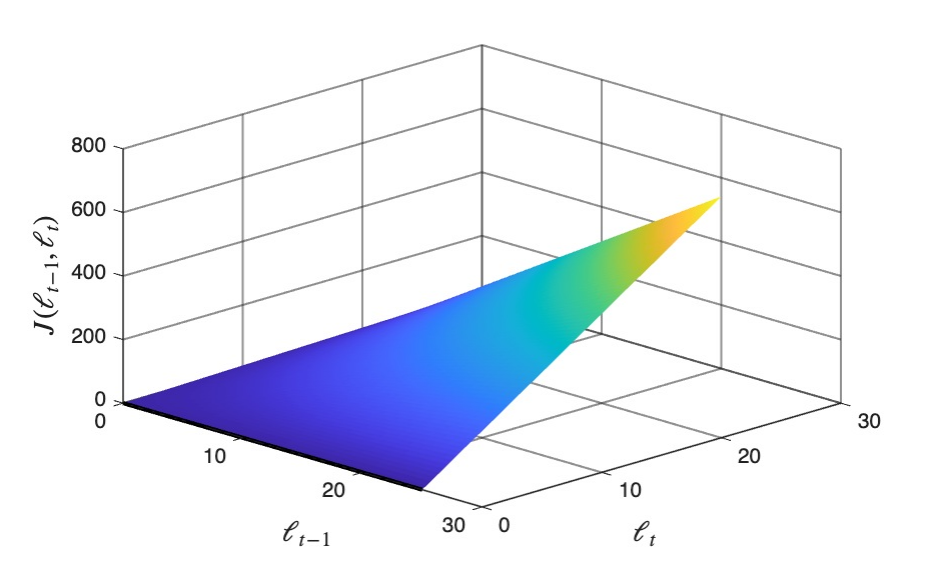}
        \captionof{figure}{{\color{black}Surface plot of the two-state value function $J(\ell_{t-1},\ell_t)$, illustrating its dependence on the retained and current likelihood ratios.}}
        \label{fig:two_state_geometry}
    \end{minipage}
\end{figure}

{\color{black}
Figure~\ref{fig:two_state_geometry} shows the value function $J(\ell_{t-1},\ell_t)$ for $\xi=2$. This is not a second detection statistic: the TSR detection statistic is $W_t^{(2)}$, whereas $J$ is the reward-to-go appearing in the optimal-stopping formulation. The surface makes explicit the two-dimensional state $(\ell_{t-1},\ell_t)$ and increases with both retained and current evidence, consistent with the monotonicity established in Section~\ref{sec:comp}. Its geometry also reflects the multiplicative term $\ell_t(1+\ell_{t-1})$: stronger retained evidence increases the value associated with a given current likelihood ratio and underlies the state dependence of the stopping decision.
}

{\color{black}
\begin{figure}[t]
    \centering
    \subfigure[Bounded-LR variance contraction.]{
        \includegraphics[width=0.48\linewidth]{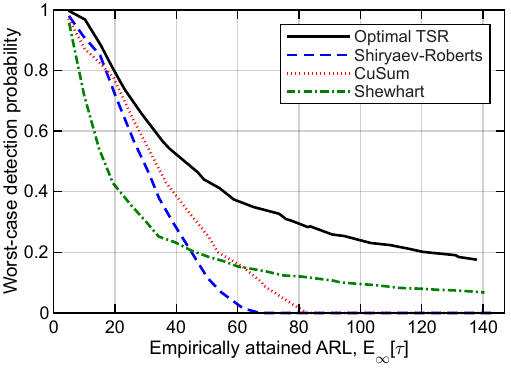}}
    \hfill
    \subfigure[Unbounded Gaussian mean shift.]{
        \includegraphics[width=0.48\linewidth]{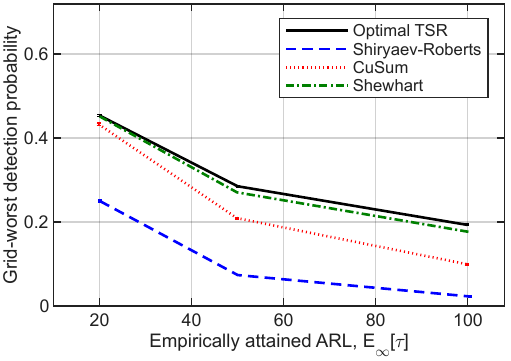}}
    \caption{Grid-worst conditional probability of detection within the admissible window versus empirically attained ARL for the TSR, Shiryaev--Roberts, CuSum, and Shewhart procedures, with $\xi=2$.}
    \label{fig:benchmark}
\end{figure}

We next assess whether tailoring the stopping rule to the finite-window probability-maximizing objective provides a material advantage over standard sequential detectors. Figure~\ref{fig:benchmark} compares the {average-optimal} TSR test with the classical Shiryaev--Roberts, CuSum, and Shewhart tests for $\xi=2$ in two Gaussian models. {Because exact finite-sample optimality is established for $\mathcal A_2$, the grid-worst quantity in Figure~\ref{fig:benchmark} is a complementary robustness comparison rather than the optimized criterion.} The left panel uses the variance-contraction model $F_0=\mathcal N(0,1)$ and $F_1=\mathcal N(0,0.2^2)$. This setting is particularly informative because the single-sample likelihood ratio is bounded, $\ell(x)=5\exp(-12x^2)\leq5$. Thus, no individual observation can provide arbitrarily large evidence for the change, making effective use of the two observations available within the admissible detection window important. The left panel shows a clear separation between the {average-optimal} TSR rule and the classical alternatives. When the ARL requirement is small, all procedures operate aggressively, and their detection probabilities are high. As the false-alarm constraint becomes more stringent, however, the TSR curve decays substantially more slowly and remains uniformly above the classical procedures over the displayed range. At moderate and large ARLs, the gain becomes pronounced: TSR retains a substantial probability of detection within the two-sample window, whereas classical Shiryaev--Roberts and CuSum eventually have essentially zero grid-worst detection probability.

The right panel provides a complementary comparison for the standard unbounded mean-shift model $F_0=\mathcal N(0,1)$ and $F_1=\mathcal N(1,1)$. At an attained ARL near $20$, CuSum and Shewhart have grid-worst success probabilities of approximately $0.432$ and $0.452$, respectively, compared with $0.455$ for TSR. Thus, the pronounced separation in the bounded-likelihood-ratio example does not persist in this aggressive operating regime. The advantage of TSR becomes clearer as the false-alarm requirement increases. Near ARL $100$, the corresponding probabilities are approximately $0.192$ for TSR, $0.176$ for Shewhart, $0.099$ for CuSum, and $0.023$ for Shiryaev--Roberts. Taken together, the two panels show that the benefit of finite-window tailoring is regime dependent: classical procedures can remain competitive when thresholds are low and individual observations can provide unbounded evidence, whereas the TSR rule retains a larger onset-robust success probability as the false-alarm constraint becomes more stringent.

For the right panel, all accumulation statistics are initialized at zero, and the grid-worst value is the minimum over onset locations $\gamma\in\{1,5,10,20,40\}$. The continuation function for TSR is computed from~\eqref{eq:bellman_xi2_comp} using $96$-point Gauss--Hermite quadrature and a value-iteration tolerance of $2\times10^{-10}$. For TSR, Shiryaev--Roberts, and CuSum, threshold calibration uses $30{,}000$ no-change sample paths. The attained ARLs are estimated from an independent set of $100{,}000$ paths, and the onset-specific success probabilities use $150{,}000$ paths per onset. The Shewhart threshold and success probability are evaluated analytically. The displayed error bars for the other procedures are pointwise $95\%$ Monte Carlo intervals. The accompanying simulation script and tabulated output record the thresholds, attained ARLs, standard errors, and fixed random seeds.

The latter behavior has a simple interpretation in this example. Since $\ell_t\leq5$, a CuSum statistic initialized at zero can accumulate at most $2\log5$ during the first two post-change observations. Likewise, a classical Shiryaev--Roberts statistic initialized at zero satisfies
\begin{align}
{\color{black}S_1^{\rm SR}}\leq 5,
\qquad
{\color{black}S_2^{\rm SR}}
=
(1+{\color{black}S_1^{\rm SR}})\ell_2
\leq 30\ .
\end{align}
Consequently, once the thresholds required to meet a stringent ARL constraint exceed these levels, the corresponding procedures cannot stop within the two-sample admissible window for an early change onset. This explains the sharp deterioration of their worst-case finite-window performance; it is a consequence of the mismatch between procedures designed to accumulate evidence over an unrestricted horizon and an objective that credits detection only within a short window. The comparison with Shewhart highlights the complementary effect. Shewhart does not suffer from accumulation of irrelevant old evidence, but it uses only the current observation and therefore cannot exploit the second sample available when $\xi=2$. TSR occupies the appropriate middle ground: its statistic retains precisely the likelihood-ratio terms associated with change onsets that can still lead to a successful decision,
\begin{align}
W_t^{(2)}
=
\ell_t(1+\ell_{t-1})\ ,
\end{align}
while its state-dependent continuation boundary accounts for the value of the retained evidence. Hence, although $W_t^{(\xi)}$ is a truncated Shiryaev--Roberts statistic, the finite-memory truncation and the associated stopping rule can lead to substantially different behavior from classical Shiryaev--Roberts and CuSum when the admissible detection window is short.
}

\FloatBarrier
\begin{figure}
    \centering
    \includegraphics[width=0.95\linewidth]{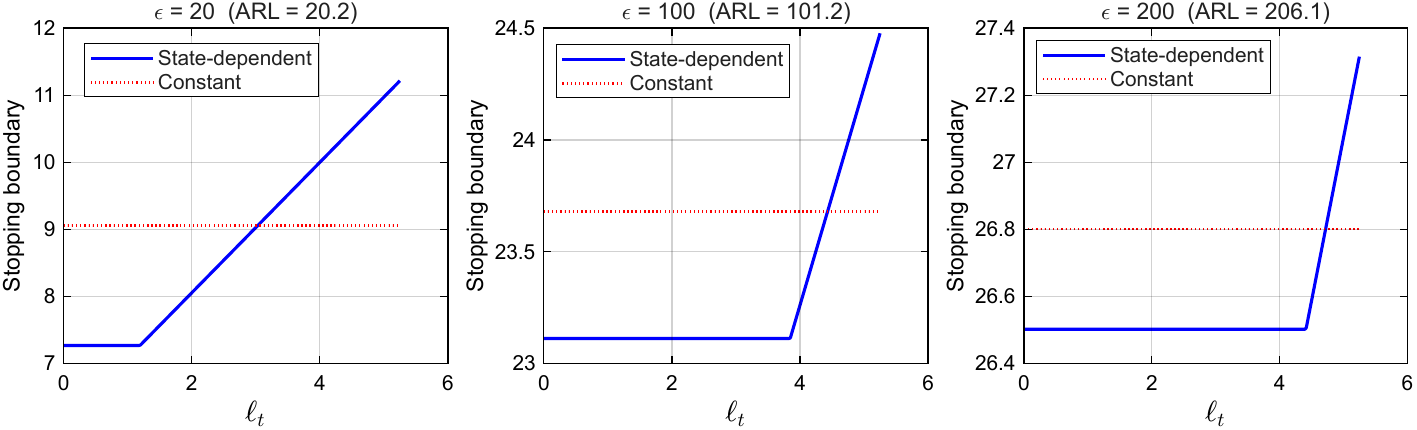}
    \caption{State-dependent continuation boundary $-\lambda+R(\ell_t)$ and a matched-ARL constant boundary.}
    \label{fig:boundary-comparison}
\end{figure}

{\color{black}
\paragraph{Structure of the state-dependent boundary.}
We next examine the role of the state-dependent continuation boundary itself for the variance-contraction setting used in the left panel of Figure~\ref{fig:benchmark}. For $\xi=2$, the state-dependent TSR rule compares the statistic
\begin{align}
W_t^{(2)}
=
\ell_t(1+\ell_{t-1})
\end{align}
with the continuation boundary
\begin{align}
h_\lambda(\ell_t)
\triangleq
-\lambda+R(\ell_t)\ ,
\end{align}
whereas the constant-boundary TSR rule replaces $h_\lambda(\ell_t)$ by a scalar threshold $A$. To isolate the effect of this state dependence, the constant threshold is calibrated to essentially the same ARL as the {average-optimal} rule. Figure~\ref{fig:boundary-comparison} shows the resulting boundaries for three false-alarm constraints. The empirically attained ARLs, $20.2$, $101.2$, and $206.1$, are close to the corresponding target values $\varepsilon=20$, $100$, and $200$. The figure shows that the {average-optimal} boundary is genuinely state dependent, but that the extent of this dependence changes with the operating regime. For $\varepsilon=20$, the variation is substantial over a broad range of $\ell_t$: the {average-optimal} boundary initially lies well below the matched constant threshold, crosses it at an intermediate likelihood-ratio value, and then rises appreciably above it. Thus, at this relatively aggressive operating point, replacing the continuation value by a single threshold removes a significant amount of state adaptation.

As the ARL constraint becomes more stringent, the state dependence becomes more localized. For $\varepsilon=100$ and $200$, the {average-optimal} boundary is nearly flat over most of the likelihood-ratio range and lies moderately below the corresponding constant threshold. It then increases sharply when $\ell_t$ approaches the upper end of its attainable range. Hence, a constant threshold can approximate the {average-optimal} boundary reasonably well over much of the state space in these regimes, but it does not reproduce the behavior of the {average-optimal} rule when the current likelihood ratio is large. In precisely those states, strong current evidence also has substantial continuation value because it becomes part of the next finite-memory TSR statistic. This observation provides a more nuanced answer to whether the state-dependent boundary is practically necessary. State dependence is not equally important at every state or every ARL: for stringent false-alarm constraints, much of the boundary is close to constant, while the principal adaptation is concentrated in high-evidence states. For less stringent constraints, the departure from a constant boundary is considerably more pronounced. Thus, the constant-boundary approximation can be useful in some operating regimes, while the exact {average-optimal} TSR construction retains a qualitatively different decision structure that becomes relevant in particular regions of the likelihood-ratio state space.

\paragraph{Finite-sample comparison with a constant boundary.}
We next isolate the practical value of state dependence by comparing the state-dependent TSR procedure with a constant-boundary TSR procedure using the same statistic and approximately matched false-alarm levels. Figure~\ref{fig:state_constant_exponential} uses an exponential model with nominal rate $1$, post-change rate $8$, and admissible-window length $\xi=3$. For each operating point, the scalar boundary of the constant-boundary TSR procedure is calibrated separately to approximately match the ARL attained by the state-dependent procedure; the horizontal coordinate reports the empirically attained ARL. {As in Figure~\ref{fig:benchmark}, the grid-worst probability reported here is a robustness diagnostic; the finite-sample exact optimality of the state-dependent rule is with respect to $\mathcal A_\xi$.}

\begin{figure}[t]
    \centering
    \includegraphics[width=0.5\linewidth]{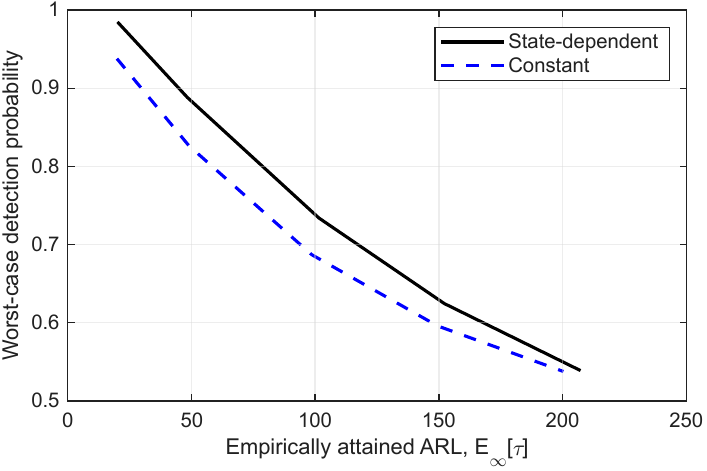}
    \caption{State-dependent and constant-boundary TSR procedures under approximately matched false-alarm levels for the exponential model with nominal rate $1$, post-change rate $8$, and admissible-window length $\xi=3$.}
    \label{fig:state_constant_exponential}
\end{figure}

Both procedures exhibit the expected decrease in finite-window detection probability as the ARL requirement becomes more stringent. The state-dependent procedure has the higher point estimate throughout the displayed range, with the clearest separation at moderate ARLs. Near ARL $50$, the worst-case detection probability increases from approximately $0.826$ for the constant-boundary rule to $0.888$ for the state-dependent rule, an absolute gain of about $0.062$ and a relative gain of approximately $7.5\%$. Near ARL $100$, the corresponding probabilities are approximately $0.687$ and $0.734$, yielding a relative improvement of about $6.8\%$. The advantage decreases as the false-alarm requirement becomes more stringent and is small at the displayed resolution near ARL $200$. Thus, the experiment shows that state dependence can provide a measurable finite-sample benefit without obtaining that gain through a higher false-alarm rate, while also illustrating that the benefit becomes small in more stringent operating regimes.

\paragraph{Information-Scale Transition.}
We next numerically illustrate the asymptotic mechanism underlying Theorem~\ref{thm:constant_tsr_asymptotic_optimality}. We consider the Gaussian mean-shift model $F_0=\mathcal N(0,1)$ and $F_1=\mathcal N(1,1)$ for which
\begin{align}
I
\triangleq
{\sf D}_{\rm KL}(F_1\Vert F_0)
=
\frac{1}{2}\ .
\end{align}
Recall the finite-window crossing probability introduced in Section~\ref{sec:constant_tsr_asymptotics},
\begin{align}
\bar p_{\xi}(A)
\triangleq
\mathbb P_1
\left(
\max_{1\leq k\leq\xi}S_k
\geq
\log A
\right).
\end{align}
For the constant-boundary TSR rule with $A=\varepsilon$, Lemma~\ref{lemma:constant_tsr_detection} shows that $\bar p_{\xi}(\varepsilon)$ provides a uniform lower bound on the per-opportunity Pollak- and Lorden-type success probabilities. Hence, the convergence
\begin{align}
\bar p_{\xi_\varepsilon}(\varepsilon)\to1
\end{align}
is precisely the mechanism used in Theorem~\ref{thm:constant_tsr_asymptotic_optimality} to establish asymptotic minimax optimality. Figure~\ref{fig:information_scale} illustrates this behavior. In the left panel, we select the admissible-window length according to
\begin{align}
\xi_\varepsilon
=
\left\lceil
c\,\frac{\log\varepsilon}{I}
\right\rceil
\end{align}
for several values of $c$. The resulting behavior changes sharply depending on whether $c$ lies below or above one. When $c<1$, the typical amount of post-change log-likelihood information available within the admissible window satisfies $\xi_\varepsilon I\approx c\log\varepsilon<\log\varepsilon$ and the crossing probability decreases as the false-alarm requirement becomes more stringent. At the critical scaling $c=1$, the available mean information and the required evidence are of the same order, and the crossing probability remains at an intermediate level over the range considered. In contrast, when $c>1$, the admissible window provides an excess amount of mean post-change information of approximately $\xi_\varepsilon I-\log\varepsilon\approx(c-1)\log\varepsilon$. Accordingly, $\bar p_{\xi_\varepsilon}(\varepsilon)$ increases with $\varepsilon$. The increase is more pronounced for larger values of $c$, since the admissible window extends farther beyond the information scale. This behavior is consistent with the sufficient condition of Theorem~\ref{thm:constant_tsr_asymptotic_optimality}: any fixed $c>1$ satisfies
\begin{align}
\limsup_{\varepsilon\to\infty}
\frac{\log\varepsilon}{\xi_\varepsilon}
<
I\ .
\end{align}
The convergence in the left panel is relatively gradual, especially when $c$ is only moderately larger than one. This is expected in the Gaussian model. Under $F_1$, the log-likelihood increments satisfy
\begin{align}
\log\frac{f_1(Z_r)}{f_0(Z_r)}
\sim
\mathcal N(I,1)\ ,
\end{align}
for the present choice of unit mean shift. With $\xi_\varepsilon\approx c\log\varepsilon/I$, the excess mean evidence above the crossing level, normalized by the standard deviation of $S_{\xi_\varepsilon}$, grows only on the order of
\begin{align}
(c-1)
\sqrt{\frac{\log\varepsilon}{2c}}\ .
\end{align}
Thus, even for large but finite $\varepsilon$, the crossing probability need not yet be close to one when $c$ is near the critical value. The right figure displays the same phenomenon using the normalized information ratio
\begin{align}
r
\triangleq
\frac{\xi I}{\log\varepsilon}.
\end{align}
The vertical dashed line at $r=1$ corresponds to the information scale
\begin{align}
\xi
=
\frac{\log\varepsilon}{I}.
\end{align}
As $\varepsilon$ increases, the curves separate increasingly on the two sides of this value. For $r<1$, the crossing probability becomes smaller with increasing $\varepsilon$, whereas for $r>1$ it becomes larger. Hence, the transition becomes progressively sharper around $r=1$, numerically illustrating the information scale identified by the asymptotic analysis. The curves with $r\leq1$ are included only as empirical contrasts, since Theorem~\ref{thm:constant_tsr_asymptotic_optimality} establishes the sufficient supercritical regime $r>1$ rather than a converse for $r\leq1$.

\begin{figure}
    \centering
    \includegraphics[width=0.95\linewidth]{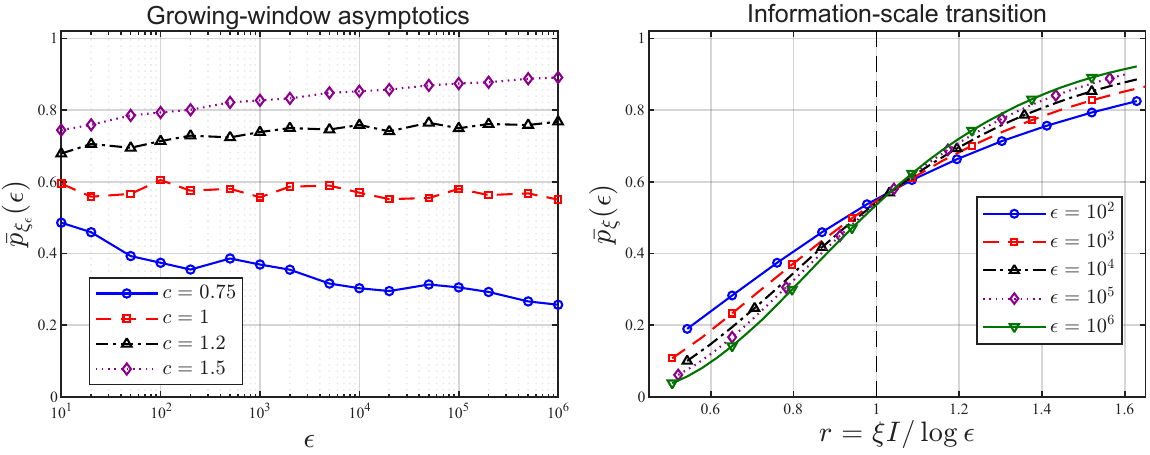}
    \caption{Information-scale transition.
    Left: finite-window crossing probability
    $\bar p_{\xi_\varepsilon}(\varepsilon)$ for
    $\xi_\varepsilon=\lceil c\log(\varepsilon)/I\rceil$.
    Curves with $c>1$ correspond to the regime covered by
    Theorem~\ref{thm:constant_tsr_asymptotic_optimality}; $c\leq1$ is included as an empirical comparison.
    Right: crossing probability versus the normalized information ratio
    $r=\xi I/\log\varepsilon$ for increasing values of $\varepsilon$.
    The dashed line marks the critical information scale $r=1$.}
    \label{fig:information_scale}
\end{figure}
}

\section{Concluding Remarks}\label{sec:conclusion}

We have studied the problem of \textcolor{black}{probability-maximizing sequential detection} in a stochastic process that may experience multiple, non-overlapping departures from a nominal distribution, each of finite duration, and followed by a return to the nominal regime. \textcolor{black}{The formulation also includes a single persistent change as a special case.} The objective has been to identify a change-point in real time, \textcolor{black}{within a prescribed admissible detection window after its occurrence}, while satisfying a false-alarm constraint. \textcolor{black}{The stopping time itself remains infinite-horizon; the detection window determines whether a stop is credited as successful for a given change onset.} To address this problem, we have adopted a probability-maximizing \textcolor{black}{formulation, considering both a survival-weighted average success criterion and Pollak- and Lorden-type minimax criteria}, under which performance is measured by the probability of stopping within the allowable detection window following a transient onset. This formulation differs fundamentally from classical quickest-change detection criteria based on expected delay and is particularly well suited to transient phenomena that may disappear before a detector has time to react. Our analysis has led to an exact characterization of the optimal stopping rule \textcolor{black}{under the survival-weighted average success criterion}. In particular, we have shown that the optimal decision rule is governed by a finite-memory statistic that aggregates likelihood-ratio evidence over the most recent observations within the \textcolor{black}{admissible detection window}. This yields a state-dependent optimal stopping structure and establishes the proposed \textcolor{black}{truncated Shiryaev--Roberts (TSR) test as the exactly optimal rule for the survival-weighted average problem}. The results have further revealed a natural and interpretable connection to classical sequential detection procedures: when the \textcolor{black}{admissible window} is $\xi=1$, the optimal rule reduces to the classical Shewhart test, corresponding to the requirement of detecting the change immediately upon its onset. \textcolor{black}{In this case, the Shewhart test is also exactly optimal under the Pollak- and Lorden-type minimax criteria.} The case $\xi=2$ provides the first nontrivial extension and illustrates explicitly how the optimal rule departs from a purely myopic likelihood-ratio threshold. More generally, the analysis has shown that enlarging the \textcolor{black}{admissible detection window} transforms the decision rule from an instantaneous test into a structured, \textcolor{black}{finite-memory, infinite-horizon stopping rule} that combines evidence from several recent samples. \textcolor{black}{For $\xi\geq 2$, we do not claim finite-sample exact minimax optimality of the state-dependent TSR rule; instead, the Pollak- and Lorden-type criteria provide complementary worst-case guarantees. We further showed that a simpler constant-boundary TSR procedure is asymptotically minimax optimal under both criteria when the admissible window grows beyond the information-accumulation scale $\log\epsilon/{\sf D}_{\rm KL}(F_1\Vert F_0)$.}

These findings suggest several directions for future work, including extensions to settings with unknown pre- and post-change distributions, overlapping or dependent transient periods, composite and nonparametric models, and decentralized or high-dimensional observation structures\textcolor{black}{, as well as scalable approximations for computing the state-dependent boundary for large admissible windows}.

\appendix

\section{\textcolor{black}{Proof of Lemma~\ref{lemma:average_identity}}}
\label{app:average_identity}
{\color{black}
For each $\gamma\geq 1$, the event $\{\tau\geq\gamma\}$ is
$\mathcal F_{\gamma-1}$-measurable, and
$\mathbb P^\xi_\gamma$ and $\mathbb P_\infty$ coincide on
$\mathcal F_{\gamma-1}$. Hence,
\begin{equation}
\mathbb P_\infty(\tau\geq\gamma)p^\xi_\gamma(\tau)
=
\mathbb P^\xi_\gamma(\gamma\leq\tau<\gamma+\xi)\ .
\end{equation}
Decomposing according to the stopping time gives
\begin{align}
\mathbb P^\xi_\gamma(\gamma\leq\tau<\gamma+\xi)
&=
\sum_{k=1}^{\xi}
\mathbb P^\xi_\gamma(\tau=\gamma+k-1)
\\
&=
\sum_{k=1}^{\xi}
\mathbb E_\infty\!\left[
\left(\prod_{j=0}^{k-1}\ell_{\gamma+j}\right)
\mathbf 1_{\{\tau=\gamma+k-1\}}
\right].
\end{align}
Summing over $\gamma\geq1$ and applying Tonelli's theorem,
\begin{align}
&\sum_{\gamma=1}^{\infty}
\mathbb P_\infty(\tau\geq\gamma)p^\xi_\gamma(\tau)
=
\mathbb E_\infty\!\left[
\sum_{k=1}^{\xi}
\sum_{\gamma=1}^{\infty}
\left(\prod_{j=0}^{k-1}\ell_{\gamma+j}\right)
\mathbf 1_{\{\tau=\gamma+k-1\}}
\right].
\end{align}
For each fixed $k$, setting $t=\gamma+k-1$ and noting that
$\gamma\geq1$ is equivalent to $t\geq k$, we obtain
\begin{align}
\sum_{\gamma=1}^{\infty}
\left(\prod_{j=0}^{k-1}\ell_{\gamma+j}\right)
\mathbf 1_{\{\tau=\gamma+k-1\}}
&=
\mathbf 1_{\{\tau\geq k\}}
\prod_{j=0}^{k-1}\ell_{\tau-j}\ .
\end{align}
Therefore,
\begin{align}
\sum_{\gamma=1}^{\infty}
\mathbb P_\infty(\tau\geq\gamma)p^\xi_\gamma(\tau)
&=
\mathbb E_\infty\!\left[
\sum_{k=1}^{\xi}
\mathbf 1_{\{\tau\geq k\}}
\prod_{j=0}^{k-1}\ell_{\tau-j}
\right]
\\
&=
\mathbb E_\infty\!\left[W^{(\xi)}_\tau\right]\ .
\end{align}
Finally, the tail-sum identity gives
\begin{align}
\sum_{\gamma=1}^{\infty}\mathbb P_\infty(\tau\geq\gamma)
=
\mathbb E_\infty[\tau]\ .
\end{align}
Dividing the preceding identity by $\mathbb E_\infty[\tau]$
and using the definitions of $w_\gamma(\tau)$ and
$\mathcal A_\xi(\tau)$ proves the result.
}

\section{Proof of Theorem~\ref{thm:upper_bound2}}
\label{app:thm:upper_bound2}

For each $\gamma\in\mathbb N$, let $\P\gamma$ denote the \emph{local} (window-$2$) post-change law under which $X_\gamma$ and $X_{\gamma+1}$ have density $f_1$ and all other $X_t$ have density $f_0$, and all $X_t$ are independent. Then for any $\gamma\in\mathbb N$ the measures $\P\gamma$ and $\P_\infty$ coincide on $\mathcal F_{\gamma-1}$. Furthermore, due to independence of samples over time, for any $\mathcal F_\gamma-$ and $\mathcal F_{\gamma+1}-$measurable events $A$ and $B$ we have
\begin{align}
\label{eq:A:rhs}
{\mathbb{P}_\gamma(A)=\mathbb{E}_\infty \big[\ell_\gamma\,\mathbf{1}_A\big]\ , \qquad 
\mathbb{P}_\gamma(B)=\mathbb{E}_\infty \big[\ell_\gamma\ell_{\gamma+1}\,\mathbf{1}_B\big]\ .}
\end{align}

\paragraph{Step 1: Bound the Pollak-like criterion for $s=1$.}
We first treat the setting {$s=1$}. {For $\xi=2$, the Pollak-like criterion \eqref{eq:Pollak_criterion} reduces to}
\begin{align}
\label{eq:bound1:d}
{\mathcal{L}_{\rm P}(\tau)=\inf_{\gamma\ge1} \mathbb{P}_\gamma(\tau\in\{\gamma,\gamma+1\}\mid \tau\ge\gamma)\ .}
\end{align}
{Fix $\gamma\ge 1$. Dropping the infimum in \eqref{eq:bound1:d} yields}
\begin{align}
{\mathcal{L}_{\rm P}(\tau)\leq \mathbb{P}_\gamma(\tau\in\{\gamma,\gamma+1\}\mid \tau\ge\gamma)\ .}
\end{align}
{Multiplying both sides by $\mathbb{P}_\gamma(\tau\ge\gamma)$ and using the definition of conditional probability, we obtain}
\begin{align}
\label{eq:pollak_step1}
{\mathbb{P}_\gamma(\tau\ge\gamma)\,\mathcal{L}_{\rm P}(\tau)\le \mathbb{P}_\gamma(\tau=\gamma)+\mathbb{P}_\gamma(\tau=\gamma+1)\ .}
\end{align}
{Since $\{\tau\ge\gamma\}\in\F_{\gamma-1}$ (because $\tau$ is a stopping time) and $\mathbb{P}_\gamma$ and $\mathbb{P}_\infty$ coincide on $\F_{\gamma-1}$,}
\begin{align}
{\mathbb{P}_\gamma(\tau\ge\gamma)=\mathbb{P}_\infty(\tau\ge\gamma)\ .}
\end{align}
{Furthermore, $\{\tau=\gamma\}\in\F_\gamma$ and $\{\tau=\gamma+1\}\in\F_{\gamma+1}$, so \eqref{eq:A:rhs} gives}
\begin{align}
{\mathbb{P}_\gamma(\tau=\gamma)=\mathbb{E}_\infty \big[\ell_\gamma\,\mathbf{1}\{\tau=\gamma\}\big],\qquad
\mathbb{P}_\gamma(\tau=\gamma+1)=\mathbb{E}_\infty \big[\ell_\gamma\ell_{\gamma+1}\,\mathbf{1}\{\tau=\gamma+1\}\big]\ .}
\end{align}
{Substituting these identities into \eqref{eq:pollak_step1} yields, for every $\gamma\ge 1$,}
\begin{align}\label{eq:pollak_pointwise_bound}
{\mathbb{P}_\infty(\tau\ge\gamma)\,\mathcal{L}_{\rm P}(\tau)\le 
\mathbb{E}_\infty \Big[\ell_\gamma\,\mathbf{1}\{\tau=\gamma\}+\ell_\gamma\ell_{\gamma+1}\,\mathbf{1}\{\tau=\gamma+1\}\Big]\ .}
\end{align}
{Summing \eqref{eq:pollak_pointwise_bound} over $\gamma\ge 1$ and using the identity $\tau=\sum_{\gamma\ge 1}\mathbf{1}\{\tau\ge\gamma\}$ (hence $\mathbb{E}_\infty[\tau]=\sum_{\gamma\ge 1}\mathbb{P}_\infty(\tau\ge\gamma)$), we obtain}
\begin{align}
{\mathbb{E}_\infty[\tau]\ \mathcal{L}_{\rm P}(\tau)
\le \mathbb{E}_\infty \big[\ell_\tau+\ell_{\tau-1}\ell_\tau\big]\ .}
\end{align}
{Therefore, for $s=1$,}
\begin{align}
{\mathcal{L}_{\rm P}(\tau)\le \frac{\mathbb{E}_\infty \big[\ell_\tau+\ell_{\tau-1}\ell_\tau\big]}{\mathbb{E}_\infty[\tau]}\ ,}
\end{align}
{where we use the convention $\ell_0=0$ so that $\ell_{\tau-1}\ell_\tau=0$ when $\tau=1$.}

\paragraph{Step 2: Bound the Lorden-like criterion for $s=1$.}
{For $\xi=2$ and $s=1$, the Lorden-like criterion \eqref{eq:Lorden_criterion} reduces to}
\begin{align}
{\mathcal{L}_{\rm L}(\tau)=\inf_{\gamma\ge1}\operatorname*{ess\,inf}_{\F_{\gamma-1}}
\mathbb{P}_\gamma(\tau\in\{\gamma,\gamma+1\}\mid \F_{\gamma-1},\tau\ge\gamma)\ .}
\end{align}
Fix $\gamma\ge1$ and define the $\F_{\gamma-1}$--measurable random variable
\begin{align}
{\pi_\gamma \dff \mathbb{P}_\gamma(\tau\in\{\gamma,\gamma+1\}\mid \F_{\gamma-1},\tau\ge\gamma)\ .}
\end{align}
We adopt the convention that $\pi_\gamma=0$ on $\{\tau<\gamma\}$.
Let
\begin{align}
a_\gamma \triangleq \operatorname*{ess\,inf}_{\F_{\gamma-1}}\pi_\gamma\ .
\end{align}
By definition of the essential infimum, $a_\gamma\le \pi_\gamma$ $\mathbb{P}_\gamma$--a.s., and hence $\mathbb{P}_\infty$--a.s.
(because $\mathbb{P}_\gamma$ and $\mathbb{P}_\infty$ coincide on $\F_{\gamma-1}$).
Furthermore, by definition of $\mathcal{L}_{\rm L}(\tau)$ as an infimum over $\gamma$, we have $\mathcal{L}_{\rm L}(\tau)\le a_\gamma$
for every $\gamma$, hence
\begin{equation}
{\mathbf{1}\{\tau\ge\gamma\}\,\mathcal{L}_{\rm L}(\tau)\le \mathbf{1}\{\tau\ge\gamma\}\,\pi_\gamma\qquad \mathbb{P}_\infty\text{-a.s.}}
\label{eq:A:Lordenpointwise}
\end{equation}
Next, take the expectation of \eqref{eq:A:Lordenpointwise}. Since $\mathbf{1}\{\tau\ge\gamma\}$ is
$\F_{\gamma-1}$--measurable and $\mathbb{P}_\gamma$ and $\mathbb{P}_\infty$ coincide on $\F_{\gamma-1}$,
\begin{align}
{\mathbb{E}_\infty \big[\mathbf{1}\{\tau\ge\gamma\}\pi_\gamma\big]
=\mathbb{E}_\gamma \big[\mathbf{1}\{\tau\ge\gamma\}\pi_\gamma\big]
=\mathbb{P}_\gamma(\tau\in\{\gamma,\gamma+1\})\ ,}
\end{align}
where the last equality is the defining property of conditional probability.
Using \eqref{eq:A:rhs} we obtain
\begin{align}
\label{eq:A:Lordenpointwise2}
{\mathbb{E}_\infty \big[\mathbf{1}\{\tau\ge\gamma\}\pi_\gamma\big]
=\mathbb{E}_\infty \Big[\ell_\gamma\,\mathbf{1}\{\tau=\gamma\}+\ell_\gamma\ell_{\gamma+1}\,\mathbf{1}\{\tau=\gamma+1\}\Big]\ .}
\end{align}
Therefore \eqref{eq:A:Lordenpointwise}--\eqref{eq:A:Lordenpointwise2} imply that for every $\gamma\ge1$,
\begin{align}
{\mathbb{P}_\infty(\tau\ge\gamma)\,\mathcal{L}_{\rm L}(\tau)\le
\mathbb{E}_\infty \Big[\ell_\gamma\,\mathbf{1}\{\tau=\gamma\}+\ell_\gamma\ell_{\gamma+1}\,\mathbf{1}\{\tau=\gamma+1\}\Big]\ .}
\end{align}
{Summing over $\gamma\ge 1$ yields}
\begin{align}
{\mathbb{E}_\infty[\tau]\ \mathcal{L}_{\rm L}(\tau)\le \mathbb{E}_\infty \big[\ell_\tau+\ell_{\tau-1}\ell_\tau\big]\ ,}
\end{align}
{and hence $\mathcal{L}_{\rm L}(\tau)\le \mathbb{E}_\infty[\ell_\tau+\ell_{\tau-1}\ell_\tau]/\mathbb{E}_\infty[\tau]$.}

\paragraph{Step 3: Extension to $s\ge1$.}
For $\xi=2$, the criteria \eqref{eq:Pollak_criterion} and \eqref{eq:Lorden_criterion} are sums over $i\in\{1,\dots,s\}$ of window-$2$ success terms.
{Applying the preceding $s=1$ bounds to each summand (corresponding to each onset time $\gamma_i$) and noting that the right-hand side depends only on $\tau$, we obtain}
\begin{align}
{\mathcal{L}_{\rm P}(\tau)\le s\cdot \frac{\mathbb{E}_\infty \big[\ell_\tau+\ell_{\tau-1}\ell_\tau\big]}{\mathbb{E}_\infty[\tau]}\ .}
\end{align}
{Moreover, since $\mathcal{L}_{\rm L}(\tau)\le \mathcal{L}_{\rm P}(\tau)$ by definition (cf. the discussion following \eqref{eq:Lorden_criterion}), the same bound applies to $\mathcal{L}_{\rm L}(\tau)$.}
This completes the proof.

\section{Proof of Lemma~\ref{lemma:eq}}
\label{app:lemma:eq}

{\color{black}Throughout Appendices~\ref{app:lemma:eq}--\ref{app:lemma:suff_infinite}, we use the shorthand $W_t=W_t^{(2)}$.} 
{Let $\nu\in\mathcal{T}$ be a feasible stopping time, i.e., $\mathbb{E}_\infty[\nu]\ge \varepsilon$ and $\mathbb{E}_\infty[\nu]<\infty$.}
{Recall the likelihood ratio $\ell_t=f_1(X_t)/f_0(X_t)$ and the convention $\ell_0=0$. For convenience, define}
\begin{align}
{W_t \dff \ell_t+\ell_{t-1}\ell_t\ ,\qquad t\in\mathbb{N}\ ,\qquad W_0\dff 0\ .}
\end{align}
{Let $\pi_0\dff \mathbb{P}_\infty(\nu=0)$. Then}
\begin{align}
{\mathbb{E}_\infty[W_\nu]=(1-\pi_0)\,\mathbb{E}_\infty[W_\nu\mid \nu>0]\ ,\qquad
\mathbb{E}_\infty[\nu]=(1-\pi_0)\,\mathbb{E}_\infty[\nu\mid \nu>0]\ ,}
\end{align}
{and hence}
\begin{align}
{\frac{\mathbb{E}_\infty[W_\nu]}{\mathbb{E}_\infty[\nu]}
=
\frac{\mathbb{E}_\infty[W_\nu\mid \nu>0]}{\mathbb{E}_\infty[\nu\mid \nu>0]}\ .}
\end{align}
{If $\mathbb{E}_\infty[\nu]=\varepsilon$, we simply set $\nu'=\nu$ and we are done.}
{Otherwise, since $\nu$ is feasible and $\mathbb{E}_\infty[\nu]<\infty$, we have $\mathbb{E}_\infty[\nu]>\varepsilon$. Define}
\begin{align}
{p \dff 1-\frac{\varepsilon}{\mathbb{E}_\infty[\nu]}\in(0,1)\ .}
\end{align}
{Let $B\sim{\rm Bernoulli}(p)$ be independent of $\{X_t\}_{t\in\mathbb{N}}$ (equivalently, measurable at time $0$ under an external randomization). Construct the randomized stopping time}
\begin{align}
{\nu' \dff 
\begin{cases}
0, & B=1,\\
\nu, & B=0.
\end{cases}}
\end{align}
{Then $\nu'$ is a (randomized) stopping time with respect to the filtration augmented by $\sigma(B)$, and it satisfies the following two properties.}

\begin{enumerate}
\item {ARL constraint with equality:} Since $B$ is independent of $\nu$ and $\nu'\equiv \nu$ on $\{B=0\}$,
\begin{align}\label{eq:equal3}
{\mathbb{E}_\infty[\nu']
=(1-p)\,\mathbb{E}_\infty[\nu]
=\varepsilon\ .}
\end{align}

\item {Ratio preservation:} Using $W_0=0$ and independence of $B$,
\begin{align}
{\mathbb{E}_\infty[W_{\nu'}]
=(1-p)\,\mathbb{E}_\infty[W_\nu]\ ,}
\end{align}
{and therefore}
\begin{align}\label{eq:ratio2}
{\frac{\mathbb{E}_\infty[W_{\nu'}]}{\mathbb{E}_\infty[\nu']}
=\frac{(1-p)\mathbb{E}_\infty[W_\nu]}{(1-p)\mathbb{E}_\infty[\nu]}
=\frac{\mathbb{E}_\infty[W_\nu]}{\mathbb{E}_\infty[\nu]}\ .}
\end{align}
{Recalling $W_\tau=\ell_\tau+\ell_{\tau-1}\ell_\tau$, \eqref{eq:ratio2} is exactly \eqref{eq:ratio}.}
\end{enumerate}
The properties \eqref{eq:equal3}--\eqref{eq:ratio2} establish that $\nu'$ is feasible with $\mathbb{E}_\infty[\nu']=\varepsilon$ and achieves the same performance ratio as $\nu$, completing the proof. Note that augmented filtration $\sigma(B)$ does not violate the sequential nature of the test, since $B$ is determined at time zero before any observations are processed.

\section{Proof of Lemma~\ref{lemma:lag}}
\label{proof:lemma:lag}

{\color{black}For notational brevity, write $\widetilde{\mathcal Q}=\widetilde{\mathcal Q}_2$ and $\mathcal V=\mathcal V_2$.} Recall the {ARL-equality} problem {(cf. \eqref{eq:qbar})}
\begin{align}
{\widetilde{\mathcal{Q}}(\varepsilon)}
\;=\;\sup_{\tau:\,{\mathbb{E}_\infty[\tau]}=\varepsilon}
\frac{{\mathbb{E}_\infty[W_\tau]}}{{\mathbb{E}_\infty[\tau]}}
\;=\;\frac{1}{\varepsilon}\sup_{\tau:\,{\mathbb{E}_\infty[\tau]}=\varepsilon}{\mathbb{E}_\infty[W_\tau]}\ ,
\end{align}
{where $W_t\dff \ell_t+\ell_{t-1}\ell_t$ for $t\in\mathbb{N}$ and $W_0\dff 0$ (under the convention $\ell_0=0$).}
{Also recall the unconstrained Lagrangian value (cf. \eqref{eq:r})}
\begin{align}
{\mathcal{V}(\lambda)\;\dff\;\sup_{\tau\in\mathcal{T}} \mathbb{E}_\infty \big[W_\tau-\lambda \tau\big]\ ,} \qquad \lambda>0\ ,
\end{align}
where the supremum is over all {$\tau\in\mathcal{T}$} with {$\mathbb{P}_\infty(\tau<\infty)=1$}. For given $\lambda$ and $\varepsilon$ define the dual function
\begin{equation}
{\phi_\varepsilon(\lambda)\;\dff\;\lambda+\frac{\mathcal{V}(\lambda)}{\varepsilon}\ .}
\label{eq:def-phi-eps}
\end{equation}
To prove this lemma, we show two properties.
\begin{enumerate}
\item[(i)] First, we show that for a fixed $\varepsilon\ge 1$, the function $\phi_\varepsilon$ is convex, lower semicontinuous, and admits at least one minimizer $\lambda_\varepsilon>0$.
\item[(ii)] Secondly, by defining $\rho\triangleq\phi_\varepsilon(\lambda_\varepsilon)$, we show that there exists a (possibly randomized) stopping time $\tau^\star$ such that
\begin{align}
{\mathbb{E}_\infty[\tau^\star]=\varepsilon\ ,}
\qquad
{\frac{\mathbb{E}_\infty[W_{\tau^\star}]}{\mathbb{E}_\infty[\tau^\star]}=\rho\ .}
\end{align}
\end{enumerate}
\paragraph{Step 1: Existence of a minimizer.}
For each stopping time $\tau$, the map $\lambda\mapsto {\mathbb{E}_\infty[W_\tau-\lambda \tau]}$ is affine.
Hence ${\mathcal{V}(\lambda)}$ is the pointwise supremum of affine functions, so {$\mathcal{V}$} is convex and nonincreasing in $\lambda$.
Therefore, $\phi_\varepsilon$ is convex and lower semicontinuous. Furthermore, {since the trivial stopping time $\tau\equiv 0$ belongs to $\mathcal{T}$ and yields $W_0=0$, we have $\mathcal{V}(\lambda)\ge 0$ for all $\lambda>0$.} Hence,
\begin{align}
{\phi_\varepsilon(\lambda)\;\ge\;\lambda\;\xrightarrow[]{\lambda\to\infty}\;+\infty\ .}
\end{align}
Thus $\phi_\varepsilon$ attains a finite minimum at some $\lambda_\varepsilon>0$.

\medskip
\paragraph{Step 2: Weak duality and attainment by (randomized) selection.}
Let $\tau$ satisfy {$\mathbb{E}_\infty[\tau]=\varepsilon$}. Then for any $\lambda>0$,
\begin{align}
{\frac{\mathbb{E}_\infty[W_\tau]}{\varepsilon}}
=\lambda+\frac{{\mathbb{E}_\infty[W_\tau-\lambda\tau]}}{\varepsilon}
\le \lambda+\frac{{\mathcal{V}(\lambda)}}{\varepsilon}
=\phi_\varepsilon(\lambda)\ .
\end{align}
Taking the supremum over such $\tau$ and then the infimum over $\lambda>0$ gives
\begin{equation}
{\sup_{\tau:\,\mathbb{E}_\infty[\tau]=\varepsilon}\frac{\mathbb{E}_\infty[W_\tau]}{\varepsilon}
\;\le\;\inf_{\lambda>0}\phi_\varepsilon(\lambda)
\;=\; \rho\ .}
\label{eq:weak-duality}
\end{equation}
Now fix a minimizer $\lambda_\varepsilon$.
{Let $\mathcal{A}(\lambda_\varepsilon)$ denote the set of optimizers of $\mathcal{V}(\lambda_\varepsilon)$, i.e.,}
\begin{align}
{\mathcal A}(\lambda_\varepsilon)\;\dff\;\Bigl\{\tau\in\mathcal{T}:\ \mathbb{E}_\infty[W_\tau-\lambda_\varepsilon\tau]=\mathcal{V}(\lambda_\varepsilon)\Bigr\}\ .
\end{align}
Standard optimal stopping arguments imply $\mathcal{A}(\lambda_\varepsilon)\neq\emptyset$. In particular, $\mathcal{V}(\lambda_\varepsilon)$ is attained by an optimal stopping time.
For any $\tau\in{\mathcal A}(\lambda_\varepsilon)$, define $D_\tau\triangleq {\mathbb{E}_\infty[\tau]}$ and $N_\tau\triangleq {\mathbb{E}_\infty[W_\tau]}$.
    Then $N_\tau-\lambda_\varepsilon D_\tau={\mathcal{V}(\lambda_\varepsilon)}$ for all $\tau\in{\mathcal A}(\lambda_\varepsilon)$. Accordingly, consider the affine functions
\begin{align}
f_\tau(\lambda)\;\triangleq\;\lambda+\frac{N_\tau-\lambda D_\tau}{\varepsilon}
=\Bigl(1-\frac{D_\tau}{\varepsilon}\Bigr)\lambda+\frac{N_\tau}{\varepsilon}\ .
\end{align}
{Then $\phi_\varepsilon(\lambda)=\sup_{\tau\in\mathcal{T}} f_\tau(\lambda)$, and the active set at $\lambda_\varepsilon$ coincides with $\mathcal{A}(\lambda_\varepsilon)$ (since $f_\tau(\lambda_\varepsilon)=\phi_\varepsilon(\lambda_\varepsilon)$ if and only if $N_\tau-\lambda_\varepsilon D_\tau=\mathcal{V}(\lambda_\varepsilon)$).} 
Because $\phi_\varepsilon$ is convex and minimized at $\lambda_\varepsilon$, {$0$ belongs to the subdifferential $\partial \phi_\varepsilon(\lambda_\varepsilon)$, which is the convex hull of the active slopes}.
Equivalently, letting
\begin{align}
m_\tau\;\triangleq\;1-\frac{D_\tau}{\varepsilon}\ ,
\qquad \tau\in{\mathcal A}(\lambda_\varepsilon)\ ,
\end{align}
we must have that the active slopes straddle $0$; i.e., letting $m_{\min}\triangleq \inf_{\tau\in{\mathcal A}(\lambda_\varepsilon)} m_\tau$ and
$m_{\max}\triangleq \sup_{\tau\in{\mathcal A}(\lambda_\varepsilon)} m_\tau$, we have $m_{\min}\le 0\le m_{\max}$. 
If there exists $\bar\tau\in{\mathcal A}(\lambda_\varepsilon)$ with $m_{\bar\tau}=0$, then $D_{\bar\tau}=\varepsilon$ and we may take
$\tau^\star=\bar\tau$. Otherwise choose $\tau_1,\tau_2\in{\mathcal A}(\lambda_\varepsilon)$ with $m_{\tau_1}>0$ and $m_{\tau_2}<0$,
and pick $\alpha\in(0,1)$ such that
\begin{align}
\alpha m_{\tau_1}+(1-\alpha)m_{\tau_2}=0
\quad\Longleftrightarrow\quad
\alpha D_{\tau_1}+(1-\alpha)D_{\tau_2}=\varepsilon\ .
\end{align}
Define the time-$0$ randomized stopping time
\begin{align}
\tau^\star=\begin{cases}
\tau_1, & \text{with probability }\alpha,\\
\tau_2, & \text{with probability }1-\alpha,
\end{cases}\ .
\end{align}
Then {$\mathbb{E}_\infty[\tau^\star]=\varepsilon$} and
\begin{align}
{\mathbb{E}_\infty[W_{\tau^\star}-\lambda_\varepsilon\tau^\star]}
=\alpha(N_{\tau_1}-\lambda_\varepsilon D_{\tau_1})+(1-\alpha)(N_{\tau_2}-\lambda_\varepsilon D_{\tau_2})
={\mathcal{V}(\lambda_\varepsilon)}\ .
\end{align}
Consequently,
\begin{align}
{\frac{\mathbb{E}_\infty[W_{\tau^\star}]}{\varepsilon}}
=\lambda_\varepsilon+\frac{{\mathbb{E}_\infty[W_{\tau^\star}-\lambda_\varepsilon\tau^\star]}}{\varepsilon}
=\lambda_\varepsilon+\frac{{\mathcal{V}(\lambda_\varepsilon)}}{\varepsilon}
=\phi_\varepsilon(\lambda_\varepsilon)
=\rho\ .
\end{align}
Combining with \eqref{eq:weak-duality} shows that $\tau^\star$ achieves the primal optimum value $\rho$,
{and therefore solves $\widetilde{\mathcal{Q}}(\varepsilon)$.}
\section{Proof of Lemma~\ref{lemma:bellman_identity}}
\label{app:lemma:bellman}

All conditional expectations are under $\mathbb{P}_\infty$. {Recall $W_t\dff \ell_t+\ell_{t-1}\ell_t$ (with $\ell_0=0$).}
{Note that $W_t$ is integrable since $\mathbb{E}_\infty[\ell_t]=1$ and, by independence under $\mathbb{P}_\infty$,}
$\mathbb{E}_\infty[\ell_{t-1}\ell_t]=\mathbb{E}_\infty[\ell_{t-1}]\,\mathbb{E}_\infty[\ell_t]\in\{0,1\}$,
resulting in $\mathbb{E}_\infty[W_t]=\mathbb{E}_\infty[\ell_t]+\mathbb{E}_\infty[\ell_{t-1}\ell_t]\leq 2<\infty$.

\paragraph{(a) Lower bound.} First, choosing $\tau\equiv t$ in \eqref{eq:inf_JR_def} indicates that $J_t\ge {W_t}$ a.s.
Next, fix {$\eta>0$}.
By the definition of $J_{t+1}(\F_{t+1})$ there exists $\tau^{{\eta}}\in\mathcal T_{t+1}$ such that
\begin{align}
J_{t+1}(\F_{t+1})\ \le\ {\mathbb{E}_\infty} \big[{W_{\tau^{\eta}}}-\lambda\big(\tau^{{\eta}}-(t{+}1)\big)\,\big|\,\F_{t+1}\big]\ +\ {\eta}
\quad\text{a.s.}
\label{eq:inf_eps_opt}
\end{align}
Using $\tau^{{\eta}}$ as a feasible action at time $t{+}1$ and then invoking the tower property we obtain
\begin{align}
J_t(\F_t)\ &\ge\ {\mathbb{E}_\infty} \big[{W_{\tau^{\eta}}}-\lambda(\tau^{{\eta}}-t)\,\big|\,\F_t\big]\\
&=\ -\lambda\ +\ {\mathbb{E}_\infty} \Big[\ {\mathbb{E}_\infty} \big[{W_{\tau^{\eta}}}-\lambda\big(\tau^{{\eta}}-(t{+}1)\big)\,\big|\,\F_{t+1}\big]\ \Big|\ \F_t\Big]\\
&\ge\ -\lambda\ +\ {\mathbb{E}_\infty} \big[J_{t+1}(\F_{t+1})\,\big|\,\F_t\big]\ -\ {\eta}
\ \\
& =\ -\lambda+R_t-{\eta}\ .
\label{eq:inf_J_geq_cont}
\end{align}
Letting {$\eta\to 0$} yields $J_t\ge -\lambda+R_t$ a.s. Combining with $J_t\ge {W_t}$ we have
\begin{align}
J_t\ \ge\ \max\{{W_t},\,-\lambda+R_t\}\quad\text{a.s.}
\label{eq:inf_bellman_geq}
\end{align}

\paragraph{(b) Upper bound.} Take an arbitrary $\tau\in\mathcal T_t$ and consider the following decomposition according to the events $\{\tau=t\}$ and $\{\tau\ge t{+}1\}$:
\begin{align}
{\mathbb{E}_\infty} \big[{W_\tau}-\lambda(\tau-t)\,\big|\,\F_t\big]
&=\ {\mathbb{E}_\infty} \big[\mathbf{1}_{\{\tau=t\}}\,{W_t}\ +\ \mathbf{1}_{\{\tau\ge t+1\}}\,[{W_\tau}-\lambda-\lambda(\tau-(t{+}1))]\,\big|\,\F_t\big]\\
&\le\ \mathbf{1}_{\{\tau=t\}}\,{W_t}\ +\ \mathbf{1}_{\{\tau\ge t+1\}}\Big(-\lambda+{\mathbb{E}_\infty} \big[J_{t+1}(\F_{t+1})\,\big|\,\F_t\big]\Big)
\nonumber\\
&=\ \mathbf{1}_{\{\tau=t\}}\,{W_t}\ +\ \mathbf{1}_{\{\tau\ge t+1\}}(-\lambda+R_t)\\
& \le\ \max\{{W_t},\,-\lambda+R_t\}\quad\text{a.s.}
\label{eq:inf_bellman_leq_one_tau}
\end{align}
where we used that $\{\tau=t\},\{\tau\ge t+1\}\in\F_t$, the tower property, the definition of $J_{t+1}$ in~\eqref{eq:inf_JR_def}, and the definition of $R_t$ in~\eqref{eq:bell-Rt}.
Taking the essential supremum over $\tau\in\mathcal T_t$ on the left-hand side yields
\begin{align}
J_t\ \le\ \max\{{W_t},\,-\lambda+R_t\}\quad\text{a.s.}
\label{eq:inf_bellman_leq}
\end{align}
Combining \eqref{eq:inf_bellman_geq} and \eqref{eq:inf_bellman_leq} proves the Bellman identity
\begin{align}
J_t(\F_t)\ =\ \max\big\{{W_t},\ -\lambda+R_t\big\}\quad\text{a.s. for all }t\ge1.
\label{eq:inf_Bellman_identity}
\end{align}

\section{Proof of Lemma~\ref{lemma:5}}
\label{app:lemma:5}

{For each $t\ge1$, we had defined the admissible set in~\ref{eq:admissible} as}
\begin{align}
{\mathcal{T}_t \dff \{\tau\in\mathcal{T}:\ \tau \ge t \quad \text{ a.s.}\}\ ,}
\end{align}
{Recall $W_t\dff \ell_t+\ell_{t-1}\ell_t$ for $t\in\mathbb{N}$ (with $\ell_0=0$).}
Define the infinite-horizon value function
\begin{align}
{J_t \dff \operatorname*{ess\,sup}_{\tau\in\mathcal{T}_t} \mathbb{E}_\infty \big[W_\tau-\lambda(\tau-t)\,\big|\, \F_t\big]\ ,}
\qquad
{R_t \dff \mathbb{E}_\infty \big[J_{t+1}\,\big|\, \F_t\big]\ .}
\end{align}
Lemma~\ref{lemma:bellman_identity} then reads
\begin{equation}
J_t=\max\{W_t,\,-\lambda+R_t\}\qquad\text{a.s. for all }t\ge 1\ .
\label{eq:E:bellman}
\end{equation}
Define the candidate stopping time
\begin{equation}
{\tau_\lambda^\ast \dff \inf\{t\ge1:\ W_t\ge -\lambda+R_t\}\ .}
\label{eq:E:taustar}
\end{equation}
{Because $\{W_t\ge -\lambda+R_t\}\in\F_t$ and $\tau_\lambda^\ast$ is the first entry time into an $\{\F_t\}$--measurable set, $\tau_\lambda^\ast$ is an $\{\F_t\}$--stopping time.}

\medskip
\paragraph{Step 1: Reduce the problem to a standard optimal stopping form.}
Define the reward process
\begin{align}
{Y_t \dff W_t-\lambda t\ ,\qquad t\ge1\ .}
\end{align}
For each $t\ge1$ define its Snell envelope
\begin{align}
{V_t \dff \operatorname*{ess\,sup}_{\tau\in\mathcal T_t} \mathbb{E}_\infty \big[Y_\tau\,\big|\,\F_t\big]\ .}
\end{align}
Note that
\begin{align}
J_t
& =\operatorname*{ess\,sup}_{\tau\in\mathcal T_t} \mathbb{E}_\infty \big[W_\tau-\lambda(\tau-t)\,\big|\,\F_t\big] \\
& =\lambda t + \operatorname*{ess\,sup}_{\tau\in\mathcal T_t}\mathbb{E}_\infty \big[W_\tau-\lambda\tau\,\big|\,\F_t\big] \\
& =\lambda t + V_t\ .
\end{align}
Hence, the stopping condition {$J_t=W_t$} is equivalent to {$V_t=Y_t$}. By \eqref{eq:E:bellman},
\begin{align}
\{J_t=W_t\}=\{W_t\ge -\lambda+R_t\}\ ,
\end{align}
so the stopping time in \eqref{eq:E:taustar} can equivalently be written as
\begin{align}
{\tau_\lambda^\ast=\inf\{t\ge1:\ V_t=Y_t\}\ .}
\end{align}

\medskip
\paragraph{Step 2: Verify the hypotheses for the Snell-envelope optimality theorem.}
Since {$W_t\ge0$}, and by using the notation $(X)^+=\max(0,X)$, we have {$Y_t^+=(W_t-\lambda t)^+\le W_t$}.
{Under $\mathbb{P}_\infty$, $\mathbb{E}_\infty[W_t]=\mathbb{E}_\infty[\ell_t]+\mathbb{E}_\infty[\ell_{t-1}\ell_t]\leq 2$ (by independence), hence}
\begin{align}
{\sup_{t\ge1}\mathbb{E}_\infty[Y_t^+]\le \sup_{t\ge1}\mathbb{E}_\infty[W_t]\leq 2<\infty\ .}
\end{align}
This standard integrability condition ensures that the Snell envelope is well defined (of class $\mathcal{D}$), so optional sampling applies to the stopped envelope.

\medskip
\paragraph{Step 3: Optimality.}
By the discrete-time Snell envelope theorem (optimal stopping in discrete time),
$\{V_t\}$ is the smallest $\{\F_t\}$--supermartingale that dominates $\{Y_t\}$, and the stopping time
{$\tau_\lambda^\ast=\inf\{t\ge1:\ V_t=Y_t\}$} is optimal:
\begin{equation}
{\mathbb{E}_\infty[Y_{\tau_\lambda^\ast}]
=\sup_{\tau\in\mathcal T_1}\mathbb{E}_\infty[Y_\tau]
=\mathbb{E}_\infty[V_1]\ .}
\label{eq:E:snellopt}
\end{equation}
Equivalently,
\begin{align}
{\sup_{\tau\in\mathcal T_1}\mathbb{E}_\infty[W_\tau-\lambda\tau]
=\mathbb{E}_\infty[W_{\tau_\lambda^\ast}-\lambda\tau_\lambda^\ast]\ .}
\end{align}
{This shows that $\tau_\lambda^\ast$ attains the supremum in $\mathcal{V}(\lambda)$ (cf. \eqref{eq:r}), completing the proof of Lemma~\ref{lemma:5}.}

\section{Proof of Lemma~\ref{lemma:suff_infinite}}
\label{app:lemma:suff_infinite}

\textbf{Step 1: Markov structure.}
Define the two-dimensional state process
\begin{align}
{Z_t \dff  (\ell_{t-1},\ell_t)\ , \qquad t\in\mathbb{N}\ .}
\end{align}
Since $\{\ell_t\}$ are i.i.d.\ under ${\mathbb{P}_\infty}$, the process
$\{Z_t:t\in {\mathbb{N}}\}$ is a time-homogeneous Markov chain on
{$[0,\infty)^2$} with transition
$Z_{t+1} = (\ell_t,\ell_{t+1})$, where $\ell_{t+1}$ is independent of $Z_t$ and distributed as $\ell_1$.
{Moreover, for $\xi=2$ the one-step reward is $W_t\dff \ell_t + \ell_{t-1}\ell_t$, which is a measurable function of $Z_t$.}

\paragraph{Step 2: Markov optimal stopping formulation.}
Recall the value-to-go function
\begin{align}
J_t(\F_t)
= \operatorname*{ess\,sup}_{\tau\in\mathcal T_t}
{\mathbb{E}_\infty} \left[ {W_\tau} - \lambda(\tau-t)\,\middle|\, \F_t\right]\ .
\end{align}
Because $Z_t$ is Markov and the reward {$W_t$} depends only on $Z_t$, the infinite-horizon optimal stopping problem is a Markov optimal
stopping problem (see, e.g., \cite[Ch.~III]{ShiryaevOptimal}). Hence, the value function depends on the past only through the current
Markov state $Z_t$. Therefore, there exists a Borel function $J:[0,\infty)^2\to\mathbb R$
such that
\begin{align}
J_t(\F_t) = J(\ell_{t-1},\ell_t)
\quad\mbox{a.s.}
\end{align}

\paragraph{Step 3: Structure of the continuation value.}
By definition,
\begin{align}
R_t(\F_t) = {\mathbb{E}_\infty} \big[J_{t+1}(\F_{t+1})\,\big|\, \F_t\big]\ .
\end{align}
Using the representation above,
\begin{align}
J_{t+1}(\F_{t+1})
= J(\ell_t,\ell_{t+1})\ .
\end{align}
Since $\ell_{t+1}$ is independent of $\F_t$ and has the same law as $\ell_1$, we obtain
\begin{align}
R_t(\F_t)
= {\mathbb{E}_\infty} \big[J(\ell_t,\ell_{t+1})\,\big|\, \ell_t\big]
=: R(\ell_t)\ ,
\end{align}
where
\begin{align}
{R(y) \dff  \mathbb{E}_\infty  \left[J(y,\tilde\ell)\right] \ ,}
\qquad \tilde\ell \stackrel{d}{=} \ell_1\ .
\end{align}
Thus $R_t$ depends on $\F_t$ only through $\ell_t$\ .

\paragraph{Step 4: Bellman representation.}
From the Bellman identity, we have 
\begin{align}
{J_t = \max\{W_t,\,-\lambda + R_t\}\ .}
\end{align}
Since {$W_t$} is a function of $(\ell_{t-1},\ell_t)$ and $R_t=R(\ell_t)$, it follows that
\begin{align}
J_t = J(\ell_{t-1},\ell_t)\ ,
\end{align}
for a suitable Borel function $J$. Combining the above steps proves that $(\ell_{t-1},\ell_t)$ is a
sufficient statistic for $J_t$ and $\ell_t$ is a sufficient statistic for $R_t$, establishing the desired result.

\section{Proof of Lemma~\ref{lemma:feasible}}
\label{app:lemma:feasible}

For the Shewhart test with stopping time \eqref{eq:stop}, define {$p\dff \mathbb{P}_\infty(\ell_1\ge \alpha)$.}
{Under $\mathbb{P}_\infty$, the likelihood ratios $\{\ell_t\}_{t\in\mathbb{N}}$ are i.i.d., hence $\tau_{\rm s}$ is geometric on $\{1,2,\dots\}$ with success probability $p$.}
Therefore,
\begin{align}
\label{feas1}
\mathbb{E}_\infty[\tau_{\rm s}] &= \sum_{t=1}^\infty \mathbb{P}_\infty(\tau_{\rm s}\ge t)\\
\label{feas2}
& = \sum_{t=1}^\infty (1-p)^{t-1}\\
\label{feas3}
&= \frac{1}{p}\\
\label{feas4}
&= \varepsilon\ ,
\end{align}
where~\eqref{feas2} follows because {$\mathbb{P}_\infty(\tau_{\rm s}\ge t)=(1-p)^{t-1}$},~\eqref{feas3} is the sum of a geometric series, and~\eqref{feas4} holds by the threshold choice \eqref{eq:threshold}, i.e., $p=\varepsilon^{-1}$.

\section{Proof of Lemma~\ref{lemma:mono_convex}}
\label{proof:lemma:mono_convex}

Fix $\lambda>0$ and $\xi\ge 2$. Work under $\mathbb{P}_\infty$, under which $\{\ell_t:t\in\mathbb{N} \}$ are i.i.d. and $\ell_t\ge 0$.
For each finite horizon $H\ge 1$, consider an \textcolor{black}{auxiliary finite-horizon truncation of the infinite-horizon optimal-stopping problem, used only for the backward-induction argument}, in which the decision-maker must stop no later than time $H$. Define the truncated reward-to-go $J_t^{H}$ and continuation value $R_t^{H}$ by backward recursion as follows. First, define the (time-$t$) immediate stopping reward
\begin{align}\label{eq:Stxi}
{W_t^{(\xi)}  \dff} \sum_{k=1}^{\xi}\;\prod_{{j}=0}^{k-1}\ell_{t-{j}}
= \ell_t\Big(1+\sum_{k=1}^{\xi-1}\prod_{{j}=1}^{k}\ell_{t-{j}}\Big)\ ,
\end{align}
{where we use the convention $\ell_u=0$ for $u\le 0$ (equivalently, the above products vanish whenever an index is nonpositive).}
Then set the terminal condition
\begin{align}
{J_H^{H} \dff W_H^{\xi}\ .}
\end{align}
For $t\in\{1,\dots, H-1\}$, define
\begin{align}
{R_t^{H} }&{\dff \mathbb{E}_\infty \big[J_{t+1}^{H} \,\big|\, \F_t\big]\ , \quad\mbox{and} \quad 
J_t^{H} \dff \max\big\{W_t^{(\xi)} ,\; -\lambda + R_t^{H}\big\}\ .}
\end{align}

\paragraph{Step 1: {$W_t^{(\xi)} $} is affine in $\ell_t$ for fixed past.}
Fix a time $t$ and fix the past likelihood ratios {$(\ell_{t-\xi+1},\ldots,\ell_{t-1})$}.
From \eqref{eq:Stxi}, {$W_t^{(\xi)} $} can be written as
\begin{align}
{W_t^{(\xi)} =\ell_t\,c_t\ ,\qquad 
c_t\dff 1+\sum_{k=1}^{\xi-1}\prod_{j=1}^{k}\ell_{t-j}\ge 1\ ,}
\end{align}
which is affine in $\ell_t$ (hence convex) and nondecreasing in $\ell_t$ {for fixed $c_t$}.

\paragraph{Step 2: Backward induction for convexity and monotonicity.}
We prove by backward induction on $t$ that, for each fixed {$(\ell_{t-\xi+1},\ldots,\ell_{t-1})$},
both mappings
\begin{align}
\ell_t \mapsto J_t^{H} \ , \quad\text{and}\quad \ell_t \mapsto R_t^{H}\ ,
\end{align}
are convex and nondecreasing.

\begin{itemize}
    \item At $t=H$, we have $J_H^{H}={W_H^{\xi}}$, which is affine (hence convex) and nondecreasing in $\ell_H$ by Step 1.
    \item Now assume the claim holds at time $t+1$ (with $t\le H-1$).
    \item Since $R_t^{H}=\mathbb{E}_\infty[J_{t+1}^{H}\mid\F_t]$ and $\ell_{t+1}$ is independent of $\F_t$ under $\mathbb{P}_\infty$, we can write
    \begin{align}
    {R_t^{H}(\ell_t;\ell_{t-\xi+1}^{t-1})
    = \mathbb{E}_\infty \Big[J_{t+1}^{H}\big(\ell_{t+1};\ell_{t-\xi+2}^{t}\big)\,\Big|\,\ell_{t-\xi+1}^{t}\Big]\ ,}
    \end{align}
    {where the dependence on $\ell_t$ enters only through the shifted history $\ell_{t-\xi+2}^{t}=(\ell_{t-\xi+2},\ldots,\ell_t)$.}
    {By the induction hypothesis, for each fixed realization of $\ell_{t+1}$ and the earlier ratios, the map $\ell_t\mapsto J_{t+1}^{H}(\ell_{t+1};\ell_{t-\xi+2}^{t})$ is convex and nondecreasing.}
    Therefore, $R_t^{H}$, being an average of convex and nondecreasing functions of $\ell_t$, is itself convex and nondecreasing in $\ell_t$ when the earlier likelihood ratios are held fixed.
    Finally,
    \begin{align}
    J_t^{H}=\max\{{W_t^{(\xi)} },-\lambda+R_t^{H}\}
    \end{align}
    is the pointwise maximum of two convex and nondecreasing functions of $\ell_t$, hence it is itself convex and nondecreasing.
\end{itemize}

\medskip
\noindent\textbf{Step 3: Infinite horizon.}
Let {$J_t^{\xi}$} and {$R_t^{\xi}$} denote the corresponding infinite-horizon reward-to-go and continuation value.
As $H$ increases, \textcolor{black}{these auxiliary finite-horizon truncations} admit larger classes of stopping times, so $J_t^{H}$ increases pointwise to {$J_t^{\xi}$}
(and consequently $R_t^{H}$ increases pointwise to {$R_t^{\xi}$}).
A pointwise supremum of convex (respectively, nondecreasing) functions is convex (respectively, nondecreasing),
so {$J_t^{\xi}$} and {$R_t^{\xi}$} inherit these properties in $\ell_t$ when the past likelihood ratios are held fixed. This completes the proof.

\section{Proof of Lemma~\ref{lemma:lipschitz_R}}
\label{proof:lemma:lipschitz_R}

Recall the definition of $R$ in~\eqref{eq:bellman_xi2_comp}
\begin{align}
R(\ell)={\mathbb{E}_\infty} \left[\max\{L(1+\ell),\,-\lambda+R(L)\}\right]\ ,
\qquad \mbox{where} \quad L \overset{d}{=} \ell_1 \text{ under } {\mathbb{P}_\infty}\ .
\end{align}
Fix $\ell_2\ge \ell_1\ge 0$ and for each realization $L\ge 0$ define 
\begin{align}
\phi_L(\ell)\triangleq \max\{L(1+\ell),\,-\lambda+R(L)\}\ .
\end{align}
Since $L(1+\ell)$ is nondecreasing in $\ell$ and the max of two functions preserves monotonicity,
$\phi_L(\ell)$ is nondecreasing in $\ell$, hence
\begin{align}
R(\ell_2)-R(\ell_1)={\mathbb{E}_\infty}[\phi_L(\ell_2)-\phi_L(\ell_1)]\ge 0\ ,
\end{align}
so $R(\cdot)$ is nondecreasing. For the Lipschitz bound, use the elementary inequality (valid for all real $a\ge a'$ and any $b$):
\begin{align}
0 \le \max\{a,b\}-\max\{a',b\}\le a-a'\ .
\end{align}
Applying this with $a=L(1+\ell_2)$, $a'=L(1+\ell_1)$ and $b=-\lambda+R(L)$ yields
\begin{align}
0 \le \phi_L(\ell_2)-\phi_L(\ell_1)\le L(\ell_2-\ell_1)\ .
\end{align}
Taking the expectation, we obtain
\begin{align}
0 \le R(\ell_2)-R(\ell_1)
\le (\ell_2-\ell_1)\,{\mathbb{E}_\infty}[L]\ .
\end{align}
Finally, ${\mathbb{E}_\infty}[L]={\mathbb{E}_\infty}[\ell_1]=1$ since $\ell_1=f_1(X_1)/f_0(X_1)$ is a likelihood ratio under ${\mathbb{P}_\infty}$.
Therefore,
\begin{align}
0 \le R(\ell_2)-R(\ell_1)\le \ell_2-\ell_1\ ,
\end{align}
i.e., $R$ is nondecreasing and $1$-Lipschitz on $[0,\infty)$.

\section{Proof of Theorem~\ref{lemma:upper_bound}}
\label{app:lemma:upper_bound}

Consider the equality-constrained problem
\begin{align}
\sup_{\tau\in{\mathcal{T}}} \ \frac{{\mathbb{E}_\infty[\ell_\tau]}}{{\mathbb{E}_\infty[\tau]}}\ ,
\quad \text{s.t.} \quad {\mathbb{E}_\infty[\tau]} = \varepsilon \ .
\end{align}
Since {$\mathbb{E}_\infty[\tau]=\varepsilon$} is fixed, this is equivalent to
\begin{align}
\sup_{\tau \in {\mathcal{T}}:\ {\mathbb{E}_\infty[\tau]}=\varepsilon} \ {\mathbb{E}_\infty[\ell_\tau]}\ .
\end{align}
For $\lambda>0$ define the unconstrained Lagrangian value
\begin{align}
{\mathcal{V}_1(\lambda)}
\triangleq
\sup_{\tau\in{\mathcal{T}}:\ \mathbb{P}_\infty(\tau<\infty)=1}
{\mathbb{E}_\infty[\ell_\tau - \lambda \tau]}\ .
\end{align}
{It is noteworthy that we use $\mathcal{V}_1$ to avoid confusion with the continuation function $R(\cdot)$ used elsewhere.}
For $t\ge1$, define the value-to-go function
\begin{align}
{J_t(\F_t)}
\triangleq
\operatorname*{ess\,sup}_{{\tau\in\mathcal{T}_t}}
{\mathbb{E}_\infty} \left[
\ell_\tau - \lambda(\tau-t)
\,\middle|\, \F_t
\right]\ ,
\end{align}
where {$\mathcal{T}_t\dff\{\tau\in\mathcal{T}:\tau\ge t\ \text{a.s.}\}$}.
Standard optimal stopping arguments imply the Bellman identity
\begin{align}
{J_t(\F_t)} = \max\Big\{\ell_t,\, -\lambda + {\mathbb{E}_\infty[J_{t+1}(\F_{t+1})\mid\F_t]} \Big\}\ .
\end{align}
Under $\mathbb{P}_\infty$, $\{\ell_t:t\in\mathbb{N}\}$ are i.i.d., and future samples are independent of $\F_t$. Hence, by stationarity,
\begin{align}
{\mathbb{E}_\infty[J_{t+1}(\F_{t+1})\mid\F_t]}
= {\mathbb{E}_\infty[J_{t+1}(\F_{t+1})]}
= {\mathbb{E}_\infty[J_{1}(\F_{1})]}\dff c_\lambda\ .
\end{align}
Therefore,
\begin{align}
J_t(\F_t) = \max\{\ell_t,\, -\lambda + c_\lambda\}\ .
\end{align}
Consequently, an optimal stopping time for {$\mathcal{V}_1(\lambda)$} is
\begin{align}
{\tau_\lambda
=
\inf\{t\ge1:\ell_t \ge c_\lambda-\lambda \}\ ,}
\end{align}
i.e., a Shewhart rule with threshold {$\alpha_\lambda\dff c_\lambda-\lambda$}.
Taking expectations of the Bellman fixed-point equation yields
\begin{align}
c_\lambda
=
{\mathbb{E}_\infty}[\max\{\ell_1,c_\lambda-\lambda\}]\ ,
\end{align}
equivalently,
\begin{align}
\lambda
= {\mathbb{E}_\infty}\big[(\ell_1-\alpha_\lambda)^+\big]\ .
\end{align}
Finally, we enforce the condition {$\mathbb{E}_\infty[\tau]=\varepsilon$}.
Choose {$\alpha$} such that
\begin{align}
{\mathbb{P}_\infty(\ell_1 \ge \alpha)
=
\varepsilon^{-1}\ .}
\end{align}
For the Shewhart stopping time {$\tau_{\rm s}\dff \inf\{t\ge1:\ell_t\ge \alpha\}$},
\begin{align}
{\mathbb{E}_\infty[\tau_{\rm s}]
=
\frac{1}{\mathbb{P}_\infty(\ell_1 \ge \alpha)}
=
\varepsilon\ .}
\end{align}
Now set
\begin{align}
{\lambda \dff \mathbb{E}_\infty[\max\{\ell_1,\alpha\}] - \alpha\ .}
\end{align}
Then {$c_\lambda=\mathbb{E}_\infty[\max\{\ell_1,\alpha\}]$} and {$\alpha_\lambda=c_\lambda-\lambda=\alpha$}, so {$\tau_{\rm s}$} attains the supremum in {$\mathcal{V}_1(\lambda)$} and satisfies the ARL constraint.
Hence, for any $\tau$ with {$\mathbb{E}_\infty[\tau]=\varepsilon$},
\begin{align}
{\mathbb{E}_\infty[\ell_\tau]}
&=
{\mathbb{E}_\infty[\ell_\tau - \lambda \tau]} + \lambda \varepsilon \\
&\le
{\mathbb{E}_\infty[\ell_{\tau_{\rm s}} - \lambda \tau_{\rm s}]} + \lambda \varepsilon\\
& =
{\mathbb{E}_\infty[\ell_{\tau_{\rm s}}]\ ,}
\end{align}
which proves that the Shewhart test {$\tau_{\rm s}$} solves \eqref{eq:p3}.

\section{Proof of Theorem~\ref{th:Sh}}
\label{app:th:Sh}

{Let $\tau_{\rm s} \dff \inf\{t\ge1:\ell_t \ge \alpha\}$,}
where $\alpha$ satisfies
\begin{align}
{\mathbb{P}_\infty(\ell_1 \ge \alpha) = \varepsilon^{-1}\ .}
\end{align}
For $\xi=1$, success at $\gamma_i$ means {$\tau=\gamma_i$}. Under the Shewhart rule,
\begin{align}
{\{\tau_{\rm s}=\gamma_i\}
=
\{\tau_{\rm s} \ge \gamma_i\}
\cap
\{\ell_{\gamma_i} \ge \alpha\}\ .}
\end{align}
{Under any admissible configuration, conditional on $\F_{\gamma_i-1}$ and $\{\tau_{\rm s}\ge\gamma_i\}$, the random variable $\ell_{\gamma_i}$ depends only on $X_{\gamma_i}\sim f_1$ and is independent of the past.}
Hence
\begin{align}
{\mathbb{P}_\theta(\tau_{\rm s}=\gamma_i
\mid
\F_{\gamma_i-1},\tau_{\rm s}\ge\gamma_i)
=
\mathbb{P}_1(\ell_1\ge\alpha)\ ,}
\end{align}
{a constant (not depending on $\gamma_i$ or $\F_{\gamma_i-1}$).}
Therefore
\begin{align}
{\mathcal{L}_{\rm L}(\tau_{\rm s})
=
\mathcal{L}_{\rm P}(\tau_{\rm s})
=
s\,\mathbb{P}_1(\ell_1\ge\alpha)\ .}
\end{align}
Using the {same change-of-measure argument as in Appendix~\ref{app:average_identity} (specialized to $\xi=1$)}, for any stopping time $\tau$ with $0<\mathbb{E}_\infty[\tau]<\infty$,
\begin{align}
{\mathcal{L}_{\rm L}(\tau)
\le
\mathcal{L}_{\rm P}(\tau)
\le
s \frac{\mathbb{E}_\infty[\ell_\tau]}{\mathbb{E}_\infty[\tau]}\ .}
\end{align}
From Theorem~\ref{lemma:upper_bound}, {the Shewhart rule solves the ratio problem with ARL equality}, and in particular,
\begin{align}
{\sup_{\tau:\ \mathbb{E}_\infty[\tau]=\varepsilon}
\frac{\mathbb{E}_\infty[\ell_\tau]}{\mathbb{E}_\infty[\tau]}
=
\frac{\mathbb{E}_\infty[\ell_{\tau_{\rm s}}]}{\varepsilon}\ .}
\end{align}
Next, note that under $\mathbb{P}_\infty$, $\tau_{\rm s}$ is geometric with success probability $\mathbb{P}_\infty(\ell_1\ge \alpha)=\varepsilon^{-1}$, hence {the likelihood ratio at the stopping time has the stationary overshoot identity}
\begin{align}
{\mathbb{E}_\infty[\ell_{\tau_{\rm s}}]
=
\mathbb{E}_\infty[\ell_1 \mid \ell_1\ge\alpha]
=
\frac{\mathbb{E}_\infty[\ell_1 \mathbf 1\{\ell_1\ge\alpha\}]}
{\mathbb{P}_\infty(\ell_1\ge\alpha)}\ .}
\end{align}
Using the likelihood-ratio change of measure,
\begin{align}
{\mathbb{E}_\infty[\ell_1 \mathbf 1\{\ell_1\ge\alpha\}]
=
\mathbb{P}_1(\ell_1\ge\alpha)\ ,}
\end{align}
and therefore
\begin{align}
{\mathbb{E}_\infty[\ell_{\tau_{\rm s}}]
=
\frac{\mathbb{P}_1(\ell_1\ge\alpha)}
{\mathbb{P}_\infty(\ell_1\ge\alpha)}
=
\varepsilon\,\mathbb{P}_1(\ell_1\ge\alpha)\ .}
\end{align}
Consequently,
\begin{align}
{\sup_{\tau:\ \mathbb{E}_\infty[\tau]\ge\varepsilon}
\mathcal{L}_{\rm P}(\tau)
=
s\,\mathbb{P}_1(\ell_1\ge\alpha)
=
\mathcal{L}_{\rm P}(\tau_{\rm s})\ ,}
\end{align}
{where the supremum over $\mathbb{E}_\infty[\tau]\ge\varepsilon$ reduces to $\mathbb{E}_\infty[\tau]=\varepsilon$ by Lemma~\ref{lemma:eq}.}
The same argument applies to $\mathcal{L}_{\rm L}$, yielding
\begin{align}
{\mathcal{L}_{\rm L}(\tau_{\rm s})
=
\mathcal{L}_{\rm P}(\tau_{\rm s})
=
\sup_{\tau:\ \mathbb{E}_\infty[\tau]\ge\varepsilon}
\mathcal{L}_{\rm L}(\tau)
=
\sup_{\tau:\ \mathbb{E}_\infty[\tau]\ge\varepsilon}
\mathcal{L}_{\rm P}(\tau)\ .}
\end{align}

{\color{black}
\section{Proof of Lemma~\ref{lemma:constant_tsr_arl}}
\label{proof:lemma:constant_tsr_arl}

Write $\tau=\tau^{\rm c}_{A,\xi}$. Define the classical Shiryaev--Roberts statistic by
\begin{align}
S_t^{\rm SR}
&\triangleq
\sum_{j=1}^{t}
\prod_{k=j}^{t}\ell_k\; ,
\qquad
S_0^{\rm SR}
\triangleq
0\; .
\label{eq:full_sr_statistic_appendix}
\end{align}
Since the TSR statistic retains only the terms associated with the most recent $\xi$ candidate onsets,
\begin{align}
0
\leq
W_t^{(\xi)}
&\leq
S_t^{\rm SR}
\qquad
\text{for every }t\; .
\label{eq:truncated_below_full_sr_appendix}
\end{align}
Under $\mathbb P_\infty$, the likelihood ratios are independent, nonnegative, and satisfy $\mathbb E_\infty[\ell_t]=1$. Hence,
\begin{align}
\mathbb E_\infty
\left[
S_t^{\rm SR}
\right]
&=
\sum_{j=1}^{t}
\mathbb E_\infty
\left[
\prod_{k=j}^{t}\ell_k
\right]
=
t
<
\infty\; .
\label{eq:sr_integrability}
\end{align}
Moreover, the recursion $S_t^{\rm SR}=(1+S_{t-1}^{\rm SR})\ell_t$ gives
\begin{align}
\mathbb E_\infty
\left[
S_t^{\rm SR}
\,\middle|\,
\mathcal F_{t-1}
\right]
&=
1+S_{t-1}^{\rm SR}\; .
\label{eq:sr_conditional_drift}
\end{align}
Therefore, $\{S_t^{\rm SR}-t\}_{t\geq0}$ is an integrable $\{\mathcal F_t\}$-martingale.

For $n\geq1$, let
\begin{align}
\tau_n
&\triangleq
\tau\wedge n\; .
\label{eq:truncated_constant_tsr_time}
\end{align}
Since $\tau_n$ is bounded, optional sampling yields
\begin{align}
\mathbb E_\infty
\left[
S_{\tau_n}^{\rm SR}
\right]
&=
\mathbb E_\infty[\tau_n]\; .
\label{eq:sr_optional_sampling}
\end{align}
On $\{\tau\leq n\}$, we have $\tau_n=\tau$ and
\begin{align}
S_{\tau_n}^{\rm SR}
&=
S_\tau^{\rm SR}
\geq
W_\tau^{(\xi)}
\geq
A\; .
\label{eq:sr_above_boundary_on_crossing}
\end{align}
Since $S_{\tau_n}^{\rm SR}\geq0$ everywhere,
\begin{align}
S_{\tau_n}^{\rm SR}
&\geq
A\,
\mathbf 1_{\{\tau\leq n\}}\; .
\label{eq:sr_indicator_lower_bound}
\end{align}
Combining~\eqref{eq:sr_optional_sampling} and~\eqref{eq:sr_indicator_lower_bound} gives
\begin{align}
\mathbb E_\infty[\tau_n]
&\geq
A\,
\mathbb P_\infty(\tau\leq n)\; .
\label{eq:constant_tsr_stopped_bound}
\end{align}
Because $\tau_n\uparrow\tau$, monotone convergence and continuity from below imply
\begin{align}
\mathbb E_\infty[\tau]
&\geq
A\,
\mathbb P_\infty(\tau<\infty)\; .
\label{eq:constant_tsr_limit_bound}
\end{align}
If $\mathbb P_\infty(\tau<\infty)=1$, this is the desired inequality. If $\mathbb P_\infty(\tau=\infty)>0$, then $\mathbb E_\infty[\tau]=\infty$, and the conclusion again follows.
}

{\color{black}
\section{Proof of Lemma~\ref{lemma:constant_tsr_detection}}
\label{proof:lemma:constant_tsr_detection}

Write $\tau=\tau^{\rm c}_{A,\xi}$, fix $\theta\in\Theta_s(\xi)$, and fix an episode onset $\gamma_i$. Define
\begin{align}
\Lambda_{i,k}
&\triangleq
\prod_{r=0}^{k-1}\ell_{\gamma_i+r}\; ,
\qquad
k\in\{1,\ldots,\xi\}\; ,
\label{eq:true_onset_partial_product}
\\
E_i
&\triangleq
\{\tau\geq\gamma_i\}\; ,
\qquad
D_i
\triangleq
\{\gamma_i\leq\tau<\gamma_i+\xi\}\; ,
\label{eq:survival_and_success_events}
\\
B_i
&\triangleq
\left\{
\max_{1\leq k\leq\xi}\Lambda_{i,k}
\geq
A
\right\}\; .
\label{eq:true_onset_crossing_event}
\end{align}
At time $\gamma_i+k-1$, the quantity $\Lambda_{i,k}$ is one of the nonnegative summands in $W_{\gamma_i+k-1}^{(\xi)}$. Thus, if the detector has survived to $\gamma_i$ and $B_i$ occurs, it stops within the admissible window:
\begin{align}
E_i\cap B_i
&\subseteq
D_i\; .
\label{eq:crossing_implies_window_success}
\end{align}
Because $\tau$ is a stopping time, $E_i\in\mathcal F_{\gamma_i-1}$. Under every $\theta\in\Theta_s(\xi)$, the observations $X_{\gamma_i},\ldots,X_{\gamma_i+\xi-1}$ are independent of $\mathcal F_{\gamma_i-1}$ and are i.i.d. according to $F_1$. Therefore,
\begin{align}
\mathbb P_\theta
\left(
B_i
\,\middle|\,
\mathcal F_{\gamma_i-1}
\right)
&=
\bar p_\xi(A)
\qquad
\mathbb P_\theta\text{-a.s.}\; .
\label{eq:true_onset_crossing_probability}
\end{align}
Conditioning the event inclusion~\eqref{eq:crossing_implies_window_success} on $\mathcal F_{\gamma_i-1}$ gives
\begin{align}
\mathbb E_\theta
\left[
\mathbf 1_{D_i}
\,\middle|\,
\mathcal F_{\gamma_i-1}
\right]
&\geq
\mathbb E_\theta
\left[
\mathbf 1_{E_i}\mathbf 1_{B_i}
\,\middle|\,
\mathcal F_{\gamma_i-1}
\right]
\nonumber\\
&=
\mathbf 1_{E_i}
\mathbb P_\theta
\left(
B_i
\,\middle|\,
\mathcal F_{\gamma_i-1}
\right)
\nonumber\\
&=
\bar p_\xi(A)
\mathbf 1_{E_i}\; .
\label{eq:conditional_success_indicator_bound}
\end{align}
This proves~\eqref{eq:conditional_window_success_bound}.

Suppose $\mathbb P_\theta(E_i)>0$. Taking expectations in~\eqref{eq:conditional_success_indicator_bound} and using $D_i\subseteq E_i$ yields the Pollak-type bound
\begin{align}
\mathbb P_\theta(D_i\mid E_i)
&\geq
\bar p_\xi(A)\; .
\label{eq:pollak_episode_lower_bound}
\end{align}
The same inequality holds history-wise on $E_i$. Hence, taking the essential infimum over the surviving pre-onset histories yields the corresponding Lorden-type bound
\begin{align}
\operatorname*{ess\,inf}_{\mathcal F_{\gamma_i-1};E_i}
\mathbb P_\theta
\left(
D_i
\,\middle|\,
\mathcal F_{\gamma_i-1},E_i
\right)
&\geq
\bar p_\xi(A)\; .
\label{eq:lorden_episode_lower_bound}
\end{align}
Here, $\operatorname*{ess\,inf}_{\mathcal F_{\gamma_i-1};E_i}$ denotes the essential infimum over the set $E_i$.

It remains to verify positive survival under~\eqref{eq:survivability_condition}. Since $F_1\ll F_0$,
\begin{align}
\mathbb P_\infty(\ell_1\leq b)
&>
0\; .
\label{eq:nominal_survivability_probability}
\end{align}
Define
\begin{align}
H_i
&\triangleq
\bigcap_{t=1}^{\gamma_i-1}
\{\ell_t\leq b\}\; ,
\label{eq:low_evidence_history}
\end{align}
with $H_i=\Omega$ when $\gamma_i=1$. Every observation before $\gamma_i$ has distribution either $F_0$ or $F_1$. By independence and the positivity of $F_0(\ell_1\leq b)$ and $F_1(\ell_1\leq b)$,
\begin{align}
\mathbb P_\theta(H_i)
&>
0\; .
\label{eq:low_evidence_history_positive}
\end{align}
On $H_i$, for every $t<\gamma_i$,
\begin{align}
W_t^{(\xi)}
&\leq
\sum_{k=1}^{\min\{\xi,t\}}b^k
<
\frac{b}{1-b}\; .
\label{eq:low_evidence_tsr_bound}
\end{align}
If~\eqref{eq:survivability_boundary_condition} holds, then $H_i\subseteq E_i$, and therefore
\begin{align}
\mathbb P_\theta(E_i)
&\geq
\mathbb P_\theta(H_i)
>
0\; .
\label{eq:positive_episode_survival}
\end{align}
The episode-wise bounds~\eqref{eq:pollak_episode_lower_bound}--\eqref{eq:lorden_episode_lower_bound} therefore hold for every $\theta\in\Theta_s(\xi)$ and every $i\in\{1,\ldots,s\}$. Summing over $i$ and taking the infimum over $\Theta_s(\xi)$ proves~\eqref{eq:aggregate_window_success_bound}.
}
{\color{black}
\section{Proof of Theorem~\ref{thm:constant_tsr_asymptotic_optimality}}
\label{proof:thm:constant_tsr_asymptotic_optimality}

Define, once for the remainder of the proof,
\begin{align}
\bar p_\varepsilon
&\triangleq
\bar p_{\xi_\varepsilon}(\varepsilon)\; .
\label{eq:asymptotic_crossing_notation}
\end{align}

\paragraph*{Step 1: The finite-window crossing probability tends to one.}
Condition~\eqref{eq:constant_tsr_growth_condition} implies that there exists $a<I$ such that, for all sufficiently large $\varepsilon$,
\begin{align}
\frac{\log\varepsilon}{\xi_\varepsilon}
&\leq
 a
<
I\; .
\label{eq:asymptotic_margin_constant}
\end{align}
Since the maximum defining $\bar p_\varepsilon$ includes $S_{\xi_\varepsilon}$,
\begin{align}
\bar p_\varepsilon
&\geq
\mathbb P_1
\left(
S_{\xi_\varepsilon}
\geq
\log\varepsilon
\right)
\geq
\mathbb P_1
\left(
\frac{S_{\xi_\varepsilon}}{\xi_\varepsilon}
\geq
 a
\right)\; .
\label{eq:crossing_probability_slln_bound}
\end{align}
The integrability condition~\eqref{eq:asymptotic_tsr_integrability} and the strong law of large numbers give
\begin{align}
\frac{S_n}{n}
&\longrightarrow
I
\qquad
\mathbb P_1\text{-a.s.}\; .
\label{eq:postchange_slln}
\end{align}
Since $\xi_\varepsilon\to\infty$ and $a<I$, equations~\eqref{eq:crossing_probability_slln_bound}--\eqref{eq:postchange_slln} imply
\begin{align}
\bar p_\varepsilon
&\longrightarrow
1\; .
\label{eq:pebar_to_one}
\end{align}

\paragraph*{Step 2: The stopping time is proper and feasible.}
Equation~\eqref{eq:crossing_probability_slln_bound} implies that the terminal product exceeds $\varepsilon$ with positive probability under $F_1^{\otimes\xi_\varepsilon}$ for all sufficiently large $\varepsilon$. Since $F_1^{\otimes\xi_\varepsilon}\ll F_0^{\otimes\xi_\varepsilon}$,
\begin{align}
q_\varepsilon
&\triangleq
\mathbb P_\infty
\left(
\prod_{r=1}^{\xi_\varepsilon}\ell_r
\geq
\varepsilon
\right)
>
0\; .
\label{eq:nominal_block_crossing_probability}
\end{align}
For $m\geq1$, define the crossing event on the $m$th disjoint block by
\begin{align}
C_{\varepsilon,m}
&\triangleq
\left\{
\prod_{r=(m-1)\xi_\varepsilon+1}^{m\xi_\varepsilon}
\ell_r
\geq
\varepsilon
\right\}\; .
\label{eq:disjoint_block_crossing_event}
\end{align}
Under $\mathbb P_\infty$, these events are independent and have probability $q_\varepsilon$. Let
\begin{align}
G_\varepsilon
&\triangleq
\inf\{m\geq1:C_{\varepsilon,m}\text{ occurs}\}\; .
\label{eq:first_crossing_block}
\end{align}
Then $G_\varepsilon$ is geometric with parameter $q_\varepsilon$. If $C_{\varepsilon,m}$ occurs, the complete product over that block is a summand of $W_{m\xi_\varepsilon}^{(\xi_\varepsilon)}$, and hence
\begin{align}
\tau_\varepsilon^{\rm c}
&\leq
\xi_\varepsilon G_\varepsilon\; .
\label{eq:constant_tsr_geometric_domination}
\end{align}
Consequently,
\begin{align}
\mathbb P_\infty
\left(
\tau_\varepsilon^{\rm c}<\infty
\right)
&=
1\; ,
\qquad
\mathbb E_\infty
\left[
\tau_\varepsilon^{\rm c}
\right]
\leq
\frac{\xi_\varepsilon}{q_\varepsilon}
<
\infty\; .
\label{eq:constant_tsr_proper_upper_bound}
\end{align}
Lemma~\ref{lemma:constant_tsr_arl} supplies the matching feasibility inequality, so
\begin{align}
\varepsilon
&\leq
\mathbb E_\infty
\left[
\tau_\varepsilon^{\rm c}
\right]
<
\infty\; .
\label{eq:arl_bounds_theorem6_proof}
\end{align}
This proves~\eqref{eq:constant_tsr_proper_feasible}.

\paragraph*{Step 3: Minimax convergence and optimality.}
Since $\varepsilon\to\infty$, condition~\eqref{eq:survivability_boundary_condition} holds with $A=\varepsilon$ for all sufficiently large $\varepsilon$. Lemma~\ref{lemma:constant_tsr_detection} therefore gives
\begin{align}
\mathcal L_Q
\left(
\tau_\varepsilon^{\rm c}
\right)
&\geq
s\,\bar p_\varepsilon\; ,
\qquad
Q\in\{{\rm L},{\rm P}\}\; .
\label{eq:minimax_lower_theorem6}
\end{align}
By feasibility and~\eqref{eq:trivial_minimax_upper_bound},
\begin{align}
s\,\bar p_\varepsilon
&\leq
\mathcal L_Q
\left(
\tau_\varepsilon^{\rm c}
\right)
\leq
\mathsf V_Q^\star
\left(
\varepsilon,\xi_\varepsilon;s
\right)
\leq
s\; .
\label{eq:squeeze_theorem6}
\end{align}
Together with~\eqref{eq:pebar_to_one}, this proves~\eqref{eq:constant_tsr_score_to_s}. It also gives the additive bound
\begin{align}
0
&\leq
\mathsf V_Q^\star
\left(
\varepsilon,\xi_\varepsilon;s
\right)
-
\mathcal L_Q
\left(
\tau_\varepsilon^{\rm c}
\right)
\leq
s(1-\bar p_\varepsilon)
\longrightarrow
0\; ,
\label{eq:additive_gap_theorem6_proof}
\end{align}
which proves~\eqref{eq:constant_tsr_additive_optimality}. Finally, $\bar p_\varepsilon>0$ for all sufficiently large $\varepsilon$, and~\eqref{eq:squeeze_theorem6} implies
\begin{align}
1
&\geq
\frac{
\mathcal L_Q
\left(
\tau_\varepsilon^{\rm c}
\right)
}{
\mathsf V_Q^\star
\left(
\varepsilon,\xi_\varepsilon;s
\right)
}
\geq
\bar p_\varepsilon
\longrightarrow
1\; .
\label{eq:relative_gap_theorem6_proof}
\end{align}
This proves~\eqref{eq:constant_tsr_relative_optimality} and completes the proof.
}

{\color{black}
\section{Proof of Corollary~\ref{cor:window_above_information_scale}}
\label{proof:cor:window_above_information_scale}

For the window sequence in~\eqref{eq:explicit_information_scaling},
\begin{align}
\limsup_{\varepsilon\to\infty}
\frac{\log\varepsilon}{\xi_\varepsilon}
&\leq
\frac{I}{1+\delta}
<
I\; .
\label{eq:corollary_growth_verification}
\end{align}
Therefore, the conclusions for $\tau^{\rm c}_{\varepsilon,\xi_\varepsilon}$ follow directly from Theorem~\ref{thm:constant_tsr_asymptotic_optimality}.

Now consider finite thresholds $\widehat A_\varepsilon$ satisfying~\eqref{eq:exactly_calibrated_constant_boundary}. Lemma~\ref{lemma:constant_tsr_arl} gives
\begin{align}
\widehat A_\varepsilon
&\leq
\varepsilon\; .
\label{eq:calibrated_boundary_upper_bound}
\end{align}
We next show that $\widehat A_\varepsilon\to\infty$. Suppose, to the contrary, that along a subsequence $\widehat A_\varepsilon\leq M$ for some finite $M>0$. Since $S_n/n\to I>0$ under $\mathbb P_1$, there exists a finite $n$ such that
\begin{align}
\mathbb P_1
\left(
\prod_{r=1}^{n}\ell(Z_r)
\geq
M
\right)
&>
0\; .
\label{eq:fixed_block_postchange_crossing}
\end{align}
Absolute continuity implies
\begin{align}
q_M
&\triangleq
\mathbb P_\infty
\left(
\prod_{r=1}^{n}\ell_r
\geq
M
\right)
>
0\; .
\label{eq:fixed_block_nominal_crossing}
\end{align}
For all sufficiently large $\varepsilon$, $\xi_\varepsilon\geq n$. Applying the event in~\eqref{eq:fixed_block_nominal_crossing} to successive independent $n$-blocks shows that
\begin{align}
\mathbb E_\infty
\left[
\tau^{\rm c}_{\widehat A_\varepsilon,\xi_\varepsilon}
\right]
&\leq
\frac{n}{q_M}
<
\infty\; .
\label{eq:calibrated_boundary_contradiction_bound}
\end{align}
This contradicts~\eqref{eq:exactly_calibrated_constant_boundary} as $\varepsilon\to\infty$. Hence,
\begin{align}
\widehat A_\varepsilon
&\longrightarrow
\infty\; .
\label{eq:calibrated_boundary_diverges}
\end{align}
In particular,~\eqref{eq:survivability_boundary_condition} holds with $A=\widehat A_\varepsilon$ for all sufficiently large $\varepsilon$. Moreover, by~\eqref{eq:calibrated_boundary_upper_bound} and monotonicity in the boundary,
\begin{align}
\bar p_{\xi_\varepsilon}(\widehat A_\varepsilon)
&\geq
\bar p_{\xi_\varepsilon}(\varepsilon)
\longrightarrow
1\; .
\label{eq:calibrated_crossing_probability}
\end{align}
The exact calibration gives feasibility and finite mean. Applying Lemma~\ref{lemma:constant_tsr_detection} and the same sandwich argument as in~\eqref{eq:squeeze_theorem6} proves all remaining conclusions.
}

\bibliographystyle{IEEEtran}
\bibliography{Consolidated}

\end{document}